\documentclass[11pt]{article}
\usepackage[margin=1.2in]{geometry}
\usepackage[utf8]{inputenc}
\usepackage[english]{babel}
\usepackage{subcaption}
\usepackage{authblk}
\usepackage{url}
\usepackage{algorithm} 
\usepackage{algorithmicx} 
\usepackage{algpseudocode} 
\usepackage{scalerel}
\usepackage{lipsum}
\usepackage{mathtools}
\usepackage{enumitem}
\usepackage[title]{appendix}

\usepackage{hyperref}
\hypersetup{
    colorlinks=true,
    linkcolor=blue,
    citecolor=magenta,      
    urlcolor=black
}

\usepackage{slashbox}
\usepackage{bbding}
\usepackage{colortbl}

\usepackage{tikz}
\usetikzlibrary{calc}
\usetikzlibrary{patterns}
\usepackage{mathrsfs}
\usetikzlibrary{arrows}

\usepackage[all]{xypic}
\usepackage{graphicx} 
\usepackage{amsthm}
\usepackage{amsfonts}
\usepackage{amsmath}
\usepackage{amssymb}
\usepackage{array}
\usepackage{blkarray}
\usepackage {multirow}
\usepackage{thm-restate}
\usepackage{xcolor}
\usepackage{subcaption}

\theoremstyle{plain}
\newtheorem{theorem}{Theorem}

\newtheorem{lemma}[theorem]{Lemma}
\newtheorem{corollary}[theorem]{Corollary}
\newtheorem{proposition}[theorem]{Proposition}

\newtheorem{observation}[theorem]{Observation}

\theoremstyle{definition}
\newtheorem{definition}{Definition}
\newtheorem{question}{Question}

\def\x{\mathbf{x}}
\def\y{\mathbf{y}}
\def\z{\mathbf{z}}
\def\0{\mathbf{0}}

\DeclareMathOperator{\supp}{supp}
\DeclareMathOperator{\diag}{diag}

\DeclareMathOperator{\Aut}{Aut}
\DeclareMathOperator{\symdif}{\nabla}

\DeclareMathOperator{\CP}{CP}

\title{On Singular Signed Graphs with Nullspace Spanned by a Full Vector: Signed Nut Graphs}

\author[1,2,3]{Nino Ba{\v s}i{\'c}}
\author[4]{Patrick~W.~Fowler}
\author[1,2,3,5]{Toma{\v z} Pisanski}
\author[6]{Irene Sciriha}

\affil[1]{FAMNIT, University of Primorska, Koper, Slovenia}
\affil[2]{IAM, University of Primorska, Koper, Slovenia}
\affil[3]{Institute of Mathematics, Physics and Mechanics, Ljubljana, Slovenia}
\affil[4]{Department of Chemistry, University of Sheffield, Sheffield S3 7HF, UK}
\affil[5]{Faculty of Mathematics and Physics, University of Ljubljana, Ljubljana, Slovenia}
\affil[6]{Department of Mathematics, Faculty of Science, University of Malta, Msida, Malta}

\date{\today}

\newcommand{\maltcross}{\scalerel*{%
    \tikz\fill
    (0.02,0.02)    -- (0.2,0.5)   -- (0,0.4)  -- (-0.2,0.5)  --
    (-0.02,0.02)   -- (-0.5,0.2)  -- (-0.4,0) -- (-0.5,-0.2) --
    (-0.02,-0.02)  -- (-0.2,-0.5) -- (0,-0.4) -- (0.2,-0.5)  --
    (0.02,-0.02)   -- (0.5,-0.2)  -- (0.4,0)  -- (0.5,0.2)   --
    cycle;}%
    {0}%
}

\begin{document}

\maketitle

\begin{abstract}
A signed graph has edge weights drawn from the set 
$\{+1,-1\}$, and is termed \emph{sign-balanced} if it is equivalent to an unsigned graph
under the operation of sign switching; otherwise it is called \emph{sign-unbalanced}.
A nut graph has a one dimensional kernel with a corresponding eigenvector that is full. 
In this paper we generalise the notion of nut graphs to signed graphs.
Orders for which unsigned regular nut graphs exist were determined recently for the degrees up to $11$.
By extending the definition to signed nut graphs,
we find all pairs $(\rho, n)$ for which a $\rho$-regular nut graph (sign-balanced or sign-unbalanced) of order $n$ exists with $\rho \le 11$.
We  devise a construction for signed nut graphs based on a smaller `seed' graph, giving infinite series of both sign-balanced and sign-unbalanced
$\rho$-regular nut graphs. 
All orders for which a \emph{complete sign-unbalanced} nut graph exists are characterised; they have
underlying graph $K_n$ with $n \equiv 1 \pmod 4$.
All orders for which a regular \emph{sign-unbalanced} nut graph with $\rho = n - 2$ exists are also characterised; they have 
an underlying cocktail-party graph $\CP(n)$ with even $n \geq 8$.

\vspace{\baselineskip}

\vspace{\baselineskip}
\noindent
\textbf{Math.\ Subj.\ Class.\ (2020):} 
05C50, 
15A18, 
05C22, 
05C92 
\end{abstract}

\section{Introduction and motivation}

Spectral graph theory is an important branch of discrete mathematics that links graphs to linear algebra. Its applications are numerous. 
For instance, in Chemistry 
the H\"{u}ckel molecular orbital theory of conjugated $\pi$ systems \cite{streitwieser1961} is essentially an exercise in applied spectral graph theory \cite{trinajstic1992}.
Unlike the adjacency matrix
itself, the spectrum of the adjacency 
matrix is an invariant.  Singular graphs,
i.e.\ graphs that have
a zero eigenvalue
of the $0$-$1$ adjacency matrix, have been studied extensively.
A special class of singular graphs consists of  the core graphs, graphs of which
the null spaces of the adjacency matrix contains a full vector (i.e.\ a vector that has no zero entry).
One subclass of core graphs, known as nut graphs, is of particular interest.
A \emph{nut graph} is a singular graph whose 
$0$-$1$ adjacency matrix has a one-dimensional kernel
(nullity $\eta = 1$) and  a full corresponding eigenvector.  
In this sense, a nut graph is a maximally-extended core graph of nullity $1$ \cite{ScirihaGutman-NutExt}.
Some properties of nut graphs
are easily proved. For instance, nut graphs are connected, non-bipartite, and have no vertices of degree one  \cite{ScirihaGutman-NutExt}. 
It is conventional to require that a nut graph has $n \geq 2$ vertices, although some authors consider $K_1$ as the trivial nut.
There exist numerous construction rules for making larger nut graphs from smaller \cite{ScirihaGutman-NutExt}.

In this contribution we will consider signed graphs, i.e.\ graphs with edge weights drawn from $\{-1, 1\}$ and ask whether they can also be nut graphs. Although the 
problem may appear to be purely mathematical, 
we were driven to study it by the chemical interest arising from the specific properties expected of conjugated $\pi$ systems whose molecular graphs are nut graphs. 
For example, in electronic structure theory, systems
based on nut graphs have distributed radical reactivity, since occupation 
by a single electron
of the molecular orbital corresponding to the kernel eigenvector
leads to spin density on all carbon centres \cite{feuerpaper}. 
In the context of theories of molecular conduction, nut graphs have a unique status as the strong omni-conductors of nullity $1$ \cite{PWF_JCP_2014}.
Extension of our considerations to signed nut graphs allows
the study of twisted carbon networks. 
M\"{o}bius carbon networks, where the carbon framework is 
embedded on a surface with a half twist are representable by signed graphs. Such molecules have been synthesised \cite{Herges2006,Rzepa2005,Yoneda2014,Kim2009}, and they obey different electron counting rules from those of the standard unweighted  H\"{u}ckel networks \cite{Heilbronner1964,Fowler02}.
The distinct electronic structure and unusual steric interactions of M\"{o}bius systems suggest potential for applications in functional materials of various types \cite{Fumitoshi2018}.

The notion of signed graphs goes back to at least the 1950s and is documented in two dynamic-survey papers \cite{ZaslavDS9,ZaslavDS8}.
The fundamentals of this theory and its standard notation was developed by Thomas Zaslavsky in 1982 in \cite{Zaslav} and in a series of other papers.
For  general discussion of
signed graphs the reader is referred to recent papers \cite{BeCiKoWa2019,GhHaMaMa2020} which set out 
notation and basic properties. 
For nut graphs, several useful
papers are available, for instance \cite{Basic2020,Sc1998,Sc2007,Sc2008}. Recently the problem of existence of regular nut graphs was posed, and solved for cubic and quartic graphs \cite{GPS}. Later, it was solved for degrees $\rho$, $\rho \leq 11$ \cite{Jan}.  In this paper we carry the problem of 
the existence of regular nut graphs over from ordinary graphs to signed graphs. 
We characterise all cases ($\rho \le 11$)
for which a $\rho$-regular nut graph of order $n$ (either signed or unsigned) exists. 
This depends on a construction,
based on work on unsigned nut graphs \cite{Jan,GPS}, which
produces larger signed nut graphs from 
smaller, whether regular or not (though it is used here for regular graphs).
Cases where a sign-unbalanced nut graph exists with $\rho = n - 1$ or $\rho = n - 2$ are also fully characterised.
We finish the paper with two questions for the open cases (i.e.\ $\rho \geq 12$).

We note that our investigation is a special case of the spectral theory of finite-dimensional symmetric linear operators with nullspace spanned
by a full vector. The case where a fixed matrix $A$ with eigenvalue $\lambda$ has a part-full eigenvector has also attracted attention \cite{Maybee1990, McDonald1993}. 

\section{Signed graphs, nut graphs and signed nut graphs}

\subsection{Signed graphs}

In this section we start with formal definitions.
A \emph{signed graph} $\Gamma = (G, \Sigma)$ is a graph $G=(V, E)$ with a distinguished subset of edges $\Sigma \subseteq E$ that we shall 
call \emph{negative edges}. Equivalently, we may consider the signed graph $(G, \sigma)$ to be a graph endowed with a mapping 
$\sigma\colon E \rightarrow \{-1,+1\}$ where $\Sigma = \{e \in E \mid \sigma(e) = -1\}$. The adjacency matrix $A(\Gamma)$ of a signed graph is a symmetric 
matrix obtained from the adjacency matrix $A(G)$ of the \emph{underlying graph} $G$ by replacing $1$ by $-1$ for entries where 
the vertices are connected by a negative edge. 

Symmetries (automorphisms)  of a signed graph $\Gamma = (G,\Sigma)$ are also symmetries of the underlying graph $G$. They 
preserve edge weights:
\begin{equation}
\Aut \Gamma = \{\alpha \in \Aut G \mid \forall e = uv \in E(G): e \in \Sigma \iff \alpha(e)= \alpha(u)\alpha(v) \in \Sigma \}.
\end{equation}
Hence, the automorphism group $\Aut \Gamma$ of a signed graph $\Gamma$ is a subgroup of the automorphism group $\Aut G$ of the underlying graph $G$.

\subsection{Singular graphs, core graphs and nut graphs}

A graph that has zero as an eigenvalue is called a \emph{singular graph}, i.e.\ a graph is singular if and only if its adjacency matrix
has a non-trivial kernel. The dimension of the kernel is 
the \emph{nullity}. An eigenvector $\x$ can be viewed as a weighting of vertices, i.e.\ a mapping $\x\colon V \to \mathbb{R}$. A vector $\x$ belongs to the
kernel, $\ker A$, i.e. $\x \in \ker A$, if and only if for each vertex $v$ the sum of entries on the
neighbours $N_G(v)$ equals~$0$:
\begin{equation}
\sum_{u \in N_G(v)} \x(u) = 0.
\label{eq:locCon}
\end{equation}
The equation \eqref{eq:locCon} is called the \emph{local condition}.
The support, $\supp \x$, of a kernel eigenvector $\x \in \ker A$ is the subset of $V$ at which $\x$ attains non-zero values:
\begin{equation}
\supp \x = \{v \in V \mid \x(v) \neq 0\}.
\label{eq:suppdef}
\end{equation}
If $\supp \x = V$, then the vector $\x$ is \emph{full}.   Define $\supp \ker A$ as follows:
\begin{equation}
\supp \ker A = \bigcup_{\x \in \ker A} \supp \x.
\label{eq:suppkerdef}
\end{equation}
A singular graph $G$ is a \emph{core graph} if $\supp \ker A = V$. A core graph of nullity $1$ is called a \emph{nut graph}.
Although the kernel of a core graph may have a basis that has no full vectors, there exists a basis with all vectors being full.

\begin{proposition}
\label{prop:1}
Each core graph admits a kernel basis that contains only full vectors.
\end{proposition}

\begin{proof}
Let $V(G) = \{1, \ldots, n\}$ and
let $\mathbf{x}_1, \ldots, \mathbf{x}_\eta$ be an arbitrary kernel basis. Let $\iota$ be the smallest integer (i.e.\  vertex label),
such that at least one of the entries $\mathbf{x}_1(\iota), \ldots, \mathbf{x}_\eta(\iota)$ is zero, and let $\mathbf{x}_\ell(\iota)$
be one of those entries.
As $G$ is a core graph,
at least one of the entries $\mathbf{x}_1(\iota), \ldots, \mathbf{x}_\eta(\iota)$ is non-zero; let us denote the first
such encountered entry by $\mathbf{x}_k(\iota)$.
We can replace vector $\mathbf{x}_\ell$ by $\mathbf{x}_\ell + \alpha  \mathbf{x}_k$, where $\alpha > 0$. If we pick $\alpha$ large enough, i.e.\ 
if 
\begin{equation} 
\alpha > \max \left\{  \frac{|\mathbf{x}_\ell(i)|}{|\mathbf{x}_k(i)|} \mid i \neq \iota, \mathbf{x}_k(i)\neq 0  \right \},
\end{equation} 
then
$\mathbf{x}_\ell(i) + \alpha  \mathbf{x}_k(i)$ will be non-zero for all $i$. No new zero entries were created in the
replacement process and at least
one zero was eliminated. We repeat this process until
no more zeros remain.
\end{proof}

\begin{corollary}
Each core graph admits a kernel basis that contains only full vectors with integer entries.
\end{corollary}

\begin{proof}
An integer basis always exists for an integer eigenvalue.
The replacement
process described in the proof of Proposition~\ref{prop:1} will keep all entries
integer if we choose an integer for the value of $\alpha$ at each step.
\end{proof}

\subsection{Switching equivalence of signed graphs}

Let $\Gamma = (G,\Sigma)$ be a signed graph over $G= (V,E)$ and let $U \subseteq V(G)$ be a set of its vertices. A \emph{switching} at $U$ is an 
operation that transforms $\Gamma = (G,\Sigma)$ to a signed graph $\Gamma^U = (G, \Sigma \symdif \partial U)$ where 
$\symdif$ denotes symmetric difference of sets and
\begin{equation}
\partial U = \{uv \in E \mid u \in U, v \notin U\}. 
\end{equation}
Note that $\partial U = \partial (V(G) \setminus U)$, $\Gamma^U = \Gamma^{V(G) \setminus U}$ and $(\Gamma^U)^U = \Gamma$. In fact, switching is an equivalence relation on the signed graphs with the same 
underlying graph.  Any graph $G$ can be regarded as a signed graph $\Gamma = (G, \emptyset)$. We extend this definition. Any signed graph is a \emph{sign-balanced} graph if it is switching equivalent to $\Gamma = (G, \emptyset)$, otherwise it is called \emph{sign-unbalanced} \cite{Zaslav,ZaslavDS9}.
 In the literature of nut graphs the term `balanced' has also been used with a different meaning \cite{feuerpaper}.

As observed, for instance in \cite{BeCiKoWa2019}, switching 
has an obvious linear algebraic description.

\begin{proposition}[\cite{BeCiKoWa2019}]
\label{prop:switchProp}
Let $A(\Gamma)$ be the adjacency matrix of signed graph $\Gamma$ and $A(\Gamma^U)$ be the corresponding adjacency matrix of the signed graph 
switched at $U$. Let $S = \diag(s_1,s_2, \ldots, s_n)$ be the diagonal matrix with $s_i = -1$ if $v_i \in U$ and $s_i = 1$ elsewhere.  Then
\begin{equation}
A(\Gamma^U) = S A(\Gamma) S.
\end{equation}
Since $S^T = S^{-1} = S$ we also have:
\begin{equation}
A(\Gamma) = S A(\Gamma^U) S.
\end{equation}
\end{proposition}

Clearly, switching-equivalent signed graphs with the same underlying graph $G$ are cospectral. In particular, all sign-balanced
graphs with the same underlying graph $G$ are cospectral.

\begin{proposition}
\label{thm:scsizes}
Let $G$ be a connected graph on $n$ vertices and $m$ edges. There are $2^m$ signed graphs over $G$, there are $2^{m-n+1}$ 
switching equivalence classes, and each class has $2^{n-1}$ signed graphs.
\end{proposition}

\begin{proof}
The essential idea of using a spanning tree is present in the earlier literature, e.g.\ Lemma~3.1 in \cite{Zaslav}. Here we use it to find cardinalities
of the switching classes, which in turn is needed for the analysis
of Algorithm~\ref{alg-1}. At the beginning we choose a fixed but arbitrary
spanning tree $T$ in $G$. We divide the argument into four steps.

\vspace{0.5\baselineskip}
\noindent
{\bf Step (a)}: For a given connected graph $G$ (and the spanning tree $T$) there are $2^m$ different signed graphs. 
Indeed, we may choose any subset $\Sigma$ of edges $E$ and make all edges in $\Sigma$  negative.
Some of the signed graphs will have all edges of $T$ positive, while others will have
some edges of $T$ negative.

\vspace{0.5\baselineskip}
\noindent
{\bf Step (b)}: Among the $2^m$ signed graphs over $G$ exactly $2^{m-n+1}$  will have all edges of $T$ positive.
Indeed, while fixing $(n-1)$ edges of $T$ positive, any selection of the remaining $(m-n+1)$ non-tree
edges determines $\Sigma$. Such a selection can be done in $2^{m-n+1}$ ways.

\vspace{0.5\baselineskip}
Hence, Step (a) gives the total number of signed graphs while Step (b) gives the number of signed graphs
having all edges of $T$ positive. 

\vspace{0.5\baselineskip}
\noindent
{\bf Step (c)}: There are $2^{n-1}$ switchings available. Namely, any switching is determined by a pair $(U,V \setminus U)$, but
$(U,V \setminus U)$ is the same switching as $(V \setminus U,U)$. Hence we have to divide $2^n$, the number of subsets of $V$, by $2$. Thus,
each switching equivalence class contains $2^{n-1}$ signed graphs. 
Dividing the total number of signed graphs $2^m$ by the cardinality of each switching class $2^{n-1}$ we obtain
the number of different switching equivalence classes: $2^{m-n+1}$.

\vspace{0.5\baselineskip}
Every switching equivalence class of signed graphs over $G$ contains exactly one signed graph with all edges of the tree $T$ positive.
Recall that in any tree there is a unique path between any two vertices. Choose any vertex $w$ from $V$. 
Let $U$ be the set of vertices $v$ in $T$ that have an even number of negative edges on the unique $w-v$ path of $T$. 
Then $V \setminus U$ contains the vertices  $v$ that have an odd number of negative edges on the path from $w$ to $v$ along $T$.
The switching $(U,V \setminus U)$ will make $T$ all positive.  
Hence, each switching class has at least one signed graph that makes $T$ all positive. However, since the cardinality
under Step (c) is the same as under Step (b), namely $2^{m-n+1}$, 
we may deduce that
each switching class contains exactly one signed graph with the edges of the tree $T$ all positive.
\end{proof}

\subsection{Signed singular graphs, signed core graphs and signed nut graphs}

One may consider the kernel of the adjacency matrix of a \emph{signed} graph. 
The local condition \eqref{eq:locCon} generalises to
\begin{equation}
\sum_{u\sim v}x_u \sigma(vu) = 0
\label{eq:newlocCon}
\end{equation}
for each choice of a pivot vertex $v$, where 
$\sigma(uv)$ is the weight 
($\pm1$) of the edge between $u$ and $v$,
where $\x(u)$ has been written as $x_u$ for brevity and this convention will be used in the rest of the paper.
Note that definitions \eqref{eq:suppdef} and \eqref{eq:suppkerdef} can be extended
to a signed graph in a natural way.
A signed graph is a \emph{signed singular graph} if it has a zero as an eigenvalue.
A signed graph is a  \emph{signed core graph} if $\supp \ker A(\Gamma) = V(G)$.
A signed graph is a \emph{signed nut graph} if its adjacency matrix $A(\Gamma)$
has nullity one and $\ker A(\Gamma)$ contains a full kernel eigenvector.
It follows that a signed nut graph is a signed core graph of nullity one.

A graph $G$ on $n$ vertices and $m$ edges gives rise to $2^m$ distinct signed graphs. If we are interested in non-isomorphic signed graphs only, 
this number may be reduced by the symmetries that preserve signs. However, there is also an equivalence relation, to be described in the next section, among the signed 
graphs $\Gamma = (G, \sigma)$ over a given underlying graph that is very convenient, as it preserves several important signed invariants
and reduces the number of graphs to be considered.

\subsection{Switching equivalence and signed singular graphs}

\noindent
Proposition~\ref{prop:switchProp} has the following immediate consequence:
\begin{proposition}
Let $\Gamma$ be a signed graph and let $U \subseteq V(G)$.
If $\Gamma$ is singular and if $\x$ 
is any of its kernel vectors, then $\Gamma^U$ is singular and the vector $\x^U$ defined as
\begin{equation}
\x^U(v) = \begin{cases}
\phantom{-} \x(v), & \text{if } v \in V(G) \setminus U, \\
- \x(v), & \text{if } v \in U,
\end{cases}
\end{equation}
is a kernel eigenvector for  $\Gamma^U$.
\end{proposition}

This proposition is helpful in the study of singular graphs. Namely, it follows that many properties of signed graphs concerning singularity
hold for the whole switching equivalence class. 

\begin{corollary}
Let $\Gamma$ and $\Gamma'$ be two switching-equivalent signed graphs. The following holds:
\begin{enumerate}[label=(\arabic*)]
\item If one of the pair is singular, then the other is also singular. In addition, $\Gamma$ and $\Gamma'$ have the same nullity.
\item If one of the pair is a core graph, then the other is also a core graph. 
\item If one of the pair is a nut graph, then the other is also a nut graph.
\end{enumerate}
\end{corollary}

In particular, this reduces the search for nut graphs to 
a search over distinct switching equivalence classes. The following fact may be useful. 

\begin{corollary}
Every switching equivalence class of signed nut graphs has exactly one representative that has kernel eigenvector with all  entries positive. 
\end{corollary}

\begin{proof}
Let $\Gamma$ be a signed nut graph and $\x$ be its kernel eigenvector. Let 
\begin{equation}
U = \{v \in V \mid \x(v) < 0 \}.
\end{equation}
The switching at $U$ gives rise to the switching-equivalent signed nut graph $\Gamma^U$ 
with an all-positive kernel eigenvector.
\end{proof}

The above corollary enables us to select for any signed nut graph $\Gamma = (G, \Sigma)$ a unique switching-equivalent
graph $\Gamma' = (G, \Sigma')$ s.t.\ the kernel eigenvector $\bf{x'}$ relative to $\Gamma'$ is given by 
${\bf x'}(v) = |{\bf x}(v)|$. This canonical choice of switching can be viewed in the more general setting of signed graphs.

Using the idea of the proof of Proposition~\ref{thm:scsizes} 
and a database of regular connected graphs of a given order \cite{McKay201494},
we may search for signed nut graphs of that order.
Let $F(n,\rho)$ be the number of connected regular graphs of order $n$ and degree $\rho$. In the worst case, Algorithm~\ref{alg-1} has to
check $2^{m-n+1}$  sign structures on each graph. Since $2m = n\rho$ this implies a maximum of $F(n,\rho)2^{m-n+1}$ tests. 

\renewcommand{\algorithmicrequire}{\textbf{Input:}}
\renewcommand{\algorithmicensure}{\textbf{Output:}}
\begin{algorithm}
	\caption{Given the class of graphs $\mathcal{G}_{n,\rho}$, i.e.\ the class of connected $\rho$-regular graphs of order $n$,
find a signed nut graph in this class.}
\label{alg-1}
\begin{algorithmic}[1]
\Require{$\mathcal{G}_{n,\rho}$, the class of all connected $\rho$-regular graphs of order $n$.}
\Ensure{A signed nut graph in $\mathcal{G}_{n,\rho}$ (or report that there is none).}
\ForAll{$G \in \mathcal{G}_{n,\rho}$}
\State $T \gets$ spanning tree of $G$ 
\ForAll{$\Sigma \subseteq E(G) \setminus T$}
\State $\Gamma \gets (G, \Sigma)$
\If{$\Gamma$ is a signed nut graph}
\State \textbf{return} $\Gamma$
\EndIf
\EndFor
\EndFor
\State \textbf{report} there is no signed nut graph in class $\mathcal{G}_{n,\rho}$
\end{algorithmic} 
\end{algorithm}

\section{Results}
 
Our contribution was prompted by recent interest in the study of families of nut graphs. 
An efficient strategy for generating nut graphs
of small order was published in 2018 \cite{CoolFowlGoed-2017} and the full collection of nut graphs found there for orders
up to 20 was reported in the House of Graphs \cite{hog}. 
For arbitrary simple graphs,
the list is complete for orders up to 12, and counts are give up to 13. A list of regular nut graphs for orders from 3 to 8 was
deposited in the same place. This list covers orders up to 22 and is complete up to order 14.
More recently, the orders for which regular nut graphs of degree $\rho$ exist have been established for $\rho \in \{3,4,5,6,7,8,9,10,11\}$. 
In~\cite{GPS} the set $N(\rho)$ was defined as the set consisting of all integers $n$ for which a $\rho$-regular nut graph
of order $n$ exists. There it was shown that
\begin{equation}
\begin{aligned}
N(1) & = N(2) = \emptyset, \\
N(3) & = \{12\} \cup \{2k \mid k\geq 9\}, \\
N(4) & = \{8,10,12\} \cup \{k \mid k\geq14\}.
\end{aligned}
\end{equation}
In \cite{Jan}, $N(\rho)$ was determined  for every $\rho$, $5 \leq \rho \leq 11$. Combining these results, we obtain the following.

\begin{restatable}[{\cite[Theorems 2, 3 and 7]{Jan}}]{theorem}{mainthm}
\label{thm_orders}
The following holds:
\begin{enumerate}
\item $N(1) = \emptyset$
\item $N(2) = \emptyset$
\item $N(3) =  \{12\} \cup \{2k \mid k\geq 9\}$
\item $N(4) = \{8,10,12\} \cup \{k \mid k\geq14\}$
\item $N(5) = \{2k \mid k\geq 5\}$
\item $N(6) = \{k \mid k\geq 12\}$
\item $N(7) = \{2k \mid k\geq 6\}$
\item $N(8) = \{12\} \cup \{k \mid k\geq 14\}$
\item $N(9) = \{2k \mid k\geq 8\}$
\item $N(10) = \{k \mid k\geq 15\}$
\item $N(11) = \{2k \mid k\geq 8\}$
\end{enumerate}
\end{restatable}

Note that, for each $\rho$ in the range $3 \leq \rho \leq 11$, the set $N(d)$ misses only a finite number of integer values. The question we tackle here is: 
which of the missing numbers can be covered by regular 
signed nut graphs? The main result of this paper is summarised  
in Table~{\ref{tab:pinky}} which embodies the results from Theorem \ref{thm_orders} above and the two new Theorems~{\ref{thm:main}}
and~\ref{thm:10} below.

\newcommand{\kC}{\checkmark}
\newcommand{\bXE}{{\small $\nexists$}}
\newcommand{\bX}{{\small $\nexists$}}
\newcommand{\rC}{\maltcross}
\newcommand{\mbg}{\cellcolor{gray!25!white}}

\begin{table}[!htb]
\centering
\renewcommand*{\arraystretch}{1.3}
\setlength{\tabcolsep}{5pt}
\begin{tabular}{|c|*{16}{|c}|}
\hline
\backslashbox{$\rho$}{$n$} & 5 & 6 & 7 & 8 & 9 & 10 & 11 & 12 & 13 & 14 & 15 & 16 & 17 & 18 & 19 & $\cdots$ \\ 
\hline\hline
3 & \mbg & \bXE & \mbg & \bXE & \mbg & \bXE & \mbg & \kC \rC & \mbg & \rC & \mbg & \rC & \mbg & \kC \rC & \mbg & $\Rightarrow$ \\
\hline
4 & \rC & \bXE & \rC & \kC \rC & \rC & \kC \rC & \rC & \kC \rC & \rC & \kC \rC & \kC \rC & \kC \rC & \kC \rC & \kC \rC & \kC \rC & $\Rightarrow$ \\
\hline
5 & \mbg & \bX & \mbg & \bXE & \mbg & \kC \rC & \mbg & \kC \rC & \mbg & \kC \rC & \mbg & \kC \rC & \mbg & \kC \rC & \mbg & $\Rightarrow$ \\
\hline
6 & \mbg & \mbg & \bX & \rC & \rC & \rC & \rC & \kC \rC & \kC \rC & \kC \rC & \kC \rC & \kC \rC & \kC \rC & \kC \rC & \kC \rC & $\Rightarrow$ \\
\hline
7 & \mbg & \mbg & \mbg & \bX & \mbg & \rC & \mbg & \kC \rC & \mbg & \kC \rC & \mbg & \kC \rC & \mbg & \kC \rC & \mbg & $\Rightarrow$ \\
\hline
8 & \mbg & \mbg & \mbg & \mbg & \rC & \rC & \rC & \kC \rC & \rC & \kC \rC & \kC \rC & \kC \rC & \kC \rC & \kC \rC & \kC \rC & $\Rightarrow$ \\
\hline
9 & \mbg & \mbg & \mbg & \mbg & \mbg & \bX & \mbg & \rC & \mbg & \rC & \mbg & \kC \rC & \mbg & \kC \rC & \mbg & $\Rightarrow$ \\
\hline
10 & \mbg & \mbg & \mbg & \mbg & \mbg & \mbg & \bX & \rC & \rC & \rC & \kC \rC & \kC \rC & \kC \rC & \kC \rC & \kC \rC & $\Rightarrow$ \\
\hline
11 & \mbg & \mbg & \mbg & \mbg & \mbg & \mbg & \mbg & \bX & \mbg & \rC & \mbg & \kC \rC & \mbg & \kC \rC & \mbg & $\Rightarrow$ \\
\hline 
\end{tabular}
\caption{Existence of regular signed and unsigned
nut graphs of order $n$ and degree $\rho$. Notation:
\kC \ldots there exists a sign-balanced nut graph;
\protect\rC \ldots there exists a sign-unbalanced nut graph;
\bXE \ldots there exists no nut graph, signed or unsigned.
Shaded squares denote parameters for which no simple graph exists. Arrows $\Rightarrow$ indicate that the table continues 
to infinity to the right: regular sign-balanced and sign-unbalanced nut graphs with the given degree $\rho$
exist for all higher values of $n$ (even $\rho$),
or all higher
even values of $n$ (odd $\rho$).
Examples for all cases marked with \protect\rC\ are given in the appendices.
}
\label{tab:pinky}
\end{table}

\begin{theorem}
\label{thm:main}
Let $N_s(\rho)$ denote the set of orders $n$ for which there
exists no $\rho$-regular sign-balanced nut graph 
but there exists a $\rho$-regular sign-unbalanced nut graph.
\begin{enumerate}
\item $N_s(1) = \emptyset$ 
\item $N_s(2) = \emptyset$ 
\item $N_s(3) = \{14,16\}$ 
\item $N_s(4) = \{5,7,9,11,13\}$ 
\item $N_s(5) = \emptyset$
\item $N_s(6) = \{8,9,10,11\}$ 
\item $N_s(7) = \{10\}$ 
\item $N_s(8) = \{9,10,11,13\}$ 
\item $N_s(9) = \{12, 14\}$ 
\item $N_s(10) = \{12, 13, 14\}$ 
\item $N_s(11) = \{14\}$ 
\end{enumerate}
\end{theorem} 

\begin{proof}
Case $N_s(1) = \emptyset$ is trivial. The only possible graph is switching-equivalent to $K_2$ which is not a nut graph.
Case $N_s(2) = \emptyset$ is also straightforward. The only possible graphs are switching equivalent to the cycle $C_n$
or the M{\"o}bius cycle $M_n$ (a cycle where exactly one edge carries weight $-1$). It is well known that
$\{\eta(C_n), \eta(M_n)\} = \{0, 2\}$ (see, e.g.\ \cite{Fowler02}) and hence neither is a nut graph.
 
From Theorem~\ref{thm_orders} we have only a finite number of further cases to check (all those pairs $(n, \rho)$ that do not
correspond to \kC\ in Table~\ref{tab:pinky}) in order to establish the theorem.
For some sets of parameters $(n,\rho)$, we were immediately able
to prove existence/non-existence by search using the straightforward Algorithm~\ref{alg-1}. 
For some parameter combinations,  
exhaustive search was unfeasible, but in these cases an example was generated by heuristic approaches.
For some cases this involved  planting a small number of negative edges in a general graph. 
For others, a negative hamiltonian cycle or perfect matching was added to an unsigned nut graph.
\end{proof} 

All signed graphs mentioned in the proof can be found in Appendix~\ref{sec:appA}.
Two further theorems extend the results of Theorem~{\ref{thm:main}} 
to infinity along the leading diagonals of the table.

\begin{theorem}
\label{thm:10}
Let $\Gamma$ be a signed graph whose underlying graph is $K_n$. If $\Gamma$ is a signed nut graph, then $n \equiv 1 \pmod 4$.
Moreover, for each $n \equiv 1 \pmod 4$, there exists a signed nut graph with underlying graph $K_n$.
\end{theorem}

\begin{proof}
Let $n = 4k + q$, where $0 \leq q \leq 3$. We divide the proof into three parts: (a) $q \in \{0, 2\}$, (b) $q = 3$, and (c) $q = 1$.

Assume  that $\Gamma$ is singular. Let $\bf x$ be a full kernel eigenvector (existence established by Proposition \ref{prop:1}).  We may assume that $\bf x$ is a  non-zero integer vector.
If $\bf x$ has no odd coordinate, we may multiply $\bf x$ by an appropriate power of $\frac{1}{2}$ so that at least one coordinate becomes odd. 
We call vertex $s$ even if $x_s \equiv 0 \pmod 2$ and odd if $x_s  \equiv1 \pmod 2$.
The local condition \eqref{eq:newlocCon}
at $r$ implies that 
there is an even number of odd vertices around $r$ and hence, together with $r$, an odd number of odd vertices in total
for a presumed $K_n$ nut graph. 

Case (a): $q \in \{0, 2\}$. This means that $n$ is even.
The signed graph $\Gamma$
must have an even vertex $t$.
The local condition at $t$ implies that 
there is an even number of odd vertices around $t$, and hence an even number of odd vertices in total, a contradiction.
This rules out the existence of signed complete nut graphs for $q \in \{0, 2\}$.

Case (b): $q \in 3$. In this case all entries $x_s$ are odd:
the parity of the sum in \eqref{eq:newlocCon} is opposite for even and odd vertices.

Let $m_+$ denote the number of edges in $\Gamma$ with positive sign. For each vertex $s$, let $\rho_+(s)$ and $\rho_-(s)$, respectively, denote the number of edges with positive
sign and negative sign that are incident with $s$. Note that $\rho_+(s) + \rho_-(s) = n - 1$.
Summing local conditions over pivots $r$:
\begin{equation}
0 = \sum_r  \rho_+(r) x_r - \sum_{r}\rho_-(r) x_r
\label{eq:rhominuseq}
\end{equation}
Hence,
\begin{equation}
\begin{aligned}
0 & =  \sum_r \{  \rho_+(r) -  \rho_-(r) \} x_r = \sum_r \{ 2 \rho_+(r)  -n + 1 \} x_r  
\end{aligned}
\end{equation}
and
\begin{align}
  \sum_r  \rho_+(r) x_r & = \frac{n - 1}{2}   \sum_r x_r =  \{2k + 1  \}   \sum_r x_r. \label{eq:evenodd}
\end{align}
The RHS of  \eqref{eq:evenodd} is an odd number as it is a product of two odd numbers, hence $\sum_r  \rho_+(r) x_r $
is odd. Therefore, the subgraph with only positive edges 
has an odd number of vertices with odd degree. By the Handshaking Lemma this is impossible.

Case (c): $q = 1$ is the only remaining possibility and the first part of the theorem follows, provided such nut graphs exist.
Now we construct a signed nut graph for each $n$ of the form $n = 4k + 1$.
A signed graph $\Gamma$ is constructed from $K_{4k+1}$ as follows. Partition the vertex set of $K_{4k+1}$ into a single vertex $r = 0$ and $k$ subsets
of $4$ vertices for $k$ copies  of the path $P_4$. Change the signs of all edges internal to each $P_4$ to $-1$.
To construct a kernel eigenvector $\bf x$ of $\Gamma$ for $\lambda = 0$ place $+1$ on vertex $0$, then entries $-1,+1,+1,-1$ on each $P_4$. We denote
by $1,2,3,4$ the vertices of the first $P_4$, by $5,6,7,8$ the vertices of the second $P_4$, etc.
Owing to the symmetry of $\Gamma$ (since automorphisms of $\Gamma$ preserve edge-weights) 
there are only three vertex types to be considered. 
\begin{enumerate}
\item For $r = 0$ all weights $\sigma(rs) = 1$ and exactly half of the $x_s$ are equal to $1$ and the other half are equal to $-1$. Hence, \eqref{eq:newlocCon}
is true in this case.
\item The vertex $r$ may be an end-vertex of any of the $k$  paths $P_4$.  The net contribution of the three remaining vertices of $P_4$ is $-1$, and all contributions
of other paths cancel out. Taking into account the edge to vertex $0$, the weighted sum in \eqref{eq:newlocCon}
 is indeed equal to $0$.
\item The vertex $r$ may be an inner vertex of any of the $k$  paths $P_4$.  Again, the net contribution of the three remaining vertices of $P_4$ is $-1$, and by the
same argument, the weighted sum in \eqref{eq:newlocCon}
is again equal to $0$.
\end{enumerate}
\par\noindent
As $\bf x$ is a full kernel eigenvector, $\Gamma$ is a signed \emph{core} graph.
 
It remains to prove that $\Gamma$ is a signed nut graph, i.e.\ with nullity $\eta(\Gamma) = 1$. This is done by showing that the constructed 
 full kernel eigenvector $\bf x$ is the only eigenvector for $\lambda = 0$ (up to a scalar multiple).
  First, note that all edges incident with vertex $0$ have weight $+1$. Hence for $r = 0$, \eqref{eq:newlocCon}
becomes:
\begin{equation}
\sum_{s \neq 0}x_s = 0
\end{equation}
It follows that
\begin{equation}
\sum_{s=0}^{n}x_s = x_0
\end{equation}
Now consider any path $P_4$ with vertices, say, $1,2,3,4$. Note that vertices $1$ and $4$ fall into one symmetry class while vertices $2$ and $3$ are in the other.  The local conditions are:
\begin{align}
0 = x_0 - x_2 + \sum_{s\neq 0,1,2}x_s  & = x_0 - x_1 - 2x_2  \\
0 = x_0 - x_1 -x_3 + \sum_{s\neq 0,1,2,3}x_s &  = x_0 - 2x_1 - x_2-2x_3  \\
0 = x_0 - x_2  - x_4 + \sum_{s\neq 0,2,3,4}x_s &  = x_0 - 2x_2 - x_3 -2x_4  \\
0 = x_0 - x_3 + \sum_{s\neq 0,3,4}x_s  & = x_0 - 2x_3 - x_4 
\end{align}
It is straightforward to show that $x_1, \ldots , x_4$ are related to $x_0$ as:
\begin{equation}
x_1 = x_4 = -x_0, x_2 = x_3 = x_0
\end{equation}
Since this holds for all $k$ path graphs $P_4$ it follows that $\Gamma$ is a signed nut graph.
\end{proof}

We can also prove a theorem that accounts for the next diagonal of Table~\ref{tab:pinky}.
The graphs on this diagonal are $\rho$-regular $\rho \ge 3$ and $\rho = n-2$,
implying even order $n \ge 6$. 
They are therefore the cocktail-party graphs,
$\CP(n)$ (also known as hyperoctahedral graphs \cite[p.~17]{Biggs}, as they are $1$-skeletons of 
the cross-polytope duals of hypercubes).
Exhaustive search shows that there is neither a signed nor an unsigned nut graph for the case $\CP(6)$.  

\begin{lemma}
Let $\Gamma = (G, \Sigma)$ be a signed graph whose underlying graph $G$ is the cocktail-party graph of order $2p$, 
$\CP(2p) \cong \overline{p\ K_2}$, $p > 1$. Then  $\Gamma$ is not a sign-balanced nut graph.
\end{lemma}

\begin{proof}
The spectrum of  $\CP(2p)$ is $\sigma (\CP(2p)) = \{2p-2, 0^{(p)}, -2^{(p - 1)} \}$ \cite[p.~17]{Biggs} and so $\eta(\CP(2p)) = p > 1$.
Therefore, $G$ is not a nut graph, and $\Gamma$ is not switching-equivalent to a sign-balanced nut graph.
\end{proof}

\begin{theorem}
\label{thm:new1}
Let $\Gamma = (G, \Sigma)$ be a signed graph whose underlying graph $G$ is the cocktail-party graph of order $2p$, 
$\CP(2p) \cong \overline{p\ K_2}$. For each even $p$, $p \geq 4$, there exists at least one sign-unbalanced nut graph $\Gamma$.
\end{theorem}

\begin{proof}
Construct $\Gamma(G, \Sigma)$ with $G = \CP(2p)$ and $\Sigma$ defined as follows.
Let $K_{2p}$ be the complete graph with $V(K_{2p}) = \{0, 1, \ldots, 2p-1\}$. Take the Hamiltonian
cycle $[0, 1, 2, \ldots, 2p-1, 0]$ and assign weights  $0$ and $-1$ to alternate edges.
All edges not in the cycle have 
weight $1$.
To obtain the signed graph $\Gamma$, remove from $K_{2p}$ all 
edges that were given weight $0$. 
Note that $\Sigma = \{(1, 2), (3, 4), \ldots, (2p-3, 2p-2), (2p-1, 0)\}$ and
that every vertex has a unique antipodal partner in $\Gamma$; the partner
pairs are the non-edges $\{(0, 1), (2, 3),\ldots, (2p-2, 2p-1)\}$.

\begin{figure}[!htb]
\centering
\includegraphics[scale=0.6]{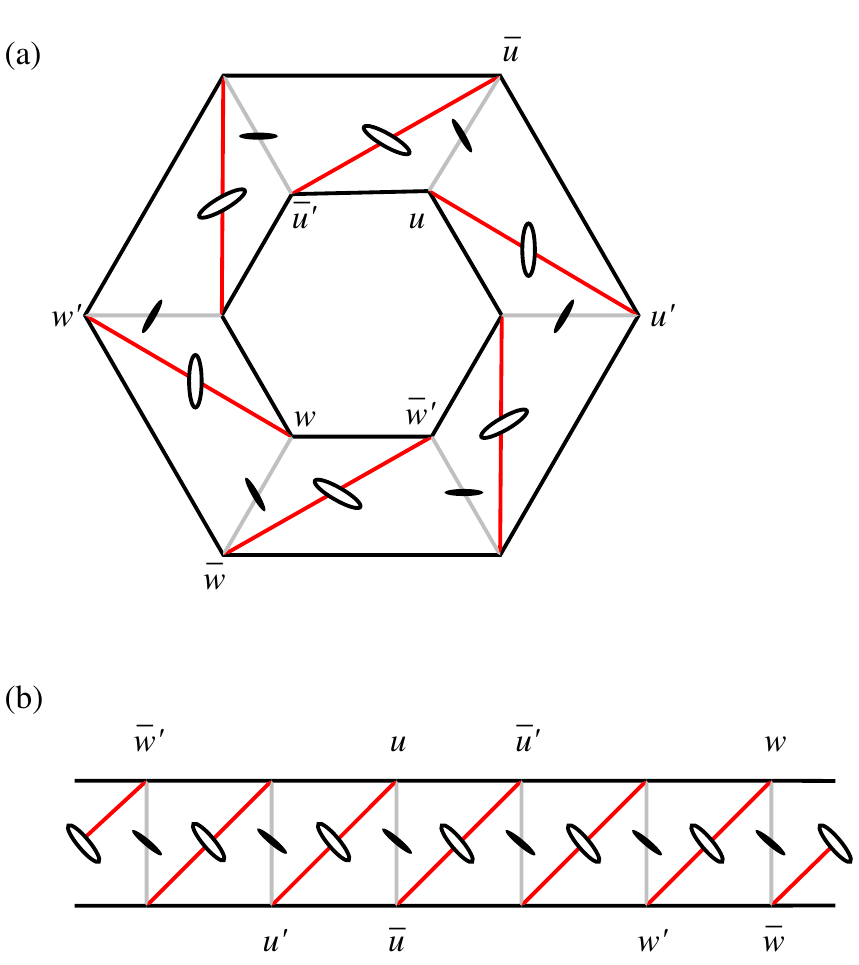}
\caption{
Construction of sign-unbalanced nut graphs $\Gamma(G,\Sigma)$ 
based on cocktail-party graphs $G = \CP(2p)$ where $2p$ is divisible by $4$.
(a)	Schlegel and (b) unfolded-net representations of 
the graph $\Gamma=(\CP(12),\Sigma)$ in a prism 
setting of the chiral point group ${\cal D}_6$. 
Black- and white-filled symbols indicate 
points where two-fold rotational axes enter and leave the prism. 
Axis 
$C_2\sp{\prime}(u)$ enters the prism midway between vertex $u$ and its graph-antipodal 
partner $\overline{u}$, to emerge between vertex $w$ and its graph-antipodal partner 
$\overline{w}$.  
Axis $C_2''(u)$ enters the prism 
between vertex 
$u$ and its neighbour on a negative edge $u\sp\prime$, to emerge between vertices 
$w$ and $w'$.
Red lines denote the edges
with weight $-1$, chosen as described in the proof of Theorem~\ref{thm:new1}.
Grey lines denote the {\lq missing edges\rq}  deleted 
from the complete graph $K_{2p}$.
Black lines denote edges of weight $+1$, though not all are shown: the graph has 
an edge of weight $+1$ connecting each vertex $w$ with all vertices other
than $\overline{w}$, $w'$ and $w$ itself. 
}
\label{fig:pwf_fig}
\end{figure}

With this assignment of $\Sigma$,
the local condition for an eigenvector $\x$ corresponding to 
an eigenvalue $\lambda$ of $\Gamma$ is
\begin{equation}
\label{eq:localc}
\lambda x_u = \big( \hspace{-10pt} \sum_{v \notin \{u, \overline{u}, u'\}} \hspace{-10pt} x_v\big) - x_{u'}
\end{equation}
for all $u \in V(G)$, where $\overline{u}$ is the antipodal partner of $u$, 
and $u'$ is the (unique) vertex connected to $u$ by a negative edge.
Condition \eqref{eq:localc} can be rewritten as
\begin{equation}
\label{eq:localc2}
\lambda x_u = \big( \hspace{-6pt} \sum_{v \in V(G)} \hspace{-6pt} x_v\big)   - x_u - x_{\overline{u}} - 2x_{u'},
\end{equation}
and summing over all $u \in V(G)$ gives
\begin{equation}
(\lambda - 2p + 4)\sum_{u \in V(G)}  x_u = 0 .
\label{eq:localc3}
\end{equation}
Hence, if $\lambda = 0$ the desired vector has $S = \sum_{v \in V(G)}  x_v = 0$
and \eqref{eq:localc2} reduces to
\begin{equation}
x_u + x_{\overline{u}} = -2x_{u'} = -2x_{\overline{u}'} = -(x_{u'} +x_{\overline{u}'}) ,
\label{eq:localc4}
\end{equation}
where the second equality follows by substituting $u$ for $\overline{u}$ in the first (with the labelling convention $\overline{u}' = (\overline{u})'$).

Note that the vector $\y$ with entries specified by
\begin{equation}
y_u = \begin{cases}
+ 1, & \text{if } u \equiv 0 \text{ or } 1 \pmod{4}, \\
- 1, & \text{otherwise},
\end{cases}
\label{eq:ydef}
\end{equation}
is both full and a kernel eigenvector of $\Gamma$. Therefore, $\Gamma$ is a signed core graph.
We now show that the nullity $\eta(\Gamma)$ is not greater than $1$.

Uniqueness of the kernel vector 
follows from a symmetry 
argument.  The dihedral group ${\cal D}_p$ is a subgroup of $\Aut|G|$ and 
$\Aut|\Gamma|$. 
This fact is given a spatial realisation by identifying the vertices of 
$\Gamma$ with those of a regular $[p]$-gonal prism, so that edges in $\Sigma$ correspond to 
consistently oriented diagonals on vertical quadrilateral faces, pairs of antipodal partners in 
$\Gamma$ lie on vertical edges of the prism, and positively weighted 
edges of $\Gamma$ are symmetrically 
distributed on faces and in the prism interior. 
The group ${\cal D}_p$ with even $p$ has two classes of two-fold 
rotational axis perpendicular to the principal ($C_p$) axis, 
 passing through midpoints of 
vertical edges, and centres 
of vertical faces, respectively. 
These generate $C_2\sp\prime$ and 
$C_2\sp{\prime\prime}$ rotations, respectively.
Figure~\ref{fig:pwf_fig} illustrates the 
case of $n = 2p = 12$. 

Consider the local rotation $C'_2(u)$. This exchanges $u$ with $\overline{ u}$ and 
$u'$ with $\overline{u}'$.  As the operation is an involution partitioning
the vertices into orbits of size $2$,
the eigenspace for any given 
eigenvalue
$\lambda$ of $\Gamma$ can be partitioned such 
that each eigenvector is either symmetric or antisymmetric under the action of 
$C_2\sp\prime(u)$. In particular, for kernel vector ${\x}$, we have 
\begin{equation} 
C_2\sp\prime(u) {\x} = \chi(C_2\sp\prime(u)) {\x}
\end{equation} 
with character $\chi(C_2\sp\prime(u)) = \pm 1$.  
From the local condition \eqref{eq:localc4} we know that $ x_{u\sp\prime} = x_{{\overline u}'}$ for 
every kernel vector of $\Gamma$, and so $\chi(C_2\sp\prime(u)) = +1$, and 
$x_{u} = x_{{\overline u}}$. Also from \eqref{eq:localc4}, we have $x_{u\sp\prime} = 
x_{{\overline u}\sp\prime} =-x_{u\sp\prime} = -x_{{\overline u}\sp\prime}$ for every kernel 
vector of $\Gamma$. As $ C_2\sp{\prime\prime}(u)$ 
exchanges $u$ and $u\sp\prime$, this relation implies $\chi(C''_2(u)) = -1$. 
Operations $ C_2\sp\prime(u)$ (respectively, $ C_2\sp{\prime\prime}(u)$) for $u$ form a 
class within the group ${\cal D}_p$, and hence 
have the same character 
$\chi(C_2\sp\prime)$ (respectively, $\chi(C_2\sp{\prime\prime})$).

We may propagate the vector ${\x}$ starting from the entry $x_u = a$, 
using consecutive $C_2\sp\prime(u)$ and $C_2\sp{\prime\prime}(u)$ operations, shifting
the axis each time by a rotation of
$\pi/p$ about the principal axis.  
This process covers every vertex with an entry $|x_v|=a$ and
returns consistently to $u$.
The resulting vector ${\x}$ is $a\y$, where $\y$ is the kernel 
vector given by \eqref{eq:ydef}; it has equal entries on antipodal partners in $G$,
but equal and opposite entries on pairs of vertices linked by a negative edge in $\Gamma$.  
The underlying graph $G$ is not a nut graph.
Therefore, $\Gamma$ is a sign-unbalanced nut graph.
\end{proof}

{\color{black}
\begin{theorem}
\label{thm:new1b}
Let $\Gamma = (G, \Sigma)$ be a signed graph whose underlying graph $G$ is the cocktail-party graph of order $2p$, 
$\CP(2p) \cong \overline{p\ K_2}$. For each odd $p$, $p \geq 5$, there exists at least one sign-unbalanced nut graph $\Gamma$.
\end{theorem}

\begin{proof}
\begin{figure}[!b]
\centering
\includegraphics[scale=0.6]{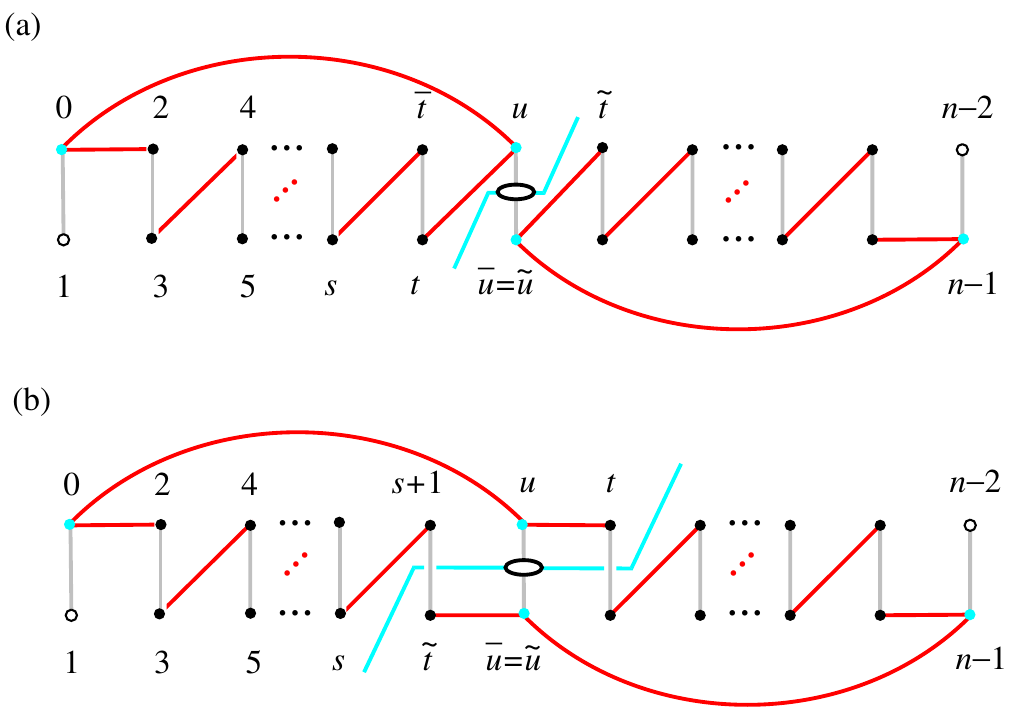}
\caption{
Construction of sign-unbalanced nut graphs $\Gamma(G,\Sigma)$ 
based on cocktail-party graphs $G = \CP(2p)$ where $2p$ is not divisible by $4$.
(a) Signed graph $\Gamma_{a}$ for the  case $n = 2p = 4q+2$, with $q$ even.
(b) Signed graph $\Gamma_{b}$ for the  case $n = 2p = 4q+2$, with $q$ odd.
In each case, the graph is shown in a shorthand ribbon format where each 
vertex lies directly above its antipodal partner.  
Colour coding of edges and non-edges is as in Figure~\ref{fig:pwf_fig}.
Vertices are colour coded by the number of incident edges  of weight $-1$: 
white for $0$,
black for $1$ and blue for $2$.
The elliptical symbol
indicates the axis of the unique two-fold rotation, $C_2(u)$ that swaps vertices
$u \equiv (n/2)-1$ and ${\overline u} \equiv (n/2)$.
(The notation ${\overline w}$ indicates the antipodal partner of vertex $w$, and 
${\widetilde w}$ its image under the rotation.)
The blue lines show the different splitting of 
$\Gamma_{a}$ and $\Gamma_{b}$ into fragments $H$ (to the left) 
and ${\widetilde H}$ (to the right).
Constructions (a) and (b) differ in the position of $t$, 
the end vertex of the path of negative edges:
$t$ is at position $u -1$ in $\Gamma_{a}$, but $u+2$ in $\Gamma_{b}$.
}
\label{fig:pwfC2_fig}
\end{figure}

Let $p= 2q+1$.
The proof separates into Case 1, where $q$ is even, and Case 2, where $q$ is odd.
In both cases we construct $\Gamma$ to have $p+1$ edges in $\Sigma$.
In both, a symmetry argument is used to split the cases further according to the character
($\pm1$) of a candidate kernel vector under a two-fold rotation.

\vspace{0.5\baselineskip}
\noindent Case 1: $n = 2p = 4q+2$, $q$ is even.
The candidate graph $\Gamma_{a}$ for this case is illustrated
in Figure~\ref{fig:pwfC2_fig}(a).
The signed graph retains only one 
axis of 
two-fold rotational symmetry, which
exchanges vertex $u$ with its antipodal partner $\overline{ u}$;
every vertex $w$ has a unique rotational partner 
${\widetilde w}$, but only for $u$ is $\overline{ w} = {\widetilde w}$.

The graph $\Gamma$
can be divided into two {\lq halves\rq} $H$ and ${\widetilde H}$
related by the $C_2(u)$ rotation (see Figure~\ref{fig:pwfC2_fig}(a)).
Fragment $H$ contains a path of three negative edges 
($2 - 0 - u - t$),
a set of $q-2$ isolated negative edges (with lower endpoints $3, 5, \ldots, s$), 
and exactly one vertex that is on no negative edge.

A full kernel vector of $\Gamma_{a}$ is easily found.
The vector $\z$ is defined to have entries in $H$
$z_0 = -z_1 = -z_2 =1$, and 
$z_{2w+1} = - z_{2w+2} = (-1)\sp{w}$
for $1 \le w \le q -1$, and entries in ${\widetilde H}$
given by $z_{\widetilde v} = -z_v$ .  The vector $\z$ constructed in this way
is full, with eigenvalue $0$
and character
$-1$ under the rotation $C_2(u)$. 

To check whether other  kernel vectors with this character exist,
consider a putative kernel vector ${\x}$.
All kernel vectors of $\Gamma_{a}$ obey the local condition (pivoting on vertex $w$)
\begin{equation}
0 = \big( \hspace{-6pt} \sum_{v \in V(G)} \hspace{-6pt} x_v\big)
 - x_{w} - x_{\overline w} -2 \hspace{-8pt} \sum_{\sigma(vw)=-1} \hspace{-8pt} x_v
 \label{eq:loco}
\end{equation}
Note that vertex $1$ (and $\widetilde{ 1} \equiv 2p - 2$) have no incident 
negative edges.
As the kernel vector $\x$ has character $-1$ under $C_2(u)$, 
$x_{\widetilde w} =-x_w$ for all $w$.
The sum $ S = \sum_{v \in V(G)} x_v$ vanishes by symmetry for vectors
with $\chi(C_2(u)) = -1$. 

Entries on vertices $0$ and $u = n/2 -1 = 2q$ are assigned variables $b = x_0$ 
and $a = x_u = x_{\overline u}$. 
From the local condition it is apparent that all entries in $\x$ are 
linear combinations of $a$ and $b$. 
Pivoting on $1, 0, 2, \ldots$ and on $u, t, {\overline t}, \ldots$ gives
\begin{equation}
\begin{aligned}
x_0 & = b, &
x_1 & = -b, &
x_2 & = -a, &
x_3 & = a - 2b, &
x_4 & = b, &  \quad  \ldots  \\
x_u & = a, &
x_{\overline u} & = -a, &
x_t   &     = x_{u-1} =  -b,  &
x_{\overline t} & = x_{u-2} = b - 2a, & 
x_{s} & = x_{u-3} = a, & \quad \ldots 
\end{aligned}
\label{eq:local}
\end{equation}
Now we propagate vector $\x$ starting from the left side of $H$ and determine 
the relationship between parameters $a$ and $b$ by requiring
consistency when the neighbourhood of $u$ is reached.
(N.B. For the smallest case, $n = 10, q = 2$, (\ref{eq:local}) with $x_t \equiv x_3$ 
already gives $a = b$ and $\x = a\z$.)

The fragment $H$ contains $(q-2)$ independent negative edges.
Consecutive application of the local condition to each vertex of this 
ladder of negative $K_2$ units, $(3, 4), (5,6), \ldots, (u-3,  u-1)$ gives 
a recurrence relation for the entries in $\bf x$
on top and bottom rows of Figure~\ref{fig:pwfC2_fig}(a),
with solution for $q \ge 4$:
\begin{equation}
\begin{aligned}
x_{2w+1} & = (-1)\sp{w+1} (w a - (w+1)b) & \quad & (1 \le w \le q-2) \\
x_{2w+4}  & = x_{2w+1} && (1 \le w \le q-2)
\end{aligned}
\label{eq:rec1}
\end{equation}
Consistency of  ({\ref{eq:local}})  and ({\ref{eq:rec1}}) 
for $x_{\overline t} \equiv x_{2w+4} = x_{2w+1}$  with $w=q-3$, requires $a = b$ and
and therefore $\x=a\z$. 
Hence, $\Gamma_{a}$ is a core with exactly one kernel vector that has
$\chi(C_2(u)) = -1$.

It remains to check that $\Gamma_{a}$ has no kernel vector with
$\chi(C_2(u)) = +1$, i.e. with
$x_u  = x_{\widetilde u} = x_{\overline u}$
We assign three variables $a = x_u$, $b = x_0$ and $c = x_1$.  
The extra parameter caters for the fact that the
sum $S$ is not now guaranteed to
vanish by symmetry. 
Pivoting on vertex $1$ shows that $S = b + c$. 
Proceeding as before, local conditions fix 
entries with lowest and 
highest labels in ${H}$:
\begin{equation}
\begin{aligned}
x_0 & = b, &
x_1 & = c, &
x_2 & = -a, &
x_3 & = a - b +c, &
x_4 & = b, & \quad  \ldots  \\
x_u & = a, &
x_{\overline u} & = a, &
x_t   &     = -a -{\textstyle{1\over2}}b +{\textstyle{1\over2}}c,  &
x_{\overline t} & = -a +{\textstyle{3\over2}b} +{\textstyle{1\over2}}c,  &
x_s & = x_{u-3}  = a, & \quad \ldots 
\end{aligned}
\label{eq:locala2}
\end{equation}
Working along the ladder of vertices in sequence from $3$,
recurrence relations are found:
\begin{equation}
\begin{aligned}
x_{4w+1} &= -2wa +(w+1)b -wc &\quad& (1 \le w \le q/2-1)  \\
x_{4w+3} &= (2w+1)a -(w+1)b +(w+1)c && (1 \le w \le q/2-2) \\
x_{4w} &= x_{4w-3} = x_{4(w-1)+1} && (1 \le w \le q/2-1)  \\
x_{4w+2} &= x_{4w-1} = x_{4(w-1)+3} && (0\le w \le q/2-2) 
\end{aligned}
\label{eq:rec2}
\end{equation}

There are three consistency conditions to be applied to reduce the number of free parameters.
Entries $x_{\overline t}$ and $x_s$ must match general relation ({\ref{eq:locala2}}),
and the sum of all entries must be consistent
with $S$ as obtained from the local condition at vertex $1$. 
The first two conditions are
\begin{align}
q(2a-b+c) & = 2a + 2c   \tag{C1}\label{eq:C1} \\
q(2a-b+c) & = 4a + b + 3c  \tag{C2}\label{eq:C2}
\end{align}
The third follows by summing the local condition for 
$\lambda = 0$ ({\ref{eq:loco}}) 
over all vertices of $\Gamma_b$, to give 
\begin{equation}
0 = (n-2)S - 2 \sum_{v \in V(G)} \rho_-(v) x_v,
\label{eq:sumc3}
\end{equation}
where $\rho_{-}$ is as introduced in the proof of
Theorem~\ref{thm:10}; see \eqref{eq:rhominuseq}.
Noting that $x_{w} = x_{\widetilde w}$
for character $+1$, $S = b + c$, and 
that apart from vertices $0$, $u$ and $1$, all vertices $w$
and their rotational partners have $\rho_-(w) = 1$, the result is 
\begin{equation}
(2q-3)(b+c) = 2a -4c. \tag{C3}\label{eq:C3}
\end{equation}
It is straightforward to show that conditions \eqref{eq:C1}
to \eqref{eq:C3} admit
only the trivial solution $a = b = c = 0$,
as the determinant of the coefficients in the $3\times 3$ set of linear equation vanishes only for $q = 0$,
therefore $\x$ is the null vector and
$\Gamma_a$ is a sign-unbalanced nut graph.

\vspace{0.5\baselineskip}
\noindent Case 2: $n = 2p = 4q+2$, $q$ is odd.
The candidate graph $\Gamma_{b}$ for this case is illustrated
in Figure~\ref{fig:pwfC2_fig}(b).
The signed graph retains the two-fold rotational symmetry
exchanging vertex $u = 2q$ with its antipodal partner ${\overline u}$,
but now the path of three negative edges has endpoint $t$ at
position $u+2$, leading 
to a different partition into fragments $H$ and ${\widetilde H}$. 
The argument follows the same lines as Case 1. A full kernel vector 
is readily constructed for $\Gamma_{b}$. 
Here, this is $\z'$, with entries in $H$ defined by
${z_0'} = -{z_1'} = -{z_2'} = {z_u'} = -{z_t'} =1$, and 
${z_{2w+1}'} = - {z_{2w+2}'} = (-1)\sp{w}$
for $1 \le w \le n/2 -4$, and entries in
${\widetilde H}$ found from
${z_v'} = -{z'_{\widetilde v}}$. 
By construction, this vector is antisymmetric under the reflection $C_2(u)$.

As in Case 1, we construct a general kernel vector $\x$
with $\chi(C_2(u)) = -1$,  working outwards from vertex $0$:
\begin{equation}
\begin{aligned}
x_0 & = b, &
x_1 & = -b, &
x_2 & = -a, &
x_3 & = a - 2b, &
x_4 & = b, & & & \quad  \ldots  \\
x_u & = a, &
x_{\overline u} & = -a, &
x_t   &     = -b,  &
x_{\widetilde t} & = b,  &
x_{s+1} & = 2a-b, &
x_s & = -a, & \quad \ldots 
\end{aligned}
\label{eq:locala3}
\end{equation}

\noindent
Here $s+1$ labels the final vertex of the ladder of $(q-2)$ 
disjoint negative edges, $s+1= u-2 = 2q-2$ (see Figure~\ref{fig:pwfC2_fig}(b)).

The recurrence relation for vertices in $H$ is 
\begin{equation}
\begin{aligned}
x_{2w} & = (-1)\sp{w+1} ( w a-(w+1)b) & \quad & (1 \le w \le q-2) \\
x_{2w+2}  & = x_{2w11} && (1 \le w \le q-2)
\end{aligned}
\label{eq:rec4}
\end{equation}
Consistency with the expression for $x_{s+1}$ requires $a = b$
and $\x=a\z'$.
Hence, $\Gamma_{b}$ is a core and has exactly one kernel vector that has $\chi(C_2(u)) = -1$.

It remains to check that $\Gamma_{b}$ has no kernel vector with $\chi(C_2(u)) = +1$.
Constructing a general kernel vector $\x$ with character $+1$, assigning parameters 
$a, b, c$ as in Case 1 gives
 the same entries as in \eqref{eq:locala2} and the same recurrence relations for the forward
 direction in $H$ as in
 \eqref{eq:rec2}, if we note that the entry for ${\overline t}$ for $\Gamma_a$ 
 now applies to ${s+1}$.
Consistency conditions \eqref{eq:C1} and \eqref{eq:C3} also apply here.
Condition \eqref{eq:C2} is replaced by \eqref{eq:C2prime}, as the offset
of vertex $s+1$ from $u$ is $2$ rather than $4$:
\begin{equation}
q(2a-b+c)  = 8a -4b + 2c'  \tag{$\text{C2}'$}\label{eq:C2prime}
\end{equation}
Again, it is straightforward to show that the three conditions admit
only the trivial solution $a = b = c = 0$, and $\x$ is the null vector. 
Hence, $\Gamma_{b}$ is  a sign-unbalanced nut graph.
\end{proof}
}

Taken together, Theorems \ref{thm:new1} and \ref{thm:new1b} characterise the sub-diagonal $(\rho, n) = (n - 2, n)$ of Table~\ref{tab:pinky}.
This diagonal is populated for all $n \geq 8$, and then by sign-unbalanced nut graphs only.

A further theorem deals with the next sub-diagonal of
Table~\ref{tab:pinky}, the cases $(\rho,n)$ of regular graphs with 
$\rho = n-3$. All graphs of this type are derived from the complete
graph $K_n$ by deletion of all edges of a union of disjoint cycles of total size $n$.
Numbers of graphs at each order follow OEIS \cite{OEISgeneral} sequence A008483 \cite{oeis}.

\begin{theorem}
\label{thm:new1c}
Let $G$ be a graph of order $n$ derived from $K_n$ by 
deletion of the edges of a union of cycles $C_p$ ($p\ge 3$), 
where $\sum p = n$. Then $G$ is not a sign-balanced nut graph.
\end{theorem}
\begin{proof}
The graph $G$ is $\rho$-regular with $\rho=n-3$.
Let the spectrum of $G$ be $\{\lambda_i\}$, $i=1, \ldots, n$. 
As $G$ is regular, the complement $\overline{G}$ has 
spectrum
$\{\theta_i\} = \{n-1-\rho, -\lambda_2-1, \ldots, -\lambda_n-1\}$,
$i=1, \ldots, n$ \cite[Section~1.3.2, p.~4]{haemers}.
Hence, the nullity of $G$ is equal to the multiplicity of eigenvalue $-1$ in 
the spectrum of $\overline{G}$.
A cycle has either two eigenvalues $-1$ ($p \equiv 0 \pmod 3$) or none (otherwise).
Therefore, the nullity of $G$ is even, and $G$ is not a nut graph.
\end{proof} 

For the case $n=7$, a sign-unbalanced nut graph based on
$G = K_7 - C_7$ with only one negative edge exists 
(see  Appendix~\ref{sec:appA}), but 
a more symmetrical choice of $\Sigma$ has
a cycle of $7$ negative edges, and all entries of the kernel vector equal.

\section{A construction for signed nut graphs}

Theorem {\ref{thm:main}}
has answered our initial question, in that
if we consider ordinary graphs as special cases of signed graphs, we need only to perform a computer search for existence of 
signed nut graphs for those values of $n$ for which no ordinary nut graph exists. 
 However, if we want to search for sign-unbalanced nut graphs with the intention of
determining the orders for which
a sign-unbalanced nut graph exists, 
then methods are needed for
generating larger signed nut graphs.
There are several known constructions that take a nut graph and produce a larger nut graph \cite{Sc2008}. 
We will revisit one such here and extend it
to signed graphs. 
This is the so-called Fowler construction for enlarging unweighted nut graphs.

Let $G$ be  a graph and $v$ a vertex of  degree $\rho$.
Let $N(v) = \{u_1, u_2, \ldots, u_\rho \}$.
Recall \cite{Jan,GPS} that the Fowler Construction, denoted $F(G, v)$, is a graph with 
\begin{equation}
\begin{aligned}
V(F(G, v)) = V(G) \sqcup \{ q_1, \ldots, q_\rho \} \sqcup \{ p_1, \ldots , p_\rho \}
\end{aligned}
\end{equation}
and
\begin{equation}
\begin{aligned}
E(F(G, v)) = (E(G) \setminus \{ vu_i \mid 1 \leq i \leq \rho \}) & \cup \{ q_ip_j \mid 1 \leq i, j \leq \rho, i \neq j \} \\
 & \cup \{ vq_i \mid 1 \leq i \leq \rho \} \cup \{ p_i u_i  \mid 1 \leq i \leq \rho \}.
\end{aligned}
\end{equation}

Here, we generalise this construction to signed graphs.

\begin{definition}
Let $\Gamma = (G, \sigma)$ be a signed graph and $v$ a
vertex of $G$ that has degree $\rho$. 
Then $F(\Gamma, v) = (F(G, v), \sigma')$, where for  $1 \leq i, j \leq \rho$, 
\begin{equation}
\sigma'(e) = \begin{cases}
1 & \text{if } e = vq_i,  \\
1 & \text{if } e = q_ip_j, \\
\sigma(vu_i) & \text{if } e = p_iu_i, \\
\sigma(e) & \text{otherwise},
\end{cases}
\end{equation}
is \emph{the Fowler Construction for signed graphs}.
\end{definition}

\begin{lemma}
Let $\Gamma = (G, \sigma)$ be a signed graph and $v$ a
vertex of $G$ that has degree $\rho$ and
let $\bf x$ be a kernel eigenvector for $\Gamma$. Then $\bf x'$, defined as
\begin{equation}
{\bf x'}(w) = \begin{cases}
-(\rho - 1){\bf x}(v) &  \textup{if } w = v, \\
\sigma(vu_i){\bf x}(u_i) &  \textup{if } w = q_i, \\
{\bf x}(v) &  \textup{if } w = p_i, \\
{\bf x}(w) & \textup{otherwise},
\end{cases}
\label{eq:feigenvector}
\end{equation}
for $w \in V(F(\Gamma, v))$,
is a kernel eigenvector for $F(\Gamma, v)$.
\end{lemma}

The local structures in the signed graphs $\Gamma$ and $F(\Gamma, v)$ are shown in Figure~\ref{Fig-FowlerExt}, which also indicates the
relationships between kernel eigenvectors in these graphs.

\begin{figure}[!htb]
\subcaptionbox{$\Gamma$\label{Fig-FowlerExta}}[.5\linewidth]{
\centering
\begin{tikzpicture}
\definecolor{mygreen}{RGB}{205, 238, 231}
\tikzstyle{vertex}=[draw,circle,font=\scriptsize,minimum size=13pt,inner sep=1pt,fill=mygreen]
\tikzstyle{edge}=[draw,thick]
\coordinate (u1) at (-2, -1);
\coordinate (u2) at (-0.7, -1);
\coordinate (ud) at (2, -1);
\path[edge,fill=yellow!20!white] (u1) .. controls ($ (u1) + (-60:0.5) $) and ($ (u2) + (-120:0.5) $) .. (u2) .. controls ($ (u2) + (-60:0.7) $) and ($ (ud) + (-120:0.7) $)
.. (ud) .. controls ($ (ud) + (-60:3) $) and ($ (u1) + (-120:3) $) .. (u1);
\node[vertex,fill=mygreen,label={[yshift=0pt,xshift=0pt]90:$a$}] (v) at (0, 0) {$v$};
\node[vertex,fill=mygreen,label={[xshift=2pt,yshift=0pt]180:$b_1$}] (ux1) at (u1) {$u_1$};
\node[vertex,fill=mygreen,label={[xshift=2pt,yshift=0pt]180:$b_2$}] (ux2) at (u2) {$u_2$};
\node[vertex,fill=mygreen,label={[xshift=-2pt,yshift=0pt]0:$b_\rho$}] (uxd) at (ud) {$u_\rho$};
\node at (0.5, -1) {$\ldots$};
\path[edge,color=red] (v) -- (ux1) node[midway,xshift=-23pt,yshift=0,color=black] {\footnotesize $\sigma(vu_1)$};
\path[edge,color=red] (v) -- (ux2) node[midway,xshift=16pt,yshift=0,color=black] {\footnotesize $\sigma(vu_2)$};
\path[edge] (v) -- (uxd) node[midway,xshift=20pt,yshift=0,color=black] {\footnotesize $\sigma(vu_\rho)$};
\end{tikzpicture}
}%
\subcaptionbox{$F(\Gamma, v)$\label{Fig-FowlerExtb}}[.5\linewidth]{
\centering
\begin{tikzpicture}
\definecolor{mygreen}{RGB}{205, 238, 231}
\tikzstyle{vertex}=[draw,circle,font=\scriptsize,minimum size=13pt,inner sep=1pt,fill=mygreen]
\tikzstyle{edge}=[draw,thick]
\coordinate (u1) at (-2, -4);
\coordinate (u2) at (-0.7, -4);
\coordinate (ud) at (2, -4);
\path[edge,fill=yellow!20!white] (u1) .. controls ($ (u1) + (-60:0.5) $) and ($ (u2) + (-120:0.5) $) .. (u2) .. controls ($ (u2) + (-60:0.7) $) and ($ (ud) + (-120:0.7) $)
.. (ud) .. controls ($ (ud) + (-60:3) $) and ($ (u1) + (-120:3) $) .. (u1);
\node[vertex,fill=mygreen,label={[yshift=-4pt,xshift=0pt]90:$(1-\rho)a$}] (v) at (0, 0) {$v$};
\node[vertex,fill=mygreen,label={[xshift=2pt,yshift=0pt]180:{\footnotesize $\sigma(vu_1)b_1$}}] (q1) at (-2, -1) {$q_1$};
\node[vertex,fill=mygreen,label={[xshift=-2pt,yshift=0pt]0:{\footnotesize $\sigma(vu_2)b_2$}}] (q2) at (-0.7, -1) {$q_2$};
\node[vertex,fill=mygreen,label={[xshift=-2pt,yshift=0pt]0:{\footnotesize $\sigma(vu_\rho)b_\rho$}}] (qd) at (2, -1) {$q_\rho$};
\node at (1.25, -1) {$\ldots$};
\node[vertex,fill=mygreen,label={[xshift=2pt,yshift=0pt]180:$a$}] (p1) at (-2, -2.5) {$p_1$};
\node[vertex,fill=mygreen,label={[xshift=2pt,yshift=0pt]180:$a$}] (p2) at (-0.7, -2.5) {$p_2$};
\node[vertex,fill=mygreen,label={[xshift=0pt,yshift=0pt]0:$a$}] (pd) at (2, -2.5) {$p_\rho$};
\node at (0.5, -2.5) {$\ldots$};
\node[vertex,fill=mygreen,label={[xshift=2pt,yshift=0pt]180:$b_1$}] (ux1) at (-2, -4) {$u_1$};
\node[vertex,fill=mygreen,label={[xshift=2pt,yshift=0pt]180:$b_2$}] (ux2) at (-0.7, -4) {$u_2$};
\node[vertex,fill=mygreen,label={[xshift=-2pt,yshift=0pt]0:$b_\rho$}] (uxd) at (2, -4) {$u_\rho$};
\node at (0.5, -4) {$\ldots$};
\path[edge] (v) -- (q1);
\path[edge] (v) -- (q2);
\path[edge] (v) -- (qd);
\path[edge] (p1) -- (q2); \path[edge] (p1) -- (qd);
\path[edge] (p2) -- (q1); \path[edge] (p2) -- (qd);
\path[edge] (pd) -- (q1); \path[edge] (pd) -- (q2);
\path[edge,color=red] (p1) -- (ux1) node[midway,xshift=-15pt,yshift=0,color=black] {\footnotesize $\sigma(vu_1)$}; 
\path[edge,color=red] (p2) -- (ux2) node[midway,xshift=16pt,yshift=0,color=black] {\footnotesize $\sigma(vu_2)$}; 
\path[edge] (pd) -- (uxd) node[midway,xshift=16pt,yshift=0,color=black] {\footnotesize $\sigma(vu_\rho)$};
\end{tikzpicture}
}%
\caption{A construction for expansion of a signed nut graph $\Gamma$ about vertex $v$ of degree $\rho$, to give $F(\Gamma, v)$.
The labelling of vertices in $\Gamma$ and $F(\Gamma, v)$ is shown within the circles that represent vertices.
Shown beside each vertex is the corresponding entry of the (unique) kernel eigenvector with integer entries for the respective graph.
Panel (a) shows the neighbourhood of vertex $v$ in $\Gamma$. 
Edges from vertex $v$ to its neighbours have weights
$\sigma(vu_i)$ which are either $+1$ or $-1$. In the figure, edges with weight $-1$ are indicated in red, as an illustration.
Edges of the remainder of the graph, indicated by the shaded bubble, may take arbitrary signs.
Panel (b) shows additional vertices and edges in $F(\Gamma, v)$.
Vertices $q_i$ inherit their entries from $\Gamma$ as described in Equation \eqref{eq:feigenvector}.
Edges $p_iu_i$ inherit their weights (signs) from $\Gamma$.
All other new edges in Panel (b) have weights $+1$.} 
\label{Fig-FowlerExt}
\end{figure}
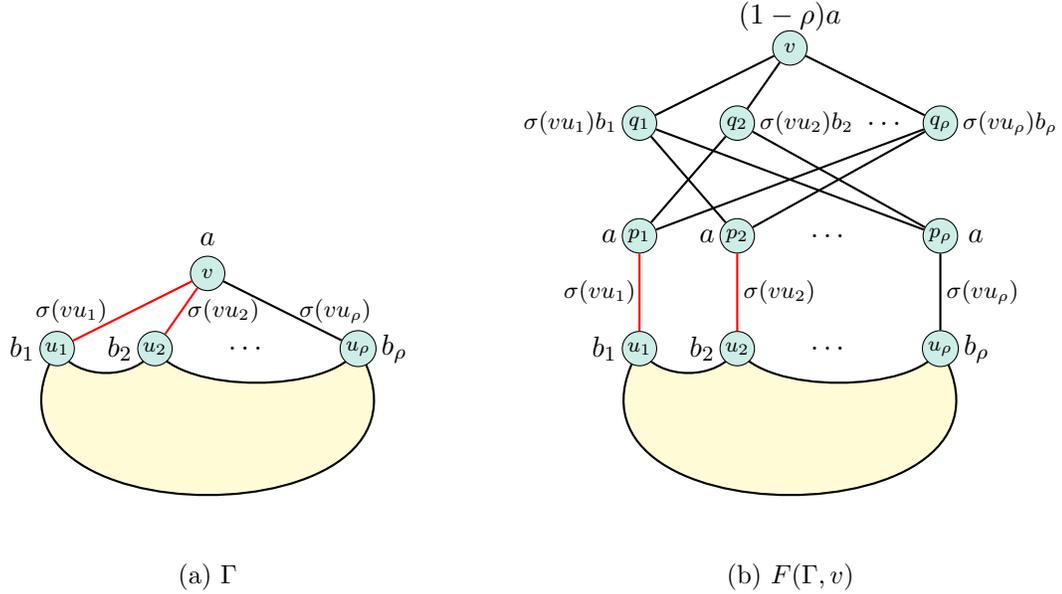

\begin{lemma}
\label{EquivLabeling}
Let $\Gamma$  be a singular signed graph, and let $\x$ be a kernel eigenvector. Let $u,v \in V$ be any two non-adjacent vertices, having the same degree,
say $\rho$, and sharing $\rho-1$ neighbours. Let $u'$ denote the neighbour of $u$ that is not a neighbour of $v$, and let $v'$ denote the neighbour of $v$
that is not a neighbour of $u$.  
If $\sigma(uw) = \sigma(wv)$ for all $w \in N(u) \setminus \{u'\}$, then
$|\x(u')| = |\x(v')|$. Moreover, $\x(u') = \x(v')$ if and only if $\sigma(vv') = \sigma(uu')$.
\end{lemma}

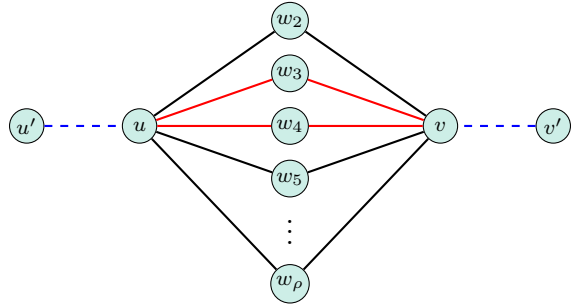
\begin{figure}[!htbp]
\centering
\begin{tikzpicture}
\definecolor{mygreen}{RGB}{205, 238, 231}
\tikzstyle{vertex}=[draw,circle,font=\scriptsize,minimum size=13pt,inner sep=1pt,fill=mygreen]
\tikzstyle{edge}=[draw,thick]
\node[vertex] (w1) at (0, 0) {$w_2$};
\node[vertex] (w2) at (0, -0.7) {$w_3$};
\node[vertex] (w3) at (0, -1.4) {$w_4$};
\node[vertex] (w4) at (0, -2.1) {$w_5$};
\node[vertex] (wk) at (0, -3.5) {$w_{\rho}$};
\node[vertex] (u) at (-2, -1.4) {$u$};
\node[vertex] (up) at (-3.5, -1.4) {$u'$};
\node[vertex] (v) at (2, -1.4) {$v$};
\node[vertex] (vp) at (3.5, -1.4) {$v'$};
\path[edge,color=blue,dashed] (up) -- (u);
\path[edge,color=blue,dashed] (vp) -- (v);
\path[edge] (u) -- (w1) -- (v);
\path[edge,color=red] (u) -- (w2) -- (v);
\path[edge,color=red] (u) -- (w3) -- (v);
\path[edge] (u) -- (w4) -- (v);
\path[edge] (u) -- (wk) -- (v);
\node at (0, -2.7) {$\vdots$};
\end{tikzpicture}
\caption{The neighbourhood of vertices $u$ and $v$ in Lemma~\ref{EquivLabeling}. The red
edges indicate a possible selection of edges with weight $-1$.}
\label{fig:locLemma}
\end{figure}

\begin{proof}
Let $N(u) \setminus \{u'\} = N(v) \setminus \{v'\} = \{w_2, \ldots, w_\rho \}$ (see Figure~\ref{fig:locLemma}). The respective local conditions at vertices $u$ and $v$ are
\begin{align}
\sigma(uu') \mathbf{x}(u') + \sum_{i=2}^{\rho} \sigma(uw_i) \mathbf{x}(w_i)& = 0, \label{eq:loc1} \\
\sigma(vv') \mathbf{x}(v') + \sum_{i=2}^{\rho} \sigma(w_iv)\mathbf{x}(w_i) & = 0. \label{eq:loc2}
\end{align}
Since $\sigma(uw_i) = \sigma(w_iv)$ for all $2 \leq i \leq \rho$, we get that
\begin{equation}
 \sigma(uu')  \mathbf{x}(u')= \sigma(vv') \mathbf{x}(v') ,
\label{eq:xeseql}
\end{equation}
by taking the difference of Equations \eqref{eq:loc1} and \eqref{eq:loc2}. Clearly,
\begin{equation}
|\mathbf{x}(u')| = |\sigma(uu') \mathbf{x}(u') | = |\sigma(vv') \mathbf{x}(v')| = |\mathbf{x}(v')|.
\end{equation}
If $\sigma(uu') = \sigma(vv')$, then Equation~\eqref{eq:xeseql} implies $\mathbf{x}(u') = \mathbf{x}(v')$.
Similarly, if $\mathbf{x}(u') = \mathbf{x}(v')$, then Equation~\eqref{eq:xeseql} implies $\sigma(uu') = \sigma(vv')$.
\end{proof}

\begin{lemma}
\label{lem:lemmaxxx}
Let $\Gamma$ and $\Gamma'$ be signed graphs over the same 
underlying
graph $G$, i.e.\ 
$\Gamma = (G, \sigma)$ and $\Gamma' = (G, \sigma')$.
Let $v$ be a vertex of $G$. Then $\Gamma$ is switching equivalent to $\Gamma'$
if and only if $F(\Gamma, v)$ is switching equivalent to $F(\Gamma',v)$.
\end{lemma}

\begin{proof}
Let $\Gamma$ and $\Gamma'$ be two signed graphs over graph $G$, say $\Gamma = (G,\Sigma)$ and $\Gamma' = (G,\Sigma')$.
Let $v \in V(G)$ and let $F(\Gamma,v)$ and $F(\Gamma',v)$ be  the corresponding Fowler constructions.  Let $\Gamma$ and $\Gamma'$ be
switching equivalent. This means that there exists $S \subset V(G)$, such that $\Gamma' = \Gamma^S$. We know that 
$\Gamma^S = \Gamma^{V(G) \setminus S}$. Without loss of generality we may assume that $v \notin S$. Let the vertex labelling of $F(\Gamma,v)$,
$F(\Gamma',v)$ and $F(G,v)$ be the same as in Figure~\ref{Fig-FowlerExt}. In particular, this means that all vertices of $G$ belong also to $F(G,v)$. Since 
$v \notin S$ we have:
\begin{equation}
F(\Gamma',v) = F(\Gamma^S,v) = F(\Gamma,v)^S.
\end{equation}
Hence it follows:
\begin{equation}
\Gamma \sim \Gamma' \Rightarrow F(\Gamma,v) \sim F(\Gamma',v).
\end{equation}
To prove the converse assume the following:
$F(\Gamma,v) \sim F(\Gamma',v)$, where
$\Gamma = (G, \Sigma)$ and $\Gamma' = (G,\Sigma')$.
Let $S \subset V(F(G,v))$ such that $v \notin S$. Since all edges above $u_1,u_2, \ldots, u_s$ in Figure \ref{Fig-FowlerExt}(\subref{Fig-FowlerExtb}) are positive in both signed graphs, it is
clear that $S \subset V(G)$ and the result follows.
\end{proof}

\begin{theorem}
\label{thm:dathm}
Let $\Gamma$ be a signed graph and $v$ any one of its vertices.
Then the nullities of $\Gamma$ and $F(\Gamma, v)$ are equal, i.e.\ $\eta(\Gamma) = \eta(F(\Gamma, v))$.
Moreover, $\ker \Gamma$ admits a full eigenvector if and only if $\ker F(\Gamma, v)$ admits a full eigenvector.
\end{theorem}

\begin{proof}
Let $u_1,\ldots, u_{\rho}$ be the neighbours of vertex $v$ in $G$. Assume first that $G$ is a core graph and that $\x$ is an admissible eigenvector. Let $\x(w)$ denote the entry of $\x$ at vertex $w$. Let $a = \x(v)$ and let $b_i = \x(u_i)$.  We now produce a vertex labelling $\x'$ of $F(\Gamma,v)$ as above. It follows that if $\x$ is a valid assignment of $G$ then $\x'$ is a valid assignment on $F(\Gamma,v)$.
Thus $\eta( F(\Gamma, v)) \geq \eta(\Gamma)$.

On the other hand, apply Lemma~\ref{EquivLabeling} to $F(\Gamma,v)$ and an admissible assignment $\x'$.  First consider vertices $q_i$ and $q_j$ and their neighbourhoods.  Lemma~\ref{EquivLabeling} implies that $\x'(p_i) = \x'(p_j)$. Hence $\x'$ is constant on $p_i$, say $\x'(p_i) = a$.  Thus, it follows that $\x'(v) = -(\rho-1)a$. The second application of the lemma goes to vertices $v$ and $p_i$. It implies that for each $i$ the values $\x'(q_i)$ and $\x'(u_i)$ are equal, namely $\x'(q_i) = \x'(u_i)$.  Finally, let $\x(w) = \x'(w)$ for every $w \in V(G) \setminus \{v\}$ and let $\x(v) = a$. Hence, the existence of an admissible $\x'$ on $F(\Gamma,v)$ implies the existence of an admissible $\x$ on $\Gamma$.  Thus $\eta( F(\Gamma, v)) \leq \eta(\Gamma)$.
\end{proof}

\begin{corollary}
Let $\Gamma = (G, \sigma)$ be a signed graph and $v \in V(G)$ any one of its vertices.
The following statements hold:
\begin{enumerate}[label=(\arabic*)]
\item $F(\Gamma, v)$ is a signed nut graph if and only if $\Gamma$ is a signed nut graph.
\item $F(\Gamma, v)$ is a sign-balanced nut graph if and only if $\Gamma$ is a sign-balanced nut graph.
\end{enumerate}
\end{corollary}

\begin{proof}
Follows directly from Lemma~\ref{lem:lemmaxxx} and Theorem~\ref{thm:dathm}. Namely, if $\Gamma$ is sign-balanced, then it is switching equivalent to the all-positive signed nut graph 
$\Gamma' = (G,\emptyset)$. However, in this case $F(\Gamma',v) = F((G,\emptyset),v) = (F(G,v),\emptyset)$. Virtually the same argument
can be used in the opposite direction.
\end{proof}

The question answered by Table~\ref{tab:pinky} and the associated theorems was about finding graph orders where \emph{only} sign-unbalanced nut graphs
would exist. The above construction allows us to say something about the cases where both sign-balanced and sign-unbalanced nut graphs will exist.
A useful observation is the following.
\begin{observation}
\label{obs:theobs}
Whenever there is a parameter pair $(n, \rho)$ marked with a $\rC$ in Table~\ref{tab:pinky}, there is an infinite
series of values $n, n + 2\rho, n + 4\rho, n + 6\rho, \ldots$ at which a sign-unbalanced nut graph exists.
\end{observation}

\begin{theorem}
\label{thm:new2}
Let $U(\rho)$ be the set consisting of all integers $n$ for which a $\rho$-regular sign-unbalanced nut graph of order $n$ exists.
The following holds:
\begin{enumerate}
\item $U(1) = \emptyset$
\item $U(2) = \emptyset$
\item $U(3) = \{2k \mid k\geq 6\}$
\item $U(4) = \{5\} \cup \{k \mid k\geq 7\}$
\item $U(5) = \{2k \mid k\geq 5\}$ 
\item $U(6) = \{k \mid k\geq 8\}$
\item $U(7) = \{2k \mid k\geq 5\}$
\item $U(8) = \{ k \mid  k \geq 9 \}$ 
\item $U(9) = \{ 2k \mid  k \geq 6 \}$
\item $U(10) = \{ k \mid   k \geq 12 \}$
\item $U(11) = \{ 2k \mid  k \geq 7 \}$
\end{enumerate}
\end{theorem}

\begin{proof}
Cases $U(1) = \emptyset$ and $U(2) = \emptyset$ follow from the proof of Theorem \ref{thm:main}.
For the remaining cases Table~{\ref{tab:pinky}} suggests a conjecture for each row, which is proved by identifying at least one sign-unbalanced nut graph for key small values of $n$, and applying
the construction.

\noindent Case $U(3)$: A check for $n = 12$ yielded an example of a sign-unbalanced nut graph. Hence by use of Theorem~\ref{thm:main}
and Observation~\ref{obs:theobs}, sign-unbalanced nut graphs exist for all even $n \geq 12$.

\noindent Case $U(4)$: Checking that a sign-unbalanced nut graph exists for orders $8, 10, 12$ and $14$ gives 
a run of cases from which the result follows by use of the construction.

\noindent Case $U(5)$: From Table~\ref{tab:pinky}, $5$-regular sign-balanced nut graphs exist for all even $n \geq 10$. Checking the
cases $n = 10, 12, \ldots, 18$ shows that $5$-regular sign-unbalanced nut graphs also exist, and the result follows by use of
the construction.

\noindent Case $U(6)$: Checking cases $n = 12, 13, \ldots, 19$ is sufficient.

\noindent Case $U(7)$: Checking cases $n = 12, 14, \ldots, 22$ is sufficient.

\noindent Case $U(8)$: Checking cases $n = 12$ and $n = 14, 15, \ldots, 24$ is sufficient.

\noindent Case $U(9)$: Checking cases $n = 16, 18, \ldots, 28$ is sufficient.

\noindent Case $U(10)$: Checking cases $n = 15, 16, \ldots, 31$ is sufficient.

\noindent Case $U(11)$: Checking cases $n = 16, 18, \ldots, 34$ is sufficient.
\end{proof}

Note that $N_s(\rho) = U(\rho) \setminus N(\rho)$. Clearly, $N(\rho) \subseteq U(\rho)$ for all $1 \leq \rho \leq 11$.

\section{Conclusion}

Invocation of signed graphs as candidates 
for nut graphs allows extension of the orders at which a nut graph exists, and leads to a proof of all cases for regular nut graphs (either sign-balanced or sign-unbalanced) with degree at most $11$.
As with unweighted nut graphs, signed nut graphs can be generated by a generic construction in which the order of a 
smaller signed nut graph increases from $n$ to $n + 2\rho$, where $\rho$ is the degree of the vertex chosen as the focus of this vertex-expansion construction.
This was used to establish Theorem~\ref{thm:new2}, which specifies infinite series of orders $n$ for which both sign-balanced and sign-unbalanced nut graphs exists.
Two natural questions arise:

\begin{question}
Does there exist a degree $\rho$ and order $n$ such that there is at least one $\rho$-regular sign-balanced nut graph of order $n$, but
no sign-unbalanced nut graph with the same parameters? 
\end{question}

\noindent
Current data collected in Table~{\ref{tab:pinky}} are compatible with the existence of three patterns for $U(\rho)$ based on divisibility by $4$. 
These give
the basis for the following question for $\rho$ large enough:

\begin{question}
For all $\rho \geq 6$, prove or find a counterexample to
\begin{equation}
U(\rho) = \begin{cases}
\{ k \mid k \in \mathbb{Z}, k \geq \rho + 1 \} & \text{if } \rho \equiv 0 \pmod{4}, \\
\{ k \mid k \in \mathbb{Z}, k \geq \rho + 2 \} & \text{if } \rho \equiv 2 \pmod{4}, \\
\{ 2k \mid k \in \mathbb{Z}, 2k \geq \rho + 3 \} & \text{otherwise}.
\end{cases}
\label{quest:2}
\end{equation}
\end{question}

\noindent
Clearly, Equation \eqref{quest:2} would not hold for $\rho < 6$ (see Theorem 18).

\section*{Acknowledgements}
We thank Thomas Zaslavsky for helpful comments on an earlier draft of this paper.
The work of Toma\v{z} Pisanski is supported in part by the Slovenian Research Agency (research program P1-0294 
and research projects N1-0032, J1-9187, J1-1690, N1-0140 and J1-2481), and in part by H2020 Teaming InnoRenew CoE.
The work of Nino Bašić is supported in part by the Slovenian Research Agency (research program P1-0294
and research projects J1-9187, J1-1691, N1-0140 and J1-2481).
The work of Irene Sciriha and Patrick Fowler is supported by The University of Malta under the project
Graph Spectra, Computer Network Design and Electrical Conductivity in Nano-Structures MATRP01-20.

\nocite{*}

\section*{ORCID iDs}
Nino Ba{\v s}i{\'c} \includegraphics[scale=0.05]{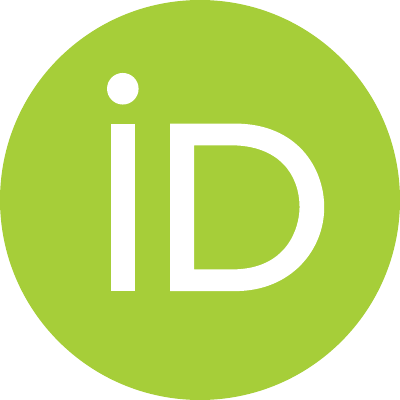} \href{https://orcid.org/0000-0002-6555-8668}{https://orcid.org/0000-0002-6555-8668}

\noindent
Patrick~W.~Fowler \includegraphics[scale=0.05]{ORCID_icon.pdf} \href{https://orcid.org/0000-0003-2106-1104}{https://orcid.org/0000-0003-2106-1104}

\noindent
Toma{\v z} Pisanski \includegraphics[scale=0.05]{ORCID_icon.pdf} \href{https://orcid.org/0000-0002-1257-5376}{https://orcid.org/0000-0002-1257-5376}

\noindent
Irene Sciriha \includegraphics[scale=0.05]{ORCID_icon.pdf} \href{https://orcid.org/0000-0002-5477-6803}{https://orcid.org/0000-0002-5477-6803}

\bibliographystyle{amcjoucc}
\bibliography{references}

\begin{thebibliography}{10}
\expandafter\ifx\csname urlstyle\endcsname\relax
  \providecommand{\doi}[1]{doi:\discretionary{}{}{}#1}\else
  \providecommand{\doi}{doi:\discretionary{}{}{}\begingroup
  \urlstyle{rm}\Url}\fi

\bibitem{BeCiKoWa2019}
F.~Belardo, S.~M. Cioab{\u a}, J.~Koolen and J.~Wang, Open problems in the
  spectral theory of signed graphs, \emph{Art Discrete Appl. Math.} \textbf{1}
  (2018), \#P2.10, \doi{10.26493/2590-9770.1286.d7b}.

\bibitem{Biggs}
N.~Biggs, \emph{Algebraic {G}raph {T}heory}, Cambridge Mathematical Library,
  Cambridge University Press, Cambridge, 2nd edition, 1993.

\bibitem{hog}
G.~Brinkmann, K.~Coolsaet, J.~Goedgebeur and H.~M{\'e}lot, House of {G}raphs:
  {A} database of interesting graphs, \emph{Discrete Appl. Math.} \textbf{161}
  (2013), 311--314, \doi{10.1016/j.dam.2012.07.018}, available at
  \url{https://hog.grinvin.org/}.

\bibitem{cubicpaper}
G.~Brinkmann, J.~Goedgebeur and B.~D. McKay, Generation of cubic graphs,
  \emph{Discrete Math. Theor. Comput. Sci.} \textbf{13} (2011), 69--80,
  \url{https://dmtcs.episciences.org/551}.

\bibitem{haemers}
A.~E. Brouwer and W.~H. Haemers, \emph{Spectra of {G}raphs}, Universitext,
  Springer, New York, 2012, \doi{10.1007/978-1-4614-1939-6}.

\bibitem{nutgen-site}
K.~Coolsaet, P.~W. Fowler and J.~Goedgebeur, Nut graphs, homepage of {N}utgen,
  \url{http://caagt.ugent.be/nutgen/}.

\bibitem{CoolFowlGoed-2017}
K.~Coolsaet, P.~W. Fowler and J.~Goedgebeur, Generation and properties of nut
  graphs, \emph{MATCH Commun. Math. Comput. Chem.} \textbf{80} (2018),
  423--444,
  \url{http://match.pmf.kg.ac.rs/electronic_versions/Match80/n2/match80n2_423-444.pdf}.

\bibitem{Fumitoshi2018}
F.~Ema, M.~Tanabe, S.~Saito, T.~Yoneda, K.~Sugisaki, T.~Tachikawa, S.~Akimoto,
  S.~Yamauchi, K.~Sato, A.~Osuka, T.~Takui and Y.~Kobori, Charge-transfer
  character drives {M}{\"o}bius antiaromaticity in the excited triplet state of
  twisted {$[28]$}hexaphyrin, \emph{J. Phys. Chem. Lett.} \textbf{9} (2018),
  2685--2690, \doi{10.1021/acs.jpclett.8b00740}.

\bibitem{Fowler02}
P.~W. Fowler, H{\"u}ckel spectra of {M}{\"o}bius {$\pi$} systems, \emph{Phys.
  Chem. Chem. Phys.} \textbf{4} (2002), 2878--2883, \doi{10.1039/b201850k}.

\bibitem{Jan}
P.~W. Fowler, J.~B. Gauci, J.~Goedgebeur, T.~Pisanski and I.~Sciriha, Existence
  of regular nut graphs for degree at most {$11$}, \emph{Discuss. Math. Graph
  Theory} \textbf{40} (2020), 533--557, \doi{10.7151/dmgt.2283}.

\bibitem{PWF_JCP_2014}
P.~W. Fowler, B.~T. Pickup, T.~Z. Todorova, M.~Borg and I.~Sciriha,
  Omni-conducting and omni-insulating molecules, \emph{J. Chem. Phys.}
  \textbf{140} (2014), 054115, \doi{10.1063/1.4863559}.

\bibitem{Basic2020}
P.~W. Fowler, T.~Pisanski and N.~Bašić, Charting the space of chemical nut
  graphs, \emph{MATCH Commun. Math. Comput. Chem.}  (2020), in press.

\bibitem{GPS}
J.~B. Gauci, T.~Pisanski and I.~Sciriha, Existence of regular nut graphs and
  the {F}owler construction, \emph{Appl. Anal. Discrete Math.}  (2020), in
  press.

\bibitem{GhHaMaMa2020}
E.~Ghorbani, W.~H. Haemers, H.~R. Maimani and L.~P. Majd, On sign-symmetric
  signed graphs, \emph{Ars Math. Contemp.} \textbf{19} (2020), 83--93,
  \doi{10.26493/1855-3974.2161.f55}.

\bibitem{GutmanSciriha-MaxSing}
I.~Gutman and I.~Sciriha, Graphs with maximum singularity, \emph{Graph Theory
  Notes N. Y.} \textbf{30} (1996), 17--20.

\bibitem{Heilbronner1964}
E.~Heilbronner, H{\"u}ckel molecular orbitals of {M}{\"o}bius-type
  conformations of annulenes, \emph{Tetrahedron Letters} \textbf{5} (1964),
  1923--1928, \doi{10.1016/s0040-4039(01)89474-0}.

\bibitem{Herges2006}
R.~Herges, Topology in chemistry: Designing {M}{\"o}bius molecules,
  \emph{{Chem. Rev.}} \textbf{{106}} ({2006}), {4820--4842},
  \doi{10.1021/cr0505425}.

\bibitem{HOLT2019}
D.~Holt and G.~Royle, A census of small transitive groups and vertex-transitive
  graphs, \emph{J. Symb. Comput.} \textbf{101} (2020), 51--60,
  \doi{10.1016/j.jsc.2019.06.006}.

\bibitem{Maybee1990}
J.~S. Maybee, D.~D. Olesky and P.~van~den Driessche, Partly zero eigenvectors,
  \emph{Linear and Multilinear Algebra} \textbf{28} (1990), 83--92,
  \doi{10.1080/03081089008818033}.

\bibitem{McDonald1993}
J.~J. McDonald, Partly zero eigenvectors and matrices that satisfy {$Au=b$},
  \emph{Linear and Multilinear Algebra} \textbf{33} (1993), 163--170,
  \doi{10.1080/03081089308818190}.

\bibitem{McKay201494}
B.~D. McKay and A.~Piperno, Practical graph isomorphism, {II}, \emph{J. Symb.
  Comput.} \textbf{60} (2014), 94--112, \doi{10.1016/j.jsc.2013.09.003}.

\bibitem{mckay1990transitive}
B.~D. McKay and G.~F. Royle, The transitive graphs with at most {$26$}
  vertices, \emph{Ars Combinatoria} \textbf{30} (1990), 161--176.

\bibitem{meringer_99}
M.~Meringer, Fast generation of regular graphs and construction of cages,
  \emph{J. Graph Theory} \textbf{30} (1999), 137--146,
  \doi{10.1002/(sici)1097-0118(199902)30:2<137::aid-jgt7>3.0.co;2-g}.

\bibitem{OEISgeneral}
{OEIS Foundation Inc.}, The {O}n-{L}ine {E}ncyclopedia of {I}nteger
  {S}equences, published electronically at \url{https://oeis.org}.

\bibitem{oeis}
{OEIS Foundation Inc.}, Sequence {A008483} in {T}he {O}n-{L}ine {E}ncyclopedia
  of {I}nteger {S}equences, published electronically at
  \url{https://oeis.org/A008483}.

\bibitem{Rzepa2005}
H.~S. Rzepa, M{\"o}bius aromaticity and delocalization, \emph{{Chem. Rev.}}
  \textbf{{105}} ({2005}), {3697--3715}, \doi{10.1021/cr030092l}.

\bibitem{Sc1998}
I.~Sciriha, On the construction of graphs of nullity one, \emph{Discrete Math.}
  \textbf{181} (1998), 193--211, \doi{10.1016/s0012-365x(97)00036-8}.

\bibitem{ScOnRnkGr99}
I.~Sciriha, On the rank of graphs, in: \emph{Combinatorics, {G}raph {T}heory,
  and {A}lgorithms, Volume {II}}, New Issues Press, Kalamazoo, MI, 1999 pp.
  769--778, {P}roceedings of the 8th {Q}uadrennial {I}nternational {C}onference
  on {G}raph {T}heory, {C}ombinatorics, {A}lgorithms, and {A}pplications,
  dedicated to the memory of {P}aul {E}rd{\H o}s, held at {W}estern {M}ichigan
  {U}niversity, {K}alamazoo, {MI}, {J}une 3 -- 7, 1996.

\bibitem{Sc2007}
I.~Sciriha, A characterization of singular graphs, \emph{Electron. J. Linear
  Algebra} \textbf{16} (2007), 451--462, \doi{10.13001/1081-3810.1215}.

\bibitem{Sc2008}
I.~Sciriha, Coalesced and embedded nut graphs in singular graphs, \emph{Ars
  Math. Contemp.} \textbf{1} (2008), 20--31, \doi{10.26493/1855-3974.20.7cc}.

\bibitem{SciCommEigDck09}
I.~Sciriha, Graphs with a common eigenvalue deck, \emph{Linear Algebra Appl.}
  \textbf{430} (2009), 78--85, \doi{10.1016/j.laa.2008.06.033}.

\bibitem{ScMaxCorSzSing09}
I.~Sciriha, Maximal core size in singular graphs, \emph{Ars Math. Contemp.}
  \textbf{2} (2009), 217--229, \doi{10.26493/1855-3974.115.891}.

\bibitem{feuerpaper}
I.~Sciriha and P.~W. Fowler, Nonbonding orbitals in fullerenes: {N}uts and
  cores in singular polyhedral graphs, \emph{J. Chem. Inf. Model.} \textbf{47}
  (2007), 1763--1775, \doi{10.1021/ci700097j}.

\bibitem{ScirihaGutman-NutExt}
I.~Sciriha and I.~Gutman, Nut graphs: {M}aximally extending cores, \emph{Util.
  Math.} \textbf{54} (1998), 257--272.

\bibitem{streitwieser1961}
A.~Streitwieser, \emph{Molecular {O}rbital {T}heory for {O}rganic {C}hemists},
  Wiley {I}nternational {E}dition, Wiley, 1961.

\bibitem{trinajstic1992}
N.~Trinajsti{\'c}, \emph{Chemical {G}raph {T}heory}, Mathematical Chemistry
  Series, CRC Press, Boca Raton, FL, 2nd edition, 1992,
  \doi{10.1007/s10910-008-9464-6}.

\bibitem{Yoneda2014}
T.~Yoneda, Y.~M. Sung, J.~M. Lim, D.~Kim and A.~Osuka, Pd{$^\text{II}$}
  complexes of [44]- and [46]decaphyrins: The largest {H}{\"u}ckel aromatic and
  antiaromatic, and {M}{\"o}bius aromatic macrocycles, \emph{Angew. Chem. Int.
  Ed.} \textbf{53} (2014), 13169--13173, \doi{10.1002/anie.201408506}.

\bibitem{Kim2009}
Z.~S. Yoon, A.~Osuka and D.~Kim, M{\"o}bius aromaticity and antiaromaticity in
  expanded porphyrins, \emph{{Nature Chemistry}} \textbf{{1}} ({2009}),
  {113--122}, \doi{10.1038/nchem.172}.

\bibitem{Zaslav}
T.~Zaslavsky, Signed graphs, \emph{Discrete Appl. Math.} \textbf{4} (1982),
  47--74, \doi{10.1016/0166-218x(82)90033-6}.

\bibitem{ZaslavDS9}
T.~Zaslavsky, Glossary of signed and gain graphs and allied areas,
  \emph{Electron. J. Combin.} \textbf{5} (1998), \#DS9 (41 pages),
  \doi{10.37236/31}.

\bibitem{ZaslavDS8}
T.~Zaslavsky, A mathematical bibliography of signed and gain graphs and allied
  areas, \emph{Electron. J. Combin.} \textbf{5} (1998), \#DS8 (124 pages),
  \doi{10.37236/29}.

\end{thebibliography}

\begin{appendices}

\newpage
\section{Small sign-unbalanced nut graphs}
\label{sec:appA}

For each parameter pair $(n, \rho)$ marked only with  \protect\rC\ in Table~{\ref{tab:pinky}}, an
example of a sign-unbalanced nut graph $\Gamma$ is given. The entry lists $n$ and $\rho$, followed by $|\Sigma|$ (the number of negative edges),
the $n\rho/2$ signed edges $E(G)$, where \rC\ in the superscript denotes a negative edge, and finally the kernel eigenvector $\bf x$ as
a list of integer entries. 

\small
\vspace{\baselineskip}
\noindent
$\boldsymbol{n = 14, \rho = 3, |\Sigma| = 2}$

\noindent
$
E(G) = \{ (0, 6), (0, 7), (0, 8), (1, 7), (1, 8), (1, 13), (2, 9), (2, 10), (2, 11), (3, 9), (3, 11)^{\rC}, (3, 12), (4, 10),  \\
(4, 12), (4, 13), (5, 11), (5, 12), (5, 13), (6, 9), (6, 10), (7, 8)^{\rC} \}
$

\noindent
$
{\bf x} = \begin{bmatrix} 1 & 1 & 3 & 1 & 1 & -2 & -4 & 2 & 2 & -2 & 1 & 1 & 3 & -4 \end{bmatrix}
$

\vspace{\baselineskip}
\noindent
$\boldsymbol{n = 16, \rho = 3, |\Sigma| = 3}$

\noindent
$
E(G) = \{ (0, 7), (0, 8), (0, 12), (1, 8)^{\rC}, (1, 9), (1, 10), (2, 9), (2, 10)^{\rC}, (2, 11), (3, 9), (3, 14), 
(3, 15)^{\rC},  (4, 10), \\ (4, 14), (4, 15), (5, 11), (5, 12), (5, 13), (6, 12), (6, 13), (6, 14), (7, 11), (7, 13), (8, 15) \}
$

\noindent
$
{\bf x} = \begin{bmatrix} 2 & 1 & 2 & -3 & 1 & -4 & 2 & 2 & -4 & -3 & -1 & 2 & 2 & -4 & 2 & -1 \end{bmatrix}
$

\vspace{\baselineskip}
\noindent
$\boldsymbol{n=  5, \rho = 4, |\Sigma| = 3}$ 

\noindent
$
E(G) = \{ (0, 1), (0, 2), (0, 3), (0, 4), (1, 2)^{\rC}, (1, 3), (1, 4), (2, 3)^{\rC}, (2, 4), (3, 4)^{\rC} \}
$

\noindent
$
{\bf x} = \begin{bmatrix} 1 & -1 & 1 & 1 & -1 \end{bmatrix}
$

\vspace{\baselineskip}
\noindent
$\boldsymbol{n = 7, \rho = 4, |\Sigma| = 1}$

\noindent
$
E(G) = \{ (0, 2), (0, 3), (0, 4), (0, 5), (1, 3)^{\rC}, (1, 4), (1, 5), (1, 6), (2, 4), (2, 5), (2, 6), (3, 5), (3, 6), (4, 6) \}
$

\noindent
$
{\bf x} = \begin{bmatrix} 1 & -1 & 1 & -1 & 1 & -1 & -1 \end{bmatrix}
$

\vspace{\baselineskip}
\noindent
$\boldsymbol{n = 9, \rho = 4, |\Sigma| = 3}$

\noindent
$
E(G) = \{ (0, 4), (0, 5), (0, 6), (0, 7), (1, 4), (1, 5)^{\rC}, (1, 6), (1, 7)^{\rC}, (2, 4), (2, 6), (2, 7)^{\rC}, (2, 8), (3, 5), (3, 6), \\ 
(3, 7), (3, 8), (4, 8), (5, 8) \}
$

\noindent
$
{\bf x} = \begin{bmatrix} 1 & -1 & 1 & -1 & -1 & 1 & 1 & -1 & -1 \end{bmatrix}
$

\vspace{\baselineskip}
\noindent
$\boldsymbol{n = 11, \rho = 4, |\Sigma| = 3}$

\noindent
$
E(G) = \{ (0, 5), (0, 6), (0, 7), (0, 8), (1, 5), (1, 8)^{\rC}, (1, 9), (1, 10), (2, 6), (2, 7), (2, 8), (2, 9), 
(3, 6)^{\rC}, (3, 7), \\ (3, 9), (3, 10), (4, 6), (4, 7)^{\rC}, (4, 9), (4, 10), (5, 8), (5, 10) \}
$

\noindent
$
{\bf x} = \begin{bmatrix} 1 & -1 & -1 & 1 & 1 & -1 & 1 & 1 & -1 & -1 & 1 \end{bmatrix}
$

\vspace{\baselineskip}
\noindent
$\boldsymbol{n = 13, \rho = 4, |\Sigma| = 5}$

\noindent
$
E(G) = \{ (0, 6), (0, 7), (0, 8), (0, 9), (1, 6), (1, 7), (1, 8)^{\rC}, (1, 9)^{\rC}, (2, 7), (2, 8), (2, 9), (2, 10)^{\rC}, (3, 7), (3, 10), \\
(3, 11)^{\rC}, (3, 12), (4, 8), (4, 10), (4, 11), (4, 12), (5, 9), (5, 10), (5, 11)^{\rC}, (5, 12), (6, 11), (6, 12) \}
$

\noindent
$
{\bf x} = \begin{bmatrix} 1 & -1 & -1 & 1 & -1 & -1 & 1 & -1 & 1 & -1 & -1 & -1 & 1 \end{bmatrix}
$

\vspace{\baselineskip}
\noindent
$\boldsymbol{n=  8, \rho = 6, |\Sigma| = 6}$

\noindent
$
E(G) = \{ (0, 1), (0, 2), (0, 3), (0, 4), (0, 5), (0, 6), (7, 1)^{\rC}, (7, 2)^{\rC}, (7, 3), (7, 4), (7, 5), (7, 6), (1, 2)^{\rC}, (1, 3)^{\rC}, \\
(1, 4), (1, 5), (2, 3), (2, 4), (2, 6)^{\rC}, (3, 5)^{\rC}, (3, 6), (4, 5), (4, 6), (5, 6) \}
$

\noindent
$
{\bf x} = \begin{bmatrix} 1 & 1 & -1 & 1 & -1 & -1 & 1 & -1 \end{bmatrix}
$

\vspace{\baselineskip}
\noindent
$\boldsymbol{n = 9, \rho = 6, |\Sigma| = 6}$

\noindent
$
E(G) = \{ (0, 3), (0, 4), (0, 5), (0, 6), (0, 7), (0, 8), (1, 3), (1, 4), (1, 5), (1, 6)^{\rC}, (1, 7), (1, 8)^{\rC}, (2, 3), (2, 4)^{\rC}, \\
(2, 5)^{\rC}, (2, 6), (2, 7), (2, 8)^{\rC}, (3, 5), (3, 6), (3, 7)^{\rC}, (4, 6), (4, 7), (4, 8), (5, 7), (5, 8), (6, 8) \}
$

\noindent
$
{\bf x} = \begin{bmatrix} 1 & -2 & 1 & -1 & -4 & 3 & -1 & 2 & 1 \end{bmatrix}
$

\pagebreak
\noindent
$\boldsymbol{n = 10, \rho = 6, |\Sigma| = 7}$

\noindent
$
E(G) = \{ (0, 4), (0, 5), (0, 6), (0, 7), (0, 8), (0, 9), (1, 4), (1, 5)^{\rC}, (1, 6)^{\rC}, (1, 7), (1, 8), (1, 9), (2, 4), (2, 5)^{\rC}, \\
(2, 6), (2, 7), (2, 8)^{\rC}, (2, 9), (3, 4), (3, 5), (3, 6)^{\rC}, (3, 7), (3, 8), (3, 9)^{\rC}, (4, 7), (4, 8), (5, 7), (5, 9)^{\rC}, (6, 8), \\
(6, 9) \}
$

\noindent
$
{\bf x} = \begin{bmatrix} 1 & 1 & -1 & -1 & -1 & 1 & -1 & 1 & -1 & 1 \end{bmatrix}
$

\vspace{\baselineskip}
\noindent
$\boldsymbol{n = 11, \rho = 6, |\Sigma| = 5}$

\noindent
$
E(G) = \{ (0, 4), (0, 5), (0, 6), (0, 7), (0, 8), (0, 9), (1, 5), (1, 6)^{\rC}, (1, 7)^{\rC}, (1, 8), (1, 9), (1, 10), (2, 5), (2, 6), \\
(2, 7)^{\rC}, (2, 8)^{\rC}, (2, 9), (2, 10), (3, 5)^{\rC}, (3, 6), (3, 7), (3, 8), (3, 9), (3, 10), (4, 6), (4, 7), (4, 8), (4, 9), (4, 10), \\
(5, 9), (5, 10), (6, 8), (7, 10) \}
$

\noindent
$
{\bf x} = \begin{bmatrix} 2 & -1 & -1 & 4 & -2 & -2 & -4 & 2 & -4 & 10 & -6 \end{bmatrix}
$

\vspace{\baselineskip}
\noindent
$\boldsymbol{n = 10, \rho = 7, |\Sigma| = 7}$

\noindent
$
E(G) = \{ (0, 3)^{\rC}, (0, 4)^{\rC}, (0, 5), (0, 6), (0, 7), (0, 8), (0, 9), (1, 3), (1, 4), (1, 5)^{\rC}, (1, 6), (1, 7), (1, 8), (1, 9), \\
(2, 3), (2, 4), (2, 5), (2, 6)^{\rC}, (2, 7), (2, 8), (2, 9), (3, 6), (3, 7)^{\rC}, (3, 8), (3, 9)^{\rC}, (4, 6), (4, 7), (4, 8), (4, 9), \\
(5, 6), (5, 7)^{\rC}, (5, 8), (5, 9), (6, 8), (7, 9) \}
$

\noindent
$
{\bf x} = \begin{bmatrix} 3 & 1 & 2 & 2 & -4 & -2 & -2 & 2 & 2 & -2 \end{bmatrix}
$

\vspace{\baselineskip}
\noindent
$\boldsymbol{n = 9, \rho = 8, |\Sigma| = 6}$

\noindent
$
E(G) = \{ (0, 1), (0, 2), (0, 3), (0, 4), (0, 5), (0, 6), (0, 7), (0, 8), (1, 2), (1, 3)^{\rC}, (1, 4), (1, 5), (1, 6)^{\rC}, (1, 7), \\
(1, 8), (2, 3), (2, 4), (2, 5), (2, 6), (2, 7), (2, 8)^{\rC}, (3, 4), (3, 5), (3, 6), (3, 7), (3, 8), (4, 5), (4, 6), (4, 7)^{\rC}, \\
(4, 8)^{\rC}, (5, 6)^{\rC}, (5, 7), (5, 8), (6, 7), (6, 8), (7, 8) \}
$

\noindent
$
{\bf x} = \begin{bmatrix} 1 & 1 & -1 & -1 & 1 & -1 & 1 & -1 & 1 \end{bmatrix}
$

\vspace{\baselineskip}
\noindent
$\boldsymbol{n=  10, \rho = 8, |\Sigma| = 6}$ 

\noindent
$
E(G) = \{ (0, 1), (0, 2), (0, 3), (0, 4), (0, 5), (0, 6), (0, 7), (0, 8), (9, 1)^{\rC}, (9, 2)^{\rC}, (9, 3), (9, 4), (9, 5), (9, 6), \\
(9, 7), (9, 8), (1, 2)^{\rC}, (1, 3)^{\rC}, (1, 4), (1, 5), (1, 6), (1, 7), (2, 3), (2, 4), (2, 5), (2, 6), (2, 8), (3, 4)^{\rC}, (3, 5), \\
(3, 7), (3, 8), (4, 6), (4, 7), (4, 8), (5, 6), (5, 7), (5, 8)^{\rC}, (6, 7), (6, 8), (7, 8) \}
$

\noindent
$
{\bf x} = \begin{bmatrix} 1 & 1 & -1 & 1 & -1 & -1 & -1 & 1 & 1 & -1 \end{bmatrix}
$

\vspace{\baselineskip}
\noindent
$\boldsymbol{n = 11, \rho = 8, |\Sigma| = 2}$

\noindent
$
E(G) = \{ (0, 2)^{\rC}, (0, 3), (0, 4)^{\rC}, (0, 5), (0, 6), (0, 7), (0, 8), (0, 9), (1, 3), (1, 4), (1, 5), (1, 6), (1, 7), (1, 8), \\
(1, 9), (1, 10), (2, 4), (2, 5), (2, 6), (2, 7), (2, 8), (2, 9), (2, 10), (3, 5), (3, 6), (3, 7), (3, 8), (3, 9), (3, 10), \\
(4, 6), (4, 7), (4, 8), (4, 9), (4, 10), (5, 7), (5, 8), (5, 9), (5, 10), (6, 8), (6, 9), (6, 10), (7, 9), (7, 10), (8, 10) \}
$

\noindent
$
{\bf x} = \begin{bmatrix} 1 & -3 & 3 & -1 & -1 & 1 & 1 & -1 & 1 & 1 & -1 \end{bmatrix}
$

\vspace{\baselineskip}
\noindent
$\boldsymbol{n = 13, \rho = 8, |\Sigma| = 4}$

\noindent
$
E(G) = \{ (0, 1), (0, 2), (0, 3), (0, 4), (0, 9), (0, 10), (0, 11), (0, 12), (1, 2), (1, 3), (1, 4), (1, 5), (1, 10), (1, 11), \\
(1, 12), (2, 3), (2, 4), (2, 5), (2, 6), (2, 11), (2, 12), (3, 4), (3, 5), (3, 6), (3, 7), (3, 12), (4, 5), (4, 6), (4, 7)^{\rC}, \\
(4, 8), (5, 6), (5, 7), (5, 8), (5, 9)^{\rC}, (6, 7), (6, 8), (6, 9), (6, 10), (7, 8), (7, 9), (7, 10), (7, 11), (8, 9), (8, 10), \\
(8, 11), (8, 12), (9, 10), (9, 11)^{\rC}, (9, 12), (10, 11)^{\rC}, (10, 12), (11, 12) \}
$

\noindent
$
{\bf x} = \begin{bmatrix} 3 & -11 & 7 & 5 & 1 & -7 & 7 & 1 & -3 & 7 & -11 & 3 & -1 \end{bmatrix}
$

\vspace{\baselineskip}
\noindent
$\boldsymbol{n=  12, \rho = 9, |\Sigma| = 7}$ 

\noindent
$
E(G) = \{ (0, 3)^{\rC}, (0, 4)^{\rC}, (0, 5)^{\rC}, (0, 6), (0, 7), (0, 8), (0, 9), (0, 10), (0, 11), (1, 3)^{\rC}, (1, 4), (1, 5), (1, 6), \\
(1, 7), (1, 8)^{\rC}, (1, 9), (1, 10), (1, 11), (2, 3), (2, 4), (2, 5), (2, 6), (2, 7), (2, 8), (2, 9), (2, 10), (2, 11), (3, 5), \\
(3, 6)^{\rC}, (3, 8), (3, 9), (3, 10), (3, 11), (4, 6), (4, 7), (4, 8), (4, 9), (4, 10), (4, 11), (5, 7)^{\rC}, (5, 8), (5, 9), (5, 10), \\
(5, 11), (6, 8), (6, 9), (6, 10), (6, 11), (7, 8), (7, 9), (7, 10), (7, 11), (8, 10), (9, 11) \} 
$

\noindent
$
{\bf x} = \begin{bmatrix} 2 & 5 & -3 & 2 & -4 & 2 & -6 & 4 & -2 & -2 & 8 & -2 \end{bmatrix}
$

\pagebreak
\noindent
$\boldsymbol{n = 14, \rho = 9, |\Sigma| = 14}$

\noindent
$
E(G) = \{ (0, 1), (0, 2), (0, 3), (0, 4), (0, 5), (0, 6), (0, 7), (0, 10)^{\rC}, (0, 11)^{\rC}, (1, 2), (1, 3), (1, 4), (1, 5), (1, 6), \\
(1, 7), (1, 8)^{\rC}, (1, 13)^{\rC}, (2, 3), (2, 4), (2, 5), (2, 6), (2, 8), (2, 10)^{\rC}, (2, 13)^{\rC}, (3, 4), (3, 5), (3, 6)^{\rC}, (3, 9), \\
(3, 10), (3, 11)^{\rC}, (4, 7), (4, 8)^{\rC}, (4, 9), (4, 10), (4, 12)^{\rC}, (5, 7)^{\rC}, (5, 8), (5, 9), (5, 11), (5, 12)^{\rC}, (6, 8), (6, 9)^{\rC}, \\
(6, 11), (6, 12), (6, 13), (7, 8), (7, 9)^{\rC}, (7, 10), (7, 12), (7, 13), (8, 11), (8, 12), (8, 13), (9, 10), (9, 11), (9, 12), \\
(9, 13), (10, 11), (10, 12), (10, 13), (11, 12), (11, 13), (12, 13) \}
$

\noindent
$
{\bf x} = \begin{bmatrix} 4 & -5 & -4 & -4 & 6 & 4 & 1 & 2 & 1 & -2 & -1 & 1 & -11 & 8 \end{bmatrix}
$

\vspace{\baselineskip}
\noindent
$\boldsymbol{n=  12, \rho = 10, |\Sigma| = 6}$ 

\noindent
$
E(G) = \{ (0, 1), (0, 2), (0, 3), (0, 4), (0, 5), (0, 6), (0, 7), (0, 8), (0, 9), (0, 10), (11, 1)^{\rC}, (11, 2)^{\rC}, (11, 3), \\
(11, 4), (11, 5), (11, 6), (11, 7), (11, 8), (11, 9), (11, 10), (1, 2), (1, 3)^{\rC}, (1, 4), (1, 5), (1, 6), (1, 7), (1, 8), \\
(1, 9), (2, 3), (2, 4)^{\rC}, (2, 5), (2, 6), (2, 7), (2, 8), (2, 10), (3, 4), (3, 5)^{\rC}, (3, 6), (3, 7), (3, 9), (3, 10), (4, 5), \\
(4, 6), (4, 8), (4, 9), (4, 10), (5, 7), (5, 8), (5, 9), (5, 10), (6, 7)^{\rC}, (6, 8), (6, 9), (6, 10), (7, 8), (7, 9), (7, 10), \\
(8, 9), (8, 10), (9, 10) \}
$

\noindent
$
{\bf x} = \begin{bmatrix} 1 & 1 & -1 & 1 & 1 & -1 & -1 & 1 & -1 & 1 & -1 & -1 \end{bmatrix}
$

\vspace{\baselineskip}
\noindent
$\boldsymbol{n = 13, \rho = 10, |\Sigma| = 7}$

\noindent
$
E(G) = \{ (0, 3)^{\rC}, (0, 4)^{\rC}, (0, 5)^{\rC}, (0, 6), (0, 7), (0, 8), (0, 9), (0, 10), (0, 11), (0, 12), (1, 3)^{\rC}, (1, 4), (1, 5), \\
(1, 6), (1, 7)^{\rC}, (1, 8), (1, 9), (1, 10), (1, 11), (1, 12), (2, 3), (2, 4)^{\rC}, (2, 5), (2, 6), (2, 7), (2, 8), (2, 9), (2, 10), \\
(2, 11), (2, 12), (3, 5), (3, 7), (3, 8), (3, 9), (3, 10), (3, 11), (3, 12), (4, 6), (4, 7), (4, 8), (4, 9), (4, 10), (4, 11), \\
(4, 12), (5, 7), (5, 8), (5, 9), (5, 10), (5, 11), (5, 12), (6, 7), (6, 8), (6, 9), (6, 10), (6, 11), (6, 12), (7, 9), (7, 10), \\
(7, 11), (8, 10), (8, 11), (8, 12), (9, 11)^{\rC}, (9, 12), (10, 12) \}
$

\noindent
$
{\bf x} = \begin{bmatrix} 1 & -4 & -3 & 4 & 2 & -4 & -2 & -2 & 4 & -4 & 6 & -4 & 4 \end{bmatrix}
$

\vspace{\baselineskip}
\noindent
$\boldsymbol{n=  14, \rho = 10, |\Sigma| = 7}$

\noindent
$
E(G) = \{  (0, 4)^{\rC}, (0, 5)^{\rC}, (0, 6)^{\rC}, (0, 7), (0, 8), (0, 9), (0, 10), (0, 11), (0, 12), (0, 13), (1, 4)^{\rC}, (1, 5), (1, 6), \\
(1, 7), (1, 8), (1, 9), (1, 10), (1, 11), (1, 12), (1, 13), (2, 4), (2, 5)^{\rC}, (2, 6), (2, 7), (2, 8), (2, 9), (2, 10), (2, 11), \\
(2, 12), (2, 13), (3, 4), (3, 5), (3, 6), (3, 7), (3, 8)^{\rC}, (3, 9), (3, 10), (3, 11), (3, 12), (3, 13), (4, 8), (4, 9), (4, 10), \\
(4, 11), (4, 12), (4, 13), (5, 8), (5, 9), (5, 10), (5, 11), (5, 12), (5, 13), (6, 8)^{\rC}, (6, 9), (6, 10), (6, 11), (6, 12), \\
(6, 13), (7, 8), (7, 9), (7, 10), (7, 11), (7, 12), (7, 13), (8, 11), (8, 12), (9, 11), (9, 13), (10, 12), (10, 13)  \}
$

\noindent
$
{\bf x} = \begin{bmatrix} 1 & -1 & -1 & 3 & -1 & -1 & 1 & 1 & -1 & -1 & -1 & 3 & 3 & -5 \end{bmatrix}
$

\vspace{\baselineskip}
\noindent
$\boldsymbol{n=  14, \rho = 11, |\Sigma| = 7}$ 

\noindent
$
E(G) = \{ (0, 3)^{\rC}, (0, 4)^{\rC}, (0, 5)^{\rC}, (0, 6), (0, 7), (0, 8), (0, 9), (0, 10), (0, 11), (0, 12), (0, 13), (1, 3)^{\rC}, (1, 4), \\
(1, 5), (1, 6), (1, 7)^{\rC}, (1, 8), (1, 9), (1, 10), (1, 11), (1, 12), (1, 13), (2, 3), (2, 4), (2, 5), (2, 6), (2, 7), (2, 8), \\
(2, 9), (2, 10), (2, 11), (2, 12), (2, 13), (3, 5), (3, 6)^{\rC}, (3, 7), (3, 9), (3, 10), (3, 11), (3, 12), (3, 13), (4, 6), \\
(4, 7), (4, 8), (4, 9), (4, 10), (4, 11), (4, 12), (4, 13), (5, 7), (5, 8), (5, 9)^{\rC}, (5, 10), (5, 11), (5, 12), (5, 13), \\ 
(6, 8), (6, 9), (6, 10), (6, 11), (6, 12), (6, 13), (7, 9), (7, 10), (7, 11), (7, 12), (7, 13), (8, 9), (8, 10), (8, 11),\\
 (8, 12), (8, 13), (9, 11), (9, 12), (10, 12), (10, 13), (11, 13) \}
$

\noindent
$
{\bf x} = \begin{bmatrix} 1 & 2 & -1 & 2 & -4 & 2 & -2 & -2 & 2 & 2 & -2 & 2 & 2 & -2 \end{bmatrix}
$

\normalsize
\newpage
\section{Sign-unbalanced nut graphs for low values of $\boldsymbol{\rho}$}
\label{sec:appB}

For parameter pairs $(n, \rho)$ in Table~{\ref{tab:pinky}} marked by  \protect\kC \rC\ 
to indicate that both a sign-balanced and a sign-unbalanced nut graphs exist, an example of a sign-unbalanced nut graph 
$\Gamma$ is given. These are needed for the proof of Theorem~{\ref{thm:new2}}.
Notation as in Appendix~\ref{sec:appA}.

\small
\vspace{\baselineskip}
\noindent
$\boldsymbol{n=  12, \rho = 3, |\Sigma| = 5}$ 

\noindent
$
E(G) = \{ (0, 5), (0, 7), (0, 8), (1, 6), (1, 7)^{\rC}, (1, 11), (2, 6)^{\rC}, (2, 8), (2, 10)^{\rC}, (3, 7), (3, 9), (3, 10)^{\rC}, (4, 8), \\
(4, 9), (4, 11)^{\rC}, (5, 10), (5, 11), (6, 9) \}
$

\noindent
$
{\bf x} = \begin{bmatrix} 2 & 1 & -1 & -1 & -1 & -2 & 2 & 1 & 1 & -2 & -1 & -1 \end{bmatrix}
$

\vspace{\baselineskip}
\noindent
$\boldsymbol{n = 8, \rho = 4, |\Sigma| = 2}$

\noindent
$
E(G) = \{ (0, 3), (0, 4), (0, 5), (0, 6), (1, 3), (1, 5), (1, 6), (1, 7), (2, 4), (2, 5)^{\rC}, (2, 6), (2, 7), (3, 4), (3, 6), \\
(4, 7)^{\rC}, (5, 7) \}
$

\noindent
$
{\bf x} = \begin{bmatrix} 1 & -1 & 1 & -1 & 1 & 1 & -1 & 1 \end{bmatrix}
$

\vspace{\baselineskip}
\noindent
$\boldsymbol{n = 10, \rho = 4, |\Sigma| = 2}$

\noindent
$
E(G) = \{ (0, 4), (0, 5), (0, 6), (0, 7), (1, 5), (1, 6), (1, 8), (1, 9), (2, 5), (2, 7), (2, 8)^{\rC}, (2, 9), (3, 6), (3, 7), \\
(3, 8), (3, 9), (4, 6), (4, 7)^{\rC}, (4, 8), (5, 9) \}
$

\noindent
$
{\bf x} = \begin{bmatrix} 1 & 1 & -1 & -1 & -1 & 1 & -1 & 1 & 1 & -1 \end{bmatrix}
$

\vspace{\baselineskip}
\noindent
$\boldsymbol{n = 12, \rho = 4, |\Sigma| = 2}$

\noindent
$
E(G) = \{ (0, 5), (0, 6), (0, 7), (0, 8), (1, 6), (1, 7), (1, 8), (1, 9), (2, 6), (2, 7)^{\rC}, (2, 8), (2, 11), (3, 7), (3, 9), \\
(3, 10), (3, 11), (4, 8), (4, 9)^{\rC}, (4, 10), (4, 11), (5, 9), (5, 10), (5, 11), (6, 10) \}
$

\noindent
$
{\bf x} = \begin{bmatrix} 1 & -1 & 1 & 1 & -1 & -1 & 1 & 1 & -1 & -1 & -1 & 1 \end{bmatrix}
$

\vspace{\baselineskip}
\noindent
$\boldsymbol{n = 14, \rho = 4, |\Sigma| = 1}$

\noindent
$
E(G) = \{ (0, 6), (0, 7), (0, 9), (0, 11), (1, 6), (1, 8), (1, 12), (1, 13), (2, 7), (2, 8), (2, 9), (2, 10), (3, 7), (3, 9),\\
 (3, 10), (3, 12), (4, 8), (4, 11), (4, 12), (4, 13), (5, 10), (5, 11), (5, 12), (5, 13)^{\rC}, (6, 10), (6, 11), (7, 9), (8, 13) \}
$

\noindent
$
{\bf x} = \begin{bmatrix} 1 & 4 & 2 & -4 & -3 & 3 & -1 & 1 & 2 & 1 & -4 & -1 & 2 & -3 \end{bmatrix}
$

\vspace{\baselineskip}
\noindent
$\boldsymbol{n = 10, \rho = 5, |\Sigma| = 6}$

\noindent
$
E(G) = \{ (0, 4), (0, 5), (0, 6), (0, 7), (0, 8), (1, 4), (1, 6)^{\rC}, (1, 7)^{\rC}, (1, 8)^{\rC}, (1, 9), (2, 5), (2, 6), (2, 7)^{\rC}, (2, 8), \\
(2, 9), (3, 5), (3, 6)^{\rC}, (3, 7), (3, 8), (3, 9), (4, 5), (4, 7), (4, 8)^{\rC}, (5, 9), (6, 9) \}
$

\noindent
$
{\bf x} = \begin{bmatrix} 2 & 1 & 1 & -1 & 1 & 1 & -2 & -2 & 2 & -3 \end{bmatrix}
$

\vspace{\baselineskip}
\noindent
$\boldsymbol{n = 12, \rho = 5, |\Sigma| = 7}$

\noindent
$
E(G) = \{ (0, 5), (0, 6), (0, 7), (0, 10), (0, 11), (1, 6), (1, 7), (1, 8), (1, 9), (1, 10)^{\rC}, (2, 6), (2, 7), (2, 8), (2, 9), \\
(2, 11), (3, 6), (3, 8), (3, 9)^{\rC}, (3, 10)^{\rC}, (3, 11), (4, 7), (4, 8)^{\rC}, (4, 9), (4, 10)^{\rC}, (4, 11), (5, 8), (5, 9), (5, 10),\\
 (5, 11)^{\rC}, (6, 7)^{\rC} \}
$

\noindent
$
{\bf x} = \begin{bmatrix} 2 & 6 & 2 & -6 & -6 & -8 & 4 & 4 & -3 & -3 & 2 & -2 \end{bmatrix}
$

\vspace{\baselineskip}
\noindent
$\boldsymbol{n=  14, \rho = 5, |\Sigma| = 7}$

\noindent
$
E(G) = \{(0, 1), (0, 2), (0, 3), (0, 4), (0, 5)^{\rC}, (1, 2), (1, 3), (1, 4), (1, 12)^{\rC}, (2, 3)^{\rC}, (2, 5), (2, 6), (3, 5), (3, 7), \\
(4, 7)^{\rC}, (4, 8), (4, 9), (5, 8), (5, 10), (6, 9), (6, 10), (6, 11), (6, 13)^{\rC}, (7, 9), (7, 11), (7, 12), (8, 9)^{\rC}, (8, 10), \\
(8, 13), (9, 13), (10, 11)^{\rC}, (10, 12), (11, 12), (11, 13), (12, 13) \}
$

\noindent
$
{\bf x} = \begin{bmatrix}    2 & 1  & 6  & -2  & -4 & 1 &  -6 &   2  & -4 & 3  &  2 &-7  &  2&   4 \end{bmatrix} 
$

\pagebreak
\noindent
$\boldsymbol{n = 16, \rho = 5, |\Sigma| = 3}$

\noindent
$
E(G) = \{ (0, 7), (0, 8), (0, 9)^{\rC}, (0, 14), (0, 15), (1, 8), (1, 9), (1, 10), (1, 11), (1, 12), (2, 8), (2, 9), (2, 10),\\
(2, 12), (2, 13), (3, 8), (3, 9), (3, 11), (3, 12), (3, 13), (4, 9), (4, 10), (4, 11), (4, 14), (4, 15), (5, 10)^{\rC}, (5, 12),\\
(5, 13), (5, 14), (5, 15), (6, 11), (6, 12), (6, 13), (6, 14)^{\rC}, (6, 15), (7, 10), (7, 11), (7, 13), (7, 15), (8, 14) \}
$

\noindent
$
{\bf x} = \begin{bmatrix} 1 & -2 & 1 & 1 & 1 & -2 & 2 & -2 & 2 & 1 & -1 & -1 & -1 & -1 & -1 & 2 \end{bmatrix}
$

\vspace{\baselineskip}
\noindent
$\boldsymbol{n=  18, \rho = 5, |\Sigma| = 18}$

\noindent
$
E(G) = \{(0, 1), (0, 2), (0, 3), (0, 5)^{\rC}, (0, 7)^{\rC}, (1, 2), (1, 3), (1, 7)^{\rC}, (1, 15)^{\rC}, (2, 4), (2, 10)^{\rC}, (2, 13)^{\rC}, (3, 5), \\
(3, 8)^{\rC}, (3, 17)^{\rC}, (4, 6), (4, 7), (4, 8)^{\rC}, (4, 10)^{\rC}, (5, 8), (5, 9), (5, 11)^{\rC}, (6, 8), (6, 10), (6, 16)^{\rC}, (6, 17)^{\rC}, \\
(7, 11), (7, 12), (8, 11), (9, 12)^{\rC}, (9, 13), (9, 14), (9, 16)^{\rC}, (10, 11), (10, 15),  (11, 14)^{\rC}, (12, 15)^{\rC}, (12, 16), \\
(12, 17), (13, 14)^{\rC}, (13, 15), (13, 16), (14, 15), (14, 17), (16, 17) \}
$

\noindent
$
{\bf x} = \begin{bmatrix} 1 &  5 & -9  &-5  & 3 & -7  & 8  &-2 & -6 &  9 &  3 & -3  & 6 &  6  & 2& -11&  -5&   5 \end{bmatrix} 
$

\vspace{\baselineskip}
\noindent
$\boldsymbol{n = 12, \rho = 6, |\Sigma| = 12}$

\noindent
$
E(G) = \{ (0, 1), (0, 2), (0, 3), (0, 4), (0, 7)^{\rC}, (0, 10)^{\rC}, (1, 2), (1, 3), (1, 4), (1, 6)^{\rC}, (1, 11)^{\rC}, (2, 3)^{\rC}, (2, 5), \\
(2, 6), (2, 7)^{\rC}, (3, 5), (3, 7), (3, 9)^{\rC}, (4, 5), (4, 6)^{\rC}, (4, 8), (4, 10)^{\rC}, (5, 8)^{\rC}, (5, 9), (5, 11)^{\rC}, (6, 7), (6, 8), \\
(6, 10), (7, 9), (7, 11), (8, 9)^{\rC}, (8, 10), (8, 11), (9, 10), (9, 11), (10, 11) \}
$

\noindent
$
{\bf x} = \begin{bmatrix} 5 & -1 & 1 & 1 & -3 & -3 & 1 & 1 & -3 & 1 & -3 & 3 \end{bmatrix}
$

\vspace{\baselineskip}
\noindent
$\boldsymbol{n=  13, \rho = 6, |\Sigma| = 8}$ 

\noindent
$
E(G) = \{
(0, 6), (0, 7), (0, 8), (0, 9), (0, 10), (0, 11), (1, 6), (1, 7), (1, 9), (1, 10)^{\rC}, (1, 11), (1, 12)^{\rC}, (2, 6), \\
(2, 8), (2, 9)^{\rC}, (2, 10), (2, 11), (2, 12)^{\rC}, (3, 7), (3, 8)^{\rC}, (3, 9), (3, 10), (3, 11), (3, 12), (4, 7)^{\rC}, (4, 8), (4, 9), \\
(4, 10), (4, 11), (4, 12), (5, 7), (5, 8), (5, 9)^{\rC}, (5, 10)^{\rC}, (5, 11), (5, 12), (6, 7), (6, 8), (6, 12) \}
$

\noindent
$
{\bf x} = \begin{bmatrix} 8 & -4 & -4 & -2 & -2 & 4 & -8 & 2 & 2 & 1 & 1 & 2 & -4 \end{bmatrix}
$

\vspace{\baselineskip}
\noindent
$\boldsymbol{n = 14, \rho = 6, |\Sigma| = 14}$

\noindent
$
E(G) = \{ (0, 1), (0, 2), (0, 3), (0, 4), (0, 8)^{\rC}, (0, 12)^{\rC}, (1, 2), (1, 3), (1, 4)^{\rC}, (1, 5), (1, 6)^{\rC}, (2, 4), (2, 6), \\
(2, 7)^{\rC}, (2, 13)^{\rC}, (3, 7), (3, 8), (3, 9)^{\rC}, (3, 11)^{\rC}, (4, 7), (4, 8), (4, 10)^{\rC}, (5, 6)^{\rC}, (5, 7), (5, 9), (5, 10), (5, 12)^{\rC}, \\
(6, 9), (6, 11), (6, 12), (7, 8)^{\rC}, (7, 10), (8, 11), (8, 13), (9, 11), (9, 12), (9, 13)^{\rC}, (10, 11)^{\rC}, (10, 12), (10, 13), \\
(11, 13), (12, 13) \}
$

\noindent
$
{\bf x} = \begin{bmatrix} 3 & -4 & 7 & -11 & 6 & 3 & -4 & -5 & -1 & -4 & 8 & -3 & -1 & 6 \end{bmatrix}
$

\vspace{\baselineskip}
\noindent
$\boldsymbol{n = 15, \rho = 6, |\Sigma| = 1}$

\noindent
$
E(G) = \{ (0, 5), (0, 6), (0, 9), (0, 11), (0, 12), (0, 13), (1, 6), (1, 7), (1, 8), (1, 9), (1, 10), (1, 11), (2, 7), (2, 8),\\
(2, 10), (2, 11)^{\rC}, (2, 12), (2, 14), (3, 7), (3, 8), (3, 11), (3, 12), (3, 13), (3, 14), (4, 9), (4, 10), (4, 11), (4, 12),\\
(4, 13), (4, 14), (5, 8), (5, 9), (5, 10), (5, 13), (5, 14), (6, 9), (6, 10), (6, 11), (6, 12), (7, 12), (7, 13), (7, 14),\\
(8, 13), (8, 14), (9, 10) \}
$

\noindent
$
{\bf x} = \begin{bmatrix}  1 & -2 & 1 & 1 & -1 & 1 & 2 & -4 & 2 & -1 & -1 & 2 & 1 & -5 & 4 \end{bmatrix}
$

\vspace{\baselineskip}
\noindent
$\boldsymbol{n = 16, \rho = 6, |\Sigma| = 8}$

\noindent
$
E(G) = \{ (0, 1), (0, 2), (0, 3), (0, 4), (0, 5), (0, 6)^{\rC}, (1, 2), (1, 3), (1, 4), (1, 5), (1, 15)^{\rC}, (2, 3), (2, 6), (2, 7), \\
(2, 12)^{\rC}, (3, 6), (3, 7)^{\rC}, (3, 8), (4, 7), (4, 9), (4, 10), (4, 11)^{\rC}, (5, 9), (5, 10)^{\rC}, (5, 11), (5, 12), (6, 11), (6, 13), \\
(6, 14), (7, 11), (7, 12), (7, 13), (8, 9)^{\rC}, (8, 10), (8, 11), (8, 12), (8, 14), (9, 13), (9, 14), (9, 15), (10, 11), \\
(10, 13), (10, 15), (12, 14), (12, 15), (13, 14)^{\rC}, (13, 15), (14, 15) \}
$

\noindent
$
{\bf x} = \begin{bmatrix} 15 & -13 & 19 & 7 & -17 & -5 & -9 & 11 & -1 & -13 & -13 & -13 & 11 & 7 & -5 & 19 \end{bmatrix}
$

\pagebreak
\noindent
$\boldsymbol{n = 17, \rho = 6, |\Sigma| = 1}$

\noindent
$
E(G) = \{ (0, 6), (0, 9), (0, 10), (0, 11), (0, 12), (0, 13)^{\rC}, (1, 7), (1, 8), (1, 10), (1, 12), (1, 15), (1, 16), (2, 7), \\
(2, 10), (2, 11), (2, 12), (2, 13), (2, 14), (3, 8), (3, 9), (3, 11), (3, 12), (3, 14), (3, 16), (4, 8), (4, 11), (4, 13),\\
 (4, 14), (4, 15), (4, 16), (5, 9), (5, 10), (5, 13), (5, 14), (5, 15), (5, 16), (6, 9), (6, 11), (6, 13), (6, 15), (6, 16), \\
(7, 8), (7, 10), (7, 14), (7, 15), (8, 13), (8, 14), (9, 12), (9, 16), (10, 12), (11, 15) \}
$

\noindent
$
{\bf x} = \begin{bmatrix} 2 & -1 & -1 & -1 & -3 & -1 & 4 & 3 & 3 & 2 & -1 & -2 & -2 & 1 & 1 & -1 & -2 \end{bmatrix}
$

\vspace{\baselineskip}
\noindent
$\boldsymbol{n = 18, \rho = 6, |\Sigma| = 1}$

\noindent
$
E(G) = \{ (0, 8), (0, 9), (0, 10), (0, 14), (0, 15), (0, 16), (1, 8), (1, 9), (1, 10), (1, 15), (1, 16), (1, 17), (2, 8), \\
(2, 9), (2, 11), (2, 12), (2, 15), (2, 17), (3, 9)^{\rC}, (3, 11), (3, 13), (3, 14), (3, 16), (3, 17), (4, 10), (4, 11), (4, 12),\\
 (4, 13), (4, 14), (4, 15), (5, 10), (5, 12), (5, 13), (5, 15), (5, 16), (5, 17), (6, 10), (6, 13), (6, 14), (6, 15), \\
(6, 16), (6, 17), (7, 11), (7, 12), (7, 13), (7, 14), (7, 16), (7, 17), (8, 11), (8, 12), (8, 13), (9, 11), (9, 12),\\
 (10, 14) \}
$

\noindent
$
{\bf x} = \begin{bmatrix} 1 & -1 & 1 & 1 & -1 & 1 & -1 & -1 & 1 & -1 & 1 & -1 & 1 & -1 & 1 & -1 & -1 & 1 \end{bmatrix}
$

\vspace{\baselineskip}
\noindent
$\boldsymbol{n = 19, \rho = 6, |\Sigma| = 2}$

\noindent
$
E(G) = \{ (0, 8), (0, 10), (0, 11), (0, 12), (0, 13), (0, 14), (1, 8), (1, 10), (1, 11), (1, 13), (1, 14)^{\rC}, (1, 17),\\
 (2, 9), (2, 11), (2, 15), (2, 16), (2, 17), (2, 18), (3, 9), (3, 12), (3, 14), (3, 16), (3, 17), (3, 18), (4, 10), (4, 12), \\
(4, 13), (4, 14), (4, 15), (4, 16), (5, 10), (5, 12), (5, 13), (5, 15), (5, 17), (5, 18), (6, 10), (6, 12), (6, 14), \\
(6, 15), (6, 16), (6, 18), (7, 11), (7, 12), (7, 13), (7, 15), (7, 17), (7, 18), (8, 11), (8, 15), (8, 16), (8, 17)^{\rC}, \\
(9, 13), (9, 14), (9, 16), (9, 18), (10, 11) \}
$

\noindent
$
{\bf x} = \begin{bmatrix} 5 & -3 & 3 & 3 & -3 & 10 & -7 & -8 & 5 & -1 & -2 & -2 & 1 & -3 & 1 & 4 & -1 & 3 & -3 \end{bmatrix}
$

\vspace{\baselineskip}
\noindent
$\boldsymbol{n = 12, \rho = 7, |\Sigma| = 2}$

\noindent
$
E(G) = \{ (0, 3), (0, 4), (0, 6), (0, 7), (0, 8), (0, 9), (0, 11), (1, 3), (1, 5), (1, 6), (1, 7), (1, 8), (1, 10), (1, 11), \\
(2, 4), (2, 5), (2, 6)^{\rC}, (2, 7), (2, 9), (2, 10), (2, 11), (3, 6), (3, 7), (3, 8), (3, 9), (3, 10)^{\rC}, (4, 6), (4, 8), (4, 9),\\
 (4, 10), (4, 11), (5, 7), (5, 8), (5, 9), (5, 10), (5, 11), (6, 9), (6, 11), (7, 9), (7, 10), (8, 10), (8, 11) \}
$

\noindent
$
{\bf x} = \begin{bmatrix} 2 & -1 & -1 & 1 & 1 & -2 & -2 & 1 & 3 & -1 & 2 & -3 \end{bmatrix}
$

\vspace{\baselineskip}
\noindent
$\boldsymbol{n = 14, \rho = 7, |\Sigma| = 1}$

\noindent
$
E(G) = \{ (0, 5), (0, 6), (0, 7), (0, 8), (0, 10), (0, 11), (0, 12), (1, 6), (1, 7), (1, 8), (1, 9), (1, 10), (1, 11), \\
(1, 13), (2, 6), (2, 7), (2, 9), (2, 10), (2, 11)^{\rC}, (2, 12), (2, 13), (3, 6), (3, 8), (3, 9), (3, 10), (3, 11), (3, 12), \\
(3, 13), (4, 7), (4, 8), (4, 9), (4, 10), (4, 11), (4, 12), (4, 13), (5, 7), (5, 8), (5, 9), (5, 10), (5, 11), (5, 12), \\
(6, 11), (6, 12), (6, 13), (7, 9), (7, 13), (8, 12), (8, 13), (9, 10) \}
$

\noindent
$
{\bf x} = \begin{bmatrix} 1 & 7 & -1 & 1 & -10 & 3 & -3 & -3 & 9 & -1 & 3 & -6 & -3 & 1 \end{bmatrix}
$

\vspace{\baselineskip}
\noindent
$\boldsymbol{n = 16, \rho = 7, |\Sigma| = 1}$

\noindent
$
E(G) = \{ (0, 6), (0, 7), (0, 9), (0, 10), (0, 11), (0, 14), (0, 15), (1, 6), (1, 8)^{\rC}, (1, 10), (1, 12), (1, 13), (1, 14),\\
(1, 15), (2, 7), (2, 8), (2, 9), (2, 11), (2, 12), (2, 13), (2, 14), (3, 7), (3, 8), (3, 10), (3, 11), (3, 12), (3, 13),\\
(3, 15), (4, 7), (4, 9), (4, 10), (4, 12), (4, 13), (4, 14), (4, 15), (5, 8), (5, 9), (5, 10), (5, 12), (5, 13), (5, 14),\\
(5, 15), (6, 9), (6, 11), (6, 12), (6, 13), (6, 15), (7, 8), (7, 11), (7, 12), (8, 11), (8, 14), (9, 11), (9, 13), (10, 14),\\
(10, 15) \}
$

\noindent
$
{\bf x} = \begin{bmatrix} 4 & -10 & -8 & -14 & -1 & 21 & 9 & 3 & 3 & 3 & -9 & -6 & 22 & -19 & -6 & 6 \end{bmatrix}
$

\vspace{\baselineskip}
\noindent
$\boldsymbol{n = 18, \rho = 7, |\Sigma| = 2}$

\noindent
$
E(G) = \{ (0, 7)^{\rC}, (0, 8), (0, 11), (0, 12), (0, 13), (0, 14), (0, 15), (1, 8), (1, 9), (1, 10), (1, 12), (1, 14), (1, 16), \\
(1, 17), (2, 9), (2, 10), (2, 11), (2, 13), (2, 14), (2, 15), (2, 17), (3, 9), (3, 10), (3, 12), (3, 13), (3, 14), (3, 15), \\
(3, 16), (4, 9), (4, 10)^{\rC}, (4, 12), (4, 13), (4, 15), (4, 16), (4, 17), (5, 9), (5, 11), (5, 12), (5, 14), (5, 15), (5, 16), \\
(5, 17), (6, 10), (6, 11), (6, 12), (6, 14), (6, 15), (6, 16), (6, 17), (7, 8), (7, 11), (7, 13), (7, 14), (7, 16), (7, 17), \\
(8, 11), (8, 13), (8, 15), (8, 17), (9, 11), (9, 12), (10, 13), (10, 16) \}
$

\noindent
$
{\bf x} = \begin{bmatrix} 2 & -3 & 1 & -3 & 4 & -7 & 8 & 2 & -5 & -1 & -1 & 2 & 6 & 4 & 3 & -8 & -3 & 1 \end{bmatrix}
$

\vspace{\baselineskip}
\noindent
$\boldsymbol{n = 20, \rho = 7, |\Sigma| = 2}$

\noindent
$
E(G) = \{ (0, 9), (0, 10), (0, 11), (0, 12), (0, 13), (0, 14), (0, 15)^{\rC}, (1, 9), (1, 10), (1, 11), (1, 13), (1, 16), \\ 
(1, 17), (1, 18), (2, 10), (2, 12), (2, 13), (2, 14), (2, 15), (2, 17), (2, 19), (3, 10), (3, 14), (3, 15), (3, 16)^{\rC}, \\
(3, 17), (3, 18), (3, 19), (4, 11), (4, 12), (4, 13), (4, 14), (4, 15), (4, 16), (4, 17), (5, 11), (5, 12), (5, 13), \\
(5, 14), (5, 16), (5, 17), (5, 19), (6, 11), (6, 14), (6, 15), (6, 16), (6, 17), (6, 18), (6, 19), (7, 12), (7, 13), \\
(7, 14), (7, 15), (7, 16), (7, 18), (7, 19), (8, 12), (8, 13), (8, 15), (8, 16), (8, 17), (8, 18), (8, 19), (9, 10), \\
(9, 11), (9, 12), (9, 18), (9, 19), (10, 11), (10, 18) \}
$

\noindent
$
{\bf x} = \begin{bmatrix} 3 & -1 & 4 & -3 & -2 & -4 & 4 & -2 & 2 & -1 & 1 & -3 & 1 & -2 & 2 & -2 & 2 & 2 & 1 & -2 \end{bmatrix}
$

\vspace{\baselineskip}
\noindent
$\boldsymbol{n = 22, \rho = 7, |\Sigma| = 2}$

\noindent
$
E(G) = \{ (0, 10)^{\rC}, (0, 11), (0, 12), (0, 13), (0, 16), (0, 18), (0, 20), (1, 11), (1, 12), (1, 13), (1, 14), (1, 15), \\
(1, 16), (1, 17), (2, 11), (2, 12), (2, 13), (2, 14), (2, 16), (2, 18), (2, 19), (3, 11), (3, 12), (3, 15), (3, 16), \\
(3, 17), (3, 18), (3, 19), (4, 11), (4, 13), (4, 16), (4, 17), (4, 19), (4, 20), (4, 21), (5, 11), (5, 15), (5, 17), \\
(5, 18), (5, 19), (5, 20), (5, 21), (6, 12), (6, 13), (6, 14), (6, 15), (6, 17), (6, 19), (6, 20), (7, 12), (7, 13), \\
(7, 14), (7, 15), (7, 18), (7, 19), (7, 21), (8, 12), (8, 14), (8, 15), (8, 16), (8, 18), (8, 20), (8, 21), (9, 13), \\
(9, 14), (9, 15), (9, 17), (9, 18), (9, 20), (9, 21), (10, 14), (10, 16), (10, 17), (10, 19), (10, 20), (10, 21), \\
(11, 21)^{\rC} \}
$

\noindent
$
{\bf x} = \big[\begin{matrix} 5 & 8 & -1 & -7 & -5 & 10 & -4 & 3 & -4 & -6 & 4 & 2 & 19 & -6 & 1 & -19 & -11 & 14 & 2 & -7 \end{matrix} \\
\begin{matrix}  -2 & 10 \end{matrix}\big]
$

\vspace{\baselineskip}
\noindent
$\boldsymbol{n = 12, \rho = 8, |\Sigma| = 2}$

\noindent
$
E(G) = \{ (0, 3)^{\rC}, (0, 4), (0, 5), (0, 6), (0, 7), (0, 8), (0, 10), (0, 11), (1, 3), (1, 5)^{\rC}, (1, 6), (1, 7), (1, 8), (1, 9), \\
(1, 10), (1, 11), (2, 4), (2, 5), (2, 6), (2, 7), (2, 8), (2, 9), (2, 10), (2, 11), (3, 6), (3, 7), (3, 8), (3, 9), (3, 10), \\
(3, 11), (4, 6), (4, 7), (4, 8), (4, 9), (4, 10), (4, 11), (5, 7), (5, 8), (5, 9), (5, 10), (5, 11), (6, 9), (6, 10), (6, 11), \\
(7, 9), (7, 10), (8, 9), (8, 11) \}
$

\noindent
$
{\bf x} = \begin{bmatrix} 1 & 1 & -1 & -1 & 1 & -1 & -1 & 1 & 1 & 1 & -1 & -1 \end{bmatrix}
$

\vspace{\baselineskip}
\noindent
$\boldsymbol{n = 14, \rho = 8, |\Sigma| = 2}$

\noindent
$
E(G) = \{ (0, 5), (0, 6)^{\rC}, (0, 7), (0, 8), (0, 9), (0, 10), (0, 12), (0, 13), (1, 5), (1, 6), (1, 7), (1, 9), (1, 10), \\
(1, 11), (1, 12), (1, 13), (2, 5), (2, 6), (2, 8), (2, 9), (2, 10), (2, 11), (2, 12), (2, 13), (3, 5), (3, 7), (3, 8), (3, 9), \\
(3, 10), (3, 11), (3, 12), (3, 13), (4, 6)^{\rC}, (4, 7), (4, 8), (4, 9), (4, 10), (4, 11), (4, 12), (4, 13), (5, 8), (5, 9), \\
(5, 10), (5, 11), (6, 8), (6, 10), (6, 11), (6, 12), (7, 9), (7, 11), (7, 12), (7, 13), (8, 11), (8, 13), (9, 13), (10, 12) \}
$

\noindent
$
{\bf x} = \begin{bmatrix} 1 & -1 & -1 & 3 & -3 & -1 & 1 & 1 & 1 & -1 & -1 & -1 & 1 & 1 \end{bmatrix}
$

\vspace{\baselineskip}
\noindent
$\boldsymbol{n = 15, \rho = 8, |\Sigma| = 1}$

\noindent
$
E(G) = \{ (0, 5), (0, 7), (0, 8), (0, 9), (0, 10), (0, 11), (0, 12), (0, 14), (1, 6), (1, 7), (1, 8), (1, 9), (1, 10),\\
 (1, 11), (1, 12), (1, 13), (2, 6), (2, 7), (2, 9), (2, 10), (2, 11), (2, 12), (2, 13), (2, 14), (3, 6), (3, 8), (3, 9), \\
(3, 10), (3, 11), (3, 12), (3, 13), (3, 14), (4, 7), (4, 8), (4, 9), (4, 10), (4, 11), (4, 12), (4, 13)^{\rC}, (4, 14), (5, 8),\\
 (5, 9), (5, 10), (5, 11), (5, 12), (5, 13), (5, 14), (6, 7), (6, 8), (6, 12), (6, 13), (6, 14), (7, 12), (7, 13), (7, 14), \\
(8, 13), (8, 14), (9, 10), (9, 11), (10, 11) \}
$

\noindent
$
{\bf x} = \begin{bmatrix} 1 & -3 & -2 & 6 & -2 & 2 & 1 & -3 & -3 & -1 & -1 & -1 & 10 & -2 & -3 \end{bmatrix}
$

\vspace{\baselineskip}
\noindent
$\boldsymbol{n = 16, \rho = 8, |\Sigma| = 2}$

\noindent
$
E(G) = \{ (0, 6)^{\rC}, (0, 7), (0, 9), (0, 10), (0, 11), (0, 12), (0, 13), (0, 14), (1, 7), (1, 8), (1, 9), (1, 10), (1, 11), \\
(1, 12), (1, 14), (1, 15), (2, 7), (2, 8), (2, 9), (2, 10), (2, 11), (2, 13), (2, 14), (2, 15), (3, 7), (3, 8), (3, 10), \\ 
(3, 11), (3, 12), (3, 13), (3, 14), (3, 15)^{\rC}, (4, 7), (4, 9), (4, 10), (4, 11), (4, 12), (4, 13), (4, 14), (4, 15), (5, 8), \\
(5, 9), (5, 10), (5, 11), (5, 12), (5, 13), (5, 14), (5, 15), (6, 8), (6, 9), (6, 10), (6, 11), (6, 12), (6, 13), (6, 14), \\
(7, 12), (7, 13), (7, 15), (8, 13), (8, 14), (8, 15), (9, 12), (9, 15), (10, 11) \}
$

\noindent
$
{\bf x} = \begin{bmatrix} 1 & -1 & 1 & -1 & 3 & -3 & 1 & -1 & -1 & 1 & -1 & -1 & -1 & -1 & 5 & -1 \end{bmatrix}
$

\pagebreak
\noindent
$\boldsymbol{n = 17, \rho = 8, |\Sigma| = 1}$

\noindent
$
E(G) = \{ (0, 6), (0, 7), (0, 8), (0, 10), (0, 11), (0, 14), (0, 15), (0, 16), (1, 6), (1, 9), (1, 10), (1, 12), (1, 13), \\
(1, 14), (1, 15), (1, 16), (2, 7), (2, 8), (2, 10), (2, 11), (2, 12), (2, 13), (2, 15)^{\rC}, (2, 16), (3, 7), (3, 9), (3, 10), \\
(3, 11), (3, 12), (3, 13), (3, 14), (3, 15), (4, 8), (4, 9), (4, 10), (4, 11), (4, 12), (4, 14), (4, 15), (4, 16), (5, 8), \\
(5, 9), (5, 10), (5, 12), (5, 13), (5, 14), (5, 15), (5, 16), (6, 7), (6, 9), (6, 11), (6, 13), (6, 14), (6, 16), (7, 8), \\
(7, 11), (7, 12), (7, 13), (8, 11), (8, 13), (8, 14), (9, 12), (9, 13), (9, 15), (10, 15), (10, 16), (11, 14), (12, 16) \}
$

\noindent
$
{\bf x} = \begin{bmatrix} 3 & -3 & -2 & -3 & 1 & 1 & 1 & 3 & 1 & 2 & -3 & -1 & 2 & -1 & -4 & 2 & 1 \end{bmatrix}
$

\vspace{\baselineskip}
\noindent
$\boldsymbol{n = 18, \rho = 8, |\Sigma| = 2}$

\noindent
$
E(G) = \{ (0, 7)^{\rC}, (0, 8), (0, 9), (0, 13), (0, 14), (0, 15), (0, 16), (0, 17), (1, 8), (1, 9), (1, 11), (1, 12)^{\rC}, (1, 13),\\
 (1, 15), (1, 16), (1, 17), (2, 8), (2, 10), (2, 11), (2, 12), (2, 13), (2, 14), (2, 15), (2, 16), (3, 8), (3, 10), (3, 12), \\
(3, 13), (3, 14), (3, 15), (3, 16), (3, 17), (4, 9), (4, 10), (4, 11), (4, 12), (4, 14), (4, 15), (4, 16), (4, 17), (5, 9), \\
(5, 11), (5, 12), (5, 13), (5, 14), (5, 15), (5, 16), (5, 17), (6, 10), (6, 11), (6, 12), (6, 13), (6, 14), (6, 15), \\
(6, 16), (6, 17), (7, 9), (7, 10), (7, 11), (7, 13), (7, 14), (7, 16), (7, 17), (8, 10), (8, 11), (8, 12), (8, 13), (9, 10), \\
(9, 15), (9, 17), (10, 12), (11, 14) \}
$

\noindent
$
{\bf x} = \begin{bmatrix} 1 & -1 & 1 & -1 & -1 & -1 & 1 & 1 & -1 & 1 & 1 & -1 & -1 & 1 & 1 & 1 & -1 & -1 \end{bmatrix}
$

\vspace{\baselineskip}
\noindent
$\boldsymbol{n = 19, \rho = 8, |\Sigma| = 1}$

\noindent
$
E(G) = \{ (0, 8), (0, 9), (0, 12), (0, 13), (0, 14), (0, 15), (0, 16), (0, 17), (1, 8), (1, 10), (1, 11), (1, 12), (1, 13), \\
(1, 14), (1, 16), (1, 18), (2, 9), (2, 10), (2, 11), (2, 12), (2, 13), (2, 14), (2, 17), (2, 18), (3, 9), (3, 10), (3, 11), \\
(3, 13), (3, 15), (3, 16), (3, 17), (3, 18), (4, 9), (4, 10), (4, 11), (4, 14), (4, 15), (4, 16), (4, 17), (4, 18), (5, 9), \\
(5, 12), (5, 13), (5, 14), (5, 15), (5, 16), (5, 17), (5, 18), (6, 10), (6, 12), (6, 13), (6, 14), (6, 15), (6, 16)^{\rC}, \\
(6, 17), (6, 18), (7, 11), (7, 12), (7, 13), (7, 14), (7, 15), (7, 16), (7, 17), (7, 18), (8, 10), (8, 11), (8, 12),\\
 (8, 13), (8, 14), (8, 15), (9, 11), (9, 12), (9, 16), (10, 15), (10, 17), (11, 18) \}
$

\noindent
$
{\bf x} = \begin{bmatrix} 4 & -6 & 2 & 3 & 3 & 3 & 1 & -9 & 2 & 3 & -7 & 3 & -13 & 9 & 9 & 1 & -5 & -6 & 2 \end{bmatrix}
$

\vspace{\baselineskip}
\noindent
$\boldsymbol{n = 20, \rho = 8, |\Sigma| = 1}$

\noindent
$
E(G) = \{ (0, 9), (0, 10), (0, 12), (0, 14), (0, 16), (0, 17), (0, 18), (0, 19), (1, 9), (1, 10), (1, 14), (1, 15), \\
(1, 16), (1, 17), (1, 18)^{\rC}, (1, 19), (2, 9), (2, 11), (2, 12), (2, 13), (2, 15), (2, 16), (2, 17), (2, 19), (3, 10), \\
(3, 11), (3, 12), (3, 13), (3, 14), (3, 15), (3, 16), (3, 17), (4, 10), (4, 11), (4, 12), (4, 13), (4, 14), (4, 16), \\
(4, 17), (4, 18), (5, 10), (5, 11), (5, 13), (5, 15), (5, 16), (5, 17), (5, 18), (5, 19), (6, 10), (6, 13), (6, 14), \\
(6, 15), (6, 16), (6, 17), (6, 18), (6, 19), (7, 11), (7, 12), (7, 13), (7, 14), (7, 15), (7, 16), (7, 18), (7, 19), \\
(8, 11), (8, 12), (8, 13), (8, 14), (8, 15), (8, 17), (8, 18), (8, 19), (9, 11), (9, 12), (9, 15), (9, 18), (9, 19), \\
(10, 12), (10, 13), (11, 14) \}
$

\noindent
$
{\bf x} = \begin{bmatrix} 4 & -2 & -2 & 4 & -6 & 4 & -5 & 3 & 3 & -5 & -1 & -1 & 2 & -1 & -1 & -2 & 2 & 2 & -2 & 3 \end{bmatrix}
$

\vspace{\baselineskip}
\noindent
$\boldsymbol{n = 21, \rho = 8, |\Sigma| = 2}$

\noindent
$
E(G) = \{ (0, 9)^{\rC}, (0, 10), (0, 11), (0, 13), (0, 15), (0, 16), (0, 19), (0, 20), (1, 10), (1, 11)^{\rC}, (1, 12), (1, 15), \\
(1, 17), (1, 18), (1, 19), (1, 20), (2, 10), (2, 12), (2, 13), (2, 14), (2, 16), (2, 17), (2, 18), (2, 19), (3, 10), \\
(3, 12), (3, 13), (3, 14), (3, 16), (3, 17), (3, 19), (3, 20), (4, 11), (4, 12), (4, 13), (4, 14), (4, 15), (4, 16), \\ 
(4, 17), (4, 18), (5, 11), (5, 12), (5, 13), (5, 15), (5, 17), (5, 18), (5, 19), (5, 20), (6, 11), (6, 12), (6, 14), \\
(6, 15), (6, 16), (6, 18), (6, 19), (6, 20), (7, 11), (7, 12), (7, 14), (7, 16), (7, 17), (7, 18), (7, 19), (7, 20), \\
(8, 11), (8, 13), (8, 14), (8, 15), (8, 16), (8, 17), (8, 18), (8, 20), (9, 13), (9, 14), (9, 15), (9, 16), (9, 17), \\
(9, 18), (9, 19), (10, 12), (10, 13), (10, 14), (10, 15), (11, 20) \}
$

\noindent
$
{\bf x} = \big[\begin{matrix} 9 & -1 & -1 & -13 & 3 & -3 & -13 & 19 & -7 & 3 & 9 & 9 & -9 & -9 & 15 & 9 & -15 & 9 & -9 & 9 \end{matrix} \\
\begin{matrix} -9 \end{matrix}\big]
$

\vspace{\baselineskip}
\noindent
$\boldsymbol{n = 22, \rho = 8, |\Sigma| = 2}$

\noindent
$
E(G) = \{ (0, 10)^{\rC}, (0, 11), (0, 12), (0, 13), (0, 14), (0, 15), (0, 16), (0, 17), (1, 11), (1, 12), (1, 13), (1, 14), \\
(1, 15), (1, 17), (1, 18), (1, 19), (2, 11), (2, 12), (2, 13), (2, 15), (2, 16), (2, 17), (2, 19), (2, 20), (3, 11), \\
(3, 12), (3, 15), (3, 16), (3, 17), (3, 18), (3, 19), (3, 21), (4, 11), (4, 13), (4, 14), (4, 15), (4, 18), (4, 19),\\
 (4, 20), (4, 21), (5, 11), (5, 13), (5, 14), (5, 16), (5, 17), (5, 19), (5, 20), (5, 21), (6, 11), (6, 14), (6, 16), \\
(6, 17), (6, 18), (6, 19), (6, 20), (6, 21), (7, 12), (7, 13), (7, 14), (7, 15), (7, 16), (7, 18), (7, 19), (7, 20), \\
(8, 12), (8, 13), (8, 14)^{\rC}, (8, 15), (8, 17), (8, 18), (8, 20), (8, 21), (9, 12), (9, 13), (9, 14), (9, 16), (9, 17), \\
(9, 18), (9, 20), (9, 21), (10, 12), (10, 15), (10, 16), (10, 18), (10, 19), (10, 20), (10, 21), (11, 21) \}
$

\noindent
$
{\bf x} = \big[\begin{matrix} 1 & 1 & -3 & -1 & -1 & 3 & -1 & -1 & 1 & -1 & 3 & -3 & -1 & -3 & 3 & 5 & -1 & 3 & -3 & -1 & 1 \end{matrix} \\
\begin{matrix}  1 \end{matrix}\big]
$

\vspace{\baselineskip}
\noindent
$\boldsymbol{n = 23, \rho = 8, |\Sigma| = 2}$

\noindent
$
E(G) = \{ (0, 11), (0, 12), (0, 13), (0, 14), (0, 15), (0, 16), (0, 17)^{\rC}, (0, 18), (1, 11), (1, 12), (1, 13), (1, 14), \\
(1, 15), (1, 16), (1, 17), (1, 18)^{\rC}, (2, 11), (2, 12), (2, 13), (2, 14), (2, 16), (2, 17), (2, 18), (2, 19), (3, 11), \\
(3, 15), (3, 16), (3, 17), (3, 18), (3, 20), (3, 21), (3, 22), (4, 12), (4, 13), (4, 14), (4, 15), (4, 16), (4, 17), \\
(4, 20), (4, 21), (5, 12), (5, 13), (5, 14), (5, 18), (5, 19), (5, 20), (5, 21), (5, 22), (6, 12), (6, 13), (6, 15), \\
(6, 16), (6, 19), (6, 20), (6, 21), (6, 22), (7, 12), (7, 14), (7, 15), (7, 17), (7, 18), (7, 19), (7, 20), (7, 22), \\
(8, 12), (8, 14), (8, 15), (8, 17), (8, 18), (8, 19), (8, 21), (8, 22), (9, 13), (9, 14), (9, 15), (9, 16), (9, 19),\\
 (9, 20), (9, 21), (9, 22), (10, 13), (10, 16), (10, 17), (10, 18), (10, 19), (10, 20), (10, 21), (10, 22), (11, 19),\\
 (11, 20), (11, 21), (11, 22) \}
$

\noindent
$
{\bf x} = \big[\begin{matrix} 12 & 6 & -9 & -9 & 3 & -9 & -9 & 3 & 3 & -9 & 15 & 15 & -15 & 25 & -15 & 25 & -35 & 5 & 5 & 15 & 5 \end{matrix} \\
 \begin{matrix} 5 & -25 \end{matrix}\big]
$

\vspace{\baselineskip}
\noindent
$\boldsymbol{n = 24, \rho = 8, |\Sigma| = 2}$

\noindent
$
E(G) = \{ (0, 1), (0, 2), (0, 3), (0, 4), (0, 5), (0, 6), (0, 7), (0, 8), (1, 2), (1, 3), (1, 4), (1, 5), (1, 6), (1, 7), \\
(1, 8), (2, 3), (2, 4), (2, 5), (2, 6), (2, 7), (2, 8), (3, 4), (3, 5), (3, 6), (3, 7), (3, 8), (4, 5), (4, 6), (4, 7), (4, 8),\\
(5, 6), (5, 9)^{\rC}, (5, 10), (6, 9), (6, 11), (7, 8), (7, 9), (7, 10), (8, 9), (8, 11), (9, 10), (9, 11), (9, 12), (9, 13)^{\rC},\\
(10, 11), (10, 12), (10, 13), (10, 14), (10, 15), (11, 12), (11, 13), (11, 14), (11, 15), (12, 13), (12, 14), (12, 15),\\ 
(12, 16), (12, 17), (13, 14), (13, 16), (13, 17), (13, 18), (14, 15), (14, 19), (14, 20), (14, 21), (15, 16), (15, 19),\\
(15, 22), (15, 23), (16, 18), (16, 20), (16, 21), (16, 22), (16, 23), (17, 18), (17, 19), (17, 20), (17, 21), (17, 22),\\
(17, 23), (18, 19), (18, 20), (18, 21), (18, 22), (18, 23), (19, 20), (19, 21), (19, 22), (19, 23), (20, 21), (20, 22),\\
 (20, 23), (21, 22), (21, 23), (22, 23) \}
$

\noindent
$
{\bf x} = \big[\begin{matrix} 1 & 1 & 1 & 1 & 1 & -4 & -4 & -4 & 8 & -6 & -7 & 5 & 2 & 4 & -6 & 9 & -11 & 12 & -1 & 9 & -11 \end{matrix} \\
\begin{matrix} -11 & 4 & 4 \end{matrix}\big]
$

\vspace{\baselineskip}
\noindent
$\boldsymbol{n = 16, \rho = 9, |\Sigma| = 2}$

\noindent
$
E(G) = \{ (0, 1), (0, 2), (0, 3)^{\rC}, (0, 4), (0, 5), (0, 6), (0, 7), (0, 8), (0, 9), (1, 2), (1, 3), (1, 4), (1, 5), (1, 6),\\
(1, 7), (1, 8), (1, 9), (2, 3), (2, 4), (2, 5), (2, 6), (2, 7), (2, 8), (2, 9), (3, 4), (3, 5), (3, 6), (3, 7), (3, 10), (3, 11),\\
(4, 5), (4, 8), (4, 10), (4, 12), (4, 13), (5, 11), (5, 12), (5, 13), (5, 14), (6, 9), (6, 11), (6, 12), (6, 14), (6, 15),\\
(7, 10), (7, 11), (7, 12), (7, 13), (7, 15), (8, 9)^{\rC}, (8, 10), (8, 13), (8, 14), (8, 15), (9, 11), (9, 13), (9, 14),\\
(9, 15), (10, 11), (10, 12), (10, 13), (10, 14), (10, 15), (11, 12), (11, 14), (11, 15), (12, 13), (12, 14), (12, 15),\\
(13, 14), (13, 15), (14, 15) \}
$

\noindent
$
{\bf x} = \begin{bmatrix} 11 & -3 & -3 & -7 & -2 & 3 & 8 & 3 & 3 & -16 & -2 & 7 & 13 & -15 & -1 & -1 \end{bmatrix}
$

\vspace{\baselineskip}
\noindent
$\boldsymbol{n = 18, \rho = 9, |\Sigma| = 2}$

\noindent
$
E(G) = \{ (0, 1), (0, 2), (0, 3), (0, 4)^{\rC}, (0, 5), (0, 6), (0, 7), (0, 8), (0, 9), (1, 2), (1, 3), (1, 4), (1, 5), (1, 6),\\
(1, 7), (1, 8), (1, 9)^{\rC}, (2, 3), (2, 4), (2, 5), (2, 6), (2, 7), (2, 8), (2, 9), (3, 4), (3, 5), (3, 6), (3, 7), (3, 8), (3, 9),\\
(4, 5), (4, 6), (4, 7), (4, 8), (4, 10), (5, 6), (5, 11), (5, 12), (5, 13), (6, 14), (6, 15), (6, 16), (7, 10), (7, 11),\\
(7, 14), (7, 17), (8, 11), (8, 12), (8, 15), (8, 16), (9, 10), (9, 11), (9, 12), (9, 13), (9, 17), (10, 12), (10, 13),\\
(10, 14), (10, 15), (10, 16), (10, 17), (11, 13), (11, 14), (11, 15), (11, 16), (11, 17), (12, 13), (12, 14), (12, 15),\\
(12, 16), (12, 17), (13, 14), (13, 15), (13, 16), (13, 17), (14, 15), (14, 16), (14, 17), (15, 16), (15, 17), (16, 17) \}
$

\noindent
$
{\bf x} = \begin{bmatrix} 14 & 24 & -10 & -10 & -12 & 12 & -23 & 9 & 3 & -17 & 9 & -17 & 27 & 7 & -2 & -8 & -8 & 4 \end{bmatrix}
$

\pagebreak
\noindent
$\boldsymbol{n = 20, \rho = 9, |\Sigma| = 1}$

\noindent
$
E(G) = \{ (0, 1), (0, 2), (0, 3), (0, 4), (0, 5), (0, 6), (0, 7), (0, 8), (0, 9), (1, 2), (1, 3), (1, 4), (1, 5), (1, 6),\\
(1, 7), (1, 8), (1, 9), (2, 3), (2, 4), (2, 5), (2, 6), (2, 7), (2, 8), (2, 9), (3, 4), (3, 5), (3, 6), (3, 7), (3, 8), (3, 9),\\
(4, 5), (4, 6), (4, 7), (4, 8), (4, 9), (5, 6), (5, 7), (5, 8), (5, 10), (6, 10), (6, 11), (6, 12), (7, 11), (7, 13), (7, 14),\\
(8, 12), (8, 13), (8, 15), (9, 11), (9, 14), (9, 15), (9, 16), (10, 12), (10, 14), (10, 15), (10, 16), (10, 17)^{\rC},\\
(10, 18), (10, 19), (11, 12), (11, 13), (11, 15), (11, 17), (11, 18), (11, 19), (12, 13), (12, 16), (12, 17), (12, 18),\\
(12, 19), (13, 14), (13, 16), (13, 17), (13, 18), (13, 19), (14, 15), (14, 16), (14, 17), (14, 18), (14, 19), (15, 16),\\
(15, 17), (15, 18), (15, 19), (16, 17), (16, 18), (16, 19), (17, 18), (17, 19), (18, 19) \}
$

\noindent
$
{\bf x} = \begin{bmatrix} 1 & 1 & 1 & 1 & 1 & -3 & 2 & -2 & -4 & 3 & -1 & -5 & 4 & 2 & 1 & -8 & 7 & 1 & -1 & -1 \end{bmatrix}
$

\vspace{\baselineskip}
\noindent
$\boldsymbol{n = 22, \rho = 9, |\Sigma| = 2}$

\noindent
$
E(G) = \{ (0, 1), (0, 2), (0, 3), (0, 4), (0, 5), (0, 6), (0, 7), (0, 8), (0, 9), (1, 2), (1, 3), (1, 4), (1, 5), (1, 6),\\
(1, 7), (1, 8), (1, 9), (2, 3), (2, 4), (2, 5), (2, 6), (2, 7), (2, 8), (2, 9), (3, 4), (3, 5), (3, 6), (3, 7), (3, 8), (3, 9),\\
(4, 5), (4, 6), (4, 7), (4, 8), (4, 9), (5, 6), (5, 7), (5, 8), (5, 9), (6, 7), (6, 8), (6, 10), (7, 8), (7, 10), (8, 11),\\
(9, 11), (9, 12), (9, 13), (10, 12), (10, 13), (10, 14), (10, 15), (10, 16), (10, 17), (10, 18), (11, 12), (11, 14)^{\rC},\\
(11, 15), (11, 16), (11, 19), (11, 20), (11, 21), (12, 13), (12, 14), (12, 15), (12, 17), (12, 19)^{\rC}, (12, 20),\\
(13, 14), (13, 16), (13, 17), (13, 18), (13, 19), (13, 21), (14, 15), (14, 16), (14, 18), (14, 20), (14, 21), (15, 17),\\
(15, 18), (15, 19), (15, 20), (15, 21), (16, 17), (16, 18), (16, 19), (16, 20), (16, 21), (17, 18), (17, 19), (17, 20),\\
(17, 21), (18, 19), (18, 20), (18, 21), (19, 20), (19, 21), (20, 21) \}
$

\noindent
$
{\bf x} = \big[\begin{matrix} 4 & 4 & 4 & 4 & 4 & 4 & -3 & -3 & -23 & 9 & 2 & -18 & -3 & -3 & -24 & 29 & -36 & 32 & 11 & 18  \end{matrix} \\
\begin{matrix} -9 & -9 \end{matrix}\big]
$

\vspace{\baselineskip}
\noindent
$\boldsymbol{n = 24, \rho = 9, |\Sigma| = 3}$

\noindent
$
E(G) = \{ (0, 1)^{\rC}, (0, 2), (0, 3), (0, 4), (0, 5), (0, 6), (0, 7), (0, 8), (0, 9), (1, 2), (1, 3), (1, 4), (1, 5), (1, 6),\\
(1, 7), (1, 8), (1, 9), (2, 3), (2, 4), (2, 5), (2, 6), (2, 7), (2, 8), (2, 9), (3, 4), (3, 5), (3, 6), (3, 10), (3, 11),\\
(3, 12), (4, 5), (4, 6), (4, 10), (4, 11), (4, 13), (5, 14), (5, 15), (5, 16), (5, 17), (6, 18), (6, 19), (6, 20), (6, 21),\\
(7, 8), (7, 10), (7, 14), (7, 15), (7, 18)^{\rC}, (7, 19), (8, 10), (8, 14), (8, 15), (8, 18), (8, 19), (9, 10), (9, 11),\\
(9, 12), (9, 13), (9, 14), (9, 15), (10, 11), (10, 12), (10, 13), (10, 14), (11, 12), (11, 13), (11, 14), (11, 15),\\
(11, 16), (12, 13), (12, 14), (12, 15), (12, 16), (12, 22), (13, 17), (13, 20), (13, 22), (13, 23), (14, 21), (14, 22),\\
(15, 18), (15, 20), (15, 23), (16, 17), (16, 19), (16, 20), (16, 21), (16, 22), (16, 23), (17, 18), (17, 19), (17, 20),\\
(17, 21), (17, 22), (17, 23), (18, 20), (18, 21), (18, 22), (18, 23), (19, 20), (19, 21), (19, 22), (19, 23)^{\rC},\\
(20, 21), (20, 23), (21, 22), (21, 23), (22, 23) \}
$

\noindent
$
{\bf x} = \big[\begin{matrix} 2 & 2 & 6 & 7 & -6 & 5 & -15 & -28 & 30 & 3 & -10 & -10 & 26 & 13 & -35 & 6 & -4 & 22 & 29 & 28 \end{matrix} \\
\begin{matrix} -22 & -46 & 30 & -33 \end{matrix}\big]
$

\vspace{\baselineskip}
\noindent
$\boldsymbol{n = 26, \rho = 9, |\Sigma| = 2}$

\noindent
$
E(G) = \{ (0, 1), (0, 2), (0, 3), (0, 4)^{\rC}, (0, 5), (0, 6), (0, 7), (0, 8), (0, 9), (1, 2), (1, 3), (1, 4), (1, 5), (1, 6),\\
(1, 7), (1, 8), (1, 9), (2, 3), (2, 4), (2, 5), (2, 6), (2, 7), (2, 8), (2, 9), (3, 4), (3, 5), (3, 6), (3, 7), (3, 8), (3, 9),\\
(4, 5), (4, 6), (4, 7), (4, 8), (4, 10), (5, 6), (5, 7), (5, 11), (5, 12), (6, 11), (6, 13), (6, 14), (7, 15), (7, 16),\\
(7, 17), (8, 9), (8, 13), (8, 15), (8, 16), (9, 10), (9, 11), (9, 12), (9, 13), (10, 11), (10, 12), (10, 13), (10, 14),\\
(10, 15), (10, 16), (10, 17), (11, 12), (11, 13), (11, 14), (11, 15), (11, 16), (12, 13), (12, 14), (12, 15), (12, 16),\\
(12, 17), (13, 18), (13, 19), (13, 20), (14, 15), (14, 18), (14, 19), (14, 21), (14, 22), (15, 18), (15, 21), (15, 23),\\
(16, 17), (16, 20), (16, 24), (16, 25), (17, 20), (17, 22)^{\rC}, (17, 23), (17, 24), (17, 25), (18, 19), (18, 21),\\
(18, 22), (18, 23), (18, 24), (18, 25), (19, 20), (19, 21), (19, 22), (19, 23), (19, 24), (19, 25), (20, 21), (20, 22),\\
(20, 23), (20, 24), (20, 25), (21, 22), (21, 23), (21, 24), (21, 25), (22, 23), (22, 24), (22, 25), (23, 24), (23, 25),\\
(24, 25) \}
$

\noindent
$
{\bf x} = \big[\begin{matrix} 8 & -10 & -10 & -10 & -9 & -5 & 24 & 16 & 28 & -42 & -25 & 12 & -21 & 28 & -4 & 20 & 25 & -9 \end{matrix} \\
\begin{matrix} 11 & 2 & 11 & -6 & -17 & -11 & -6 & -6 \end{matrix}\big]
$

\pagebreak
\noindent
$\boldsymbol{n = 28, \rho = 9, |\Sigma| = 1}$

\noindent
$
E(G) = \{ (0, 1), (0, 2), (0, 3), (0, 4), (0, 5), (0, 6), (0, 7), (0, 8), (0, 9), (1, 2), (1, 3), (1, 4), (1, 5), (1, 6),\\
(1, 7), (1, 8), (1, 9), (2, 3), (2, 4), (2, 5), (2, 6), (2, 7), (2, 8), (2, 9), (3, 4), (3, 5), (3, 6), (3, 10), (3, 11),\\
(3, 12), (4, 5), (4, 6), (4, 10), (4, 11), (4, 13), (5, 10), (5, 12), (5, 13), (5, 14), (6, 12), (6, 15), (6, 16), (6, 17),\\
(7, 8), (7, 9), (7, 10), (7, 12), (7, 15), (7, 18), (8, 9), (8, 10), (8, 12), (8, 15), (8, 18), (9, 10), (9, 12), (9, 15),\\
(9, 18), (10, 11), (10, 12), (10, 13), (11, 12), (11, 13), (11, 14), (11, 15), (11, 16), (11, 17), (12, 13), (13, 14),\\
(13, 15), (13, 16), (13, 17), (14, 15), (14, 16), (14, 17), (14, 18), (14, 19), (14, 20), (15, 16), (15, 19), (16, 19),\\
(16, 21), (16, 22)^{\rC}, (16, 23), (17, 19), (17, 20), (17, 24), (17, 25), (17, 26), (18, 21), (18, 22), (18, 23),\\
(18, 24), (18, 27), (19, 20), (19, 21), (19, 25), (19, 26), (19, 27), (20, 22), (20, 23), (20, 24), (20, 25), (20, 26),\\
(20, 27), (21, 22), (21, 23), (21, 24), (21, 25), (21, 26), (21, 27), (22, 23), (22, 24), (22, 25), (22, 26), (22, 27),\\
(23, 24), (23, 25), (23, 26), (23, 27), (24, 25), (24, 26), (24, 27), (25, 26), (25, 27), (26, 27) \}
$

\noindent
$
{\bf x} = \big[\begin{matrix} 1 & 1 & 1 & -5 & 8 & -4 & -4 & 1 & 1 & 1 & 3 & 3 & -9 & 4 & -4 & 3 & -1 & 1 & -2 & -1 & -3 & 1 & 1 \end{matrix} \\
\begin{matrix} -1 & 1 & 2 & 2 & -1 \end{matrix}\big]
$

\vspace{\baselineskip}
\noindent
$\boldsymbol{n = 15, \rho = 10, |\Sigma| = 2}$

\noindent
$
E(G) = \{ (0, 1), (0, 2), (0, 3), (0, 4)^{\rC}, (0, 5), (0, 6), (0, 7), (0, 8), (0, 9), (0, 10), (1, 2), (1, 3), (1, 4), (1, 5),\\
(1, 6), (1, 7), (1, 8), (1, 9), (1, 11), (2, 3), (2, 4), (2, 5), (2, 6), (2, 7), (2, 8), (2, 10), (2, 12), (3, 4), (3, 5),\\
(3, 6), (3, 7), (3, 9), (3, 11), (3, 13), (4, 5), (4, 6), (4, 7), (4, 10), (4, 13), (4, 14)^{\rC}, (5, 6), (5, 9), (5, 11), (5, 12),\\
(5, 14), (6, 11), (6, 12), (6, 13), (6, 14), (7, 8), (7, 10), (7, 12), (7, 13), (7, 14), (8, 9), (8, 10), (8, 11), (8, 12),\\
(8, 13), (8, 14), (9, 10), (9, 11), (9, 12), (9, 13), (9, 14), (10, 11), (10, 12), (10, 13), (10, 14), (11, 12),\\
(11, 13), (11, 14), (12, 13), (12, 14), (13, 14) \}
$

\noindent
$
{\bf x} = \begin{bmatrix} 3 & -2 & -2 & 4 & -3 & -3 & 4 & 4 & -2 & -3 & -3 & -2 & -2 & 4 & 3 \end{bmatrix}
$

\vspace{\baselineskip}
\noindent
$\boldsymbol{n = 16, \rho = 10, |\Sigma| = 2}$

\noindent
$
E(G) = \{ (0, 1)^{\rC}, (0, 2), (0, 3), (0, 4), (0, 5)^{\rC}, (0, 6), (0, 7), (0, 8), (0, 9), (0, 10), (1, 2), (1, 3), (1, 4), (1, 5),\\
(1, 6), (1, 7), (1, 8), (1, 9), (1, 10), (2, 3), (2, 4), (2, 5), (2, 6), (2, 7), (2, 8), (2, 9), (2, 11), (3, 4), (3, 5), (3, 6),\\
(3, 7), (3, 8), (3, 10), (3, 11), (4, 5), (4, 6), (4, 11), (4, 12), (4, 13), (4, 14), (5, 7), (5, 12), (5, 13), (5, 14),\\
(5, 15), (6, 8), (6, 12), (6, 13), (6, 14), (6, 15), (7, 9), (7, 10), (7, 11), (7, 12), (7, 15), (8, 9), (8, 12), (8, 13),\\
(8, 14), (8, 15), (9, 10), (9, 11), (9, 13), (9, 14), (9, 15), (10, 11), (10, 12), (10, 13), (10, 14), (10, 15),\\
(11, 12), (11, 13), (11, 14), (11, 15), (12, 13), (12, 14), (12, 15), (13, 14), (13, 15), (14, 15) \}
$

\noindent
$
{\bf x} = \begin{bmatrix} 1 & -5 & 5 & -3 & -1 & 3 & -1 & 3 & 1 & 1 & -7 & 1 & 1 & -1 & -1 & 3 \end{bmatrix}
$

\vspace{\baselineskip}
\noindent
$\boldsymbol{n = 17, \rho = 10, |\Sigma| = 2}$

\noindent
$
E(G) = \{ (0, 1)^{\rC}, (0, 2), (0, 3), (0, 4), (0, 5), (0, 6), (0, 7), (0, 8), (0, 9), (0, 10), (1, 2), (1, 3), (1, 4), (1, 5),\\
(1, 6), (1, 7), (1, 8), (1, 9), (1, 10), (2, 3), (2, 4), (2, 5), (2, 6), (2, 7), (2, 8), (2, 9), (2, 10), (3, 4), (3, 5), (3, 6),\\
(3, 7), (3, 8), (3, 9), (3, 11), (4, 5), (4, 6), (4, 7), (4, 10), (4, 12), (4, 13), (5, 8), (5, 12)^{\rC}, (5, 13), (5, 14),\\
(5, 15), (6, 10), (6, 11), (6, 12), (6, 14), (6, 16), (7, 12), (7, 13), (7, 14), (7, 15), (7, 16), (8, 11), (8, 13),\\
(8, 14), (8, 15), (8, 16), (9, 11), (9, 12), (9, 13), (9, 14), (9, 15), (9, 16), (10, 11), (10, 12), (10, 14), (10, 15),\\
(10, 16), (11, 12), (11, 13), (11, 14), (11, 15), (11, 16), (12, 13), (12, 15), (12, 16), (13, 14), (13, 15), (13, 16),\\
(14, 15), (14, 16), (15, 16) \}
$

\noindent
$
{\bf x} = \begin{bmatrix} 1 & 1 & 3 & -2 & -3 & -1 & -2 & 2 & -1 & 3 & 2 & -3 & -1 & -3 & 1 & 2 & 1 \end{bmatrix}
$

\vspace{\baselineskip}
\noindent
$\boldsymbol{n = 18, \rho = 10, |\Sigma| = 2}$

\noindent
$
E(G) = \{ (0, 1), (0, 2), (0, 3), (0, 4), (0, 5)^{\rC}, (0, 6), (0, 7), (0, 8), (0, 9)^{\rC}, (0, 10), (1, 2), (1, 3), (1, 4), (1, 5),\\
(1, 6), (1, 7), (1, 8), (1, 9), (1, 10), (2, 3), (2, 4), (2, 5), (2, 6), (2, 7), (2, 8), (2, 9), (2, 10), (3, 4), (3, 5), (3, 6),\\
(3, 7), (3, 8), (3, 9), (3, 10), (4, 5), (4, 6), (4, 7), (4, 8), (4, 11), (4, 12), (5, 6), (5, 9), (5, 13), (5, 14), (5, 15),\\
(6, 10), (6, 13), (6, 16), (6, 17), (7, 11), (7, 12), (7, 13), (7, 14), (7, 15), (8, 13), (8, 14), (8, 15), (8, 16),\\
(8, 17), (9, 11), (9, 12), (9, 14), (9, 16), (9, 17), (10, 11), (10, 12), (10, 15), (10, 16), (10, 17), (11, 12),\\
(11, 13), (11, 14), (11, 15), (11, 16), (11, 17), (12, 13), (12, 14), (12, 15), (12, 16), (12, 17), (13, 14), (13, 15),\\
(13, 16), (13, 17), (14, 15), (14, 16), (14, 17), (15, 16), (15, 17), (16, 17) \}
$

\noindent
$
{\bf x} = \begin{bmatrix} 1 & -3 & -3 & -3 & 13 & -5 & 7 & -4 & 2 & 3 & -11 & 4 & 4 & 3 & -1 & -15 & 4 & 4 \end{bmatrix}
$

\vspace{\baselineskip}
\noindent
$\boldsymbol{n = 19, \rho = 10, |\Sigma| = 2}$

\noindent
$
E(G) = \{ (0, 1), (0, 2), (0, 3), (0, 4), (0, 5), (0, 6), (0, 7), (0, 8), (0, 9), (0, 10), (1, 2), (1, 3), (1, 4), (1, 5),\\
(1, 6), (1, 7), (1, 8), (1, 9), (1, 10), (2, 3), (2, 4), (2, 5), (2, 6), (2, 7), (2, 8), (2, 9), (2, 10), (3, 4), (3, 5), (3, 6),\\
(3, 7), (3, 8), (3, 9), (3, 10), (4, 5)^{\rC}, (4, 6), (4, 7), (4, 8), (4, 9), (4, 11), (5, 6), (5, 7), (5, 10), (5, 11), (5, 12),\\
(6, 10), (6, 13), (6, 14), (6, 15), (7, 12), (7, 13), (7, 16), (7, 17), (8, 11), (8, 14), (8, 16), (8, 17), (8, 18),\\
(9, 13), (9, 15), (9, 16), (9, 17), (9, 18), (10, 12), (10, 14), (10, 15), (10, 18), (11, 12), (11, 13), (11, 14),\\
(11, 15), (11, 16), (11, 17), (11, 18), (12, 13), (12, 14), (12, 15), (12, 16), (12, 17), (12, 18), (13, 14), (13, 15),\\
(13, 16), (13, 17), (13, 18), (14, 15), (14, 16), (14, 17), (14, 18), (15, 16), (15, 17), (15, 18)^{\rC}, (16, 17),\\
(16, 18), (17, 18) \}
$

\noindent
$
{\bf x} = \begin{bmatrix} 1 & 1 & 1 & 1 & -5 & 3 & -6 & -4 & 4 & 4 & 1 & 1 & -1 & -7 & -2 & 6 & 3 & 3 & -4 \end{bmatrix}
$

\vspace{\baselineskip}
\noindent
$\boldsymbol{n = 20, \rho = 10, |\Sigma| = 2}$

\noindent
$
E(G) = \{ (0, 1), (0, 2), (0, 3), (0, 4), (0, 5)^{\rC}, (0, 6), (0, 7), (0, 8), (0, 9), (0, 10), (1, 2), (1, 3), (1, 4), (1, 5),\\
(1, 6), (1, 7), (1, 8), (1, 9), (1, 10), (2, 3), (2, 4), (2, 5), (2, 6), (2, 7), (2, 8), (2, 9), (2, 10), (3, 4), (3, 5), (3, 6),\\
(3, 7), (3, 8), (3, 9), (3, 10), (4, 5), (4, 6), (4, 7), (4, 8), (4, 9), (4, 10), (5, 6), (5, 7), (5, 8), (5, 11), (5, 12),\\
(6, 7), (6, 9), (6, 11), (6, 13), (7, 14), (7, 15), (7, 16)^{\rC}, (8, 11), (8, 13), (8, 17), (8, 18), (9, 12), (9, 14), (9, 15),\\
(9, 19), (10, 13), (10, 16), (10, 17), (10, 18), (10, 19), (11, 12), (11, 14), (11, 15), (11, 16), (11, 17), (11, 18),\\
(11, 19), (12, 13), (12, 14), (12, 15), (12, 16), (12, 17), (12, 18), (12, 19), (13, 14), (13, 15), (13, 16), (13, 17),\\
(13, 18), (13, 19), (14, 15), (14, 16), (14, 17), (14, 18), (14, 19), (15, 16), (15, 17), (15, 18), (15, 19), (16, 17),\\
(16, 18), (16, 19), (17, 18), (17, 19), (18, 19) \}
$

\noindent
$
{\bf x} = \begin{bmatrix} 1 & -1 & -1 & -1 & -1 & -1 & 7 & -1 & -1 & -1 & -1 & 1 & -1 & 5 & -1 & -1 & 1 & -1 & -1 & -1 \end{bmatrix}
$

\vspace{\baselineskip}
\noindent
$\boldsymbol{n = 21, \rho = 10, |\Sigma| = 2}$

\noindent
$
E(G) = \{ (0, 1), (0, 2), (0, 3), (0, 4), (0, 5), (0, 6), (0, 7), (0, 8), (0, 9), (0, 10), (1, 2), (1, 3), (1, 4), (1, 5),\\
(1, 6), (1, 7), (1, 8), (1, 9), (1, 10), (2, 3), (2, 4), (2, 5), (2, 6), (2, 7), (2, 8), (2, 9), (2, 10), (3, 4), (3, 5), (3, 6),\\
(3, 7), (3, 8), (3, 9), (3, 10), (4, 5), (4, 6), (4, 7), (4, 8), (4, 9), (4, 10), (5, 6)^{\rC}, (5, 7), (5, 8), (5, 9), (5, 11),\\
(6, 7), (6, 8), (6, 10), (6, 12)^{\rC}, (7, 11), (7, 13), (7, 14), (8, 15), (8, 16), (8, 17), (9, 13), (9, 14), (9, 18), (9, 19),\\
(10, 12), (10, 13), (10, 15), (10, 20), (11, 12), (11, 14), (11, 15), (11, 16), (11, 17), (11, 18), (11, 19), (11, 20),\\
(12, 13), (12, 15), (12, 16), (12, 17), (12, 18), (12, 19), (12, 20), (13, 14), (13, 16), (13, 17), (13, 18), (13, 19),\\
(13, 20), (14, 15), (14, 16), (14, 17), (14, 18), (14, 19), (14, 20), (15, 16), (15, 17), (15, 18), (15, 19), (15, 20),\\
(16, 17), (16, 18), (16, 19), (16, 20), (17, 18), (17, 19), (17, 20), (18, 19), (18, 20), (19, 20) \}
$

\noindent
$
{\bf x} = \begin{bmatrix} 1 & 1 & 1 & 1 & 1 & -3 & 5 & -9 & -11 & 9 & 5 & 11 & -7 & -24 & 6 & 17 & -12 & -12 & 8 & 8 & 4 \end{bmatrix}
$

\vspace{\baselineskip}
\noindent
$\boldsymbol{n = 22, \rho = 10, |\Sigma| = 3}$

\noindent
$
E(G) = \{ (0, 1)^{\rC}, (0, 2), (0, 3), (0, 4), (0, 5), (0, 6), (0, 7), (0, 8), (0, 9), (0, 10), (1, 2), (1, 3), (1, 4), (1, 5),\\
(1, 6), (1, 7), (1, 8), (1, 9), (1, 10), (2, 3), (2, 4), (2, 5), (2, 6), (2, 7), (2, 8), (2, 9)^{\rC}, (2, 10), (3, 4), (3, 5),\\
(3, 6), (3, 7), (3, 8), (3, 9), (3, 10), (4, 5), (4, 6), (4, 7), (4, 8), (4, 9), (4, 10), (5, 6), (5, 7), (5, 8), (5, 9),\\
(5, 10), (6, 7), (6, 8), (6, 9), (6, 11), (7, 8), (7, 11), (7, 12), (8, 11), (8, 13), (9, 14), (9, 15), (9, 16), (10, 11),\\
(10, 12), (10, 14), (10, 17)^{\rC}, (11, 13), (11, 15), (11, 18), (11, 19), (11, 20), (11, 21), (12, 13), (12, 14),\\
(12, 16), (12, 17), (12, 18), (12, 19), (12, 20), (12, 21), (13, 15), (13, 16), (13, 17), (13, 18), (13, 19), (13, 20),\\
(13, 21), (14, 15), (14, 16), (14, 17), (14, 18), (14, 19), (14, 20), (14, 21), (15, 16), (15, 17), (15, 18), (15, 19),\\
(15, 20), (15, 21), (16, 17), (16, 18), (16, 19), (16, 20), (16, 21), (17, 18), (17, 19), (17, 20), (17, 21), (18, 19),\\
(18, 20), (18, 21), (19, 20), (19, 21), (20, 21) \}
$

\noindent
$
{\bf x} = \begin{bmatrix} 1 & 1 & 17 & 3 & 3 & 3 & -7 & -17 & 3 & -7 & 3 & -7 & -17 & 3 & 1 & 11 & 1 & 5 & 1 & 1 & 1 & 1 \end{bmatrix}
$

\vspace{\baselineskip}
\noindent
$\boldsymbol{n = 23, \rho = 10, |\Sigma| = 2}$

\noindent
$
E(G) = \{ (0, 1), (0, 2), (0, 3), (0, 4), (0, 5), (0, 6), (0, 7), (0, 8)^{\rC}, (0, 9), (0, 10), (1, 2), (1, 3), (1, 4), (1, 5),\\
(1, 6), (1, 7), (1, 8), (1, 9), (1, 10), (2, 3), (2, 4), (2, 5), (2, 6), (2, 7), (2, 8), (2, 9), (2, 10), (3, 4), (3, 5), (3, 6),\\
(3, 7), (3, 8), (3, 9), (3, 10), (4, 5), (4, 6), (4, 7), (4, 8), (4, 9), (4, 10), (5, 6), (5, 7), (5, 11), (5, 12), (5, 13),\\
(6, 8), (6, 11), (6, 12), (6, 14), (7, 11), (7, 15), (7, 16), (7, 17), (8, 13), (8, 18)^{\rC}, (8, 19), (8, 20), (9, 10),\\
(9, 11), (9, 12), (9, 13), (9, 14), (10, 11), (10, 12), (10, 13), (10, 15), (11, 15), (11, 16), (11, 18), (11, 19),\\
(11, 21), (12, 13), (12, 16), (12, 17), (12, 18), (12, 19), (12, 22), (13, 14), (13, 18), (13, 20), (13, 21), (13, 22),\\
(14, 15), (14, 16), (14, 17), (14, 19), (14, 20), (14, 21), (14, 22), (15, 17), (15, 18), (15, 19), (15, 20), (15, 21),\\
(15, 22), (16, 17), (16, 18), (16, 19), (16, 20), (16, 21), (16, 22), (17, 18), (17, 19), (17, 20), (17, 21), (17, 22),\\
(18, 20), (18, 21), (18, 22), (19, 20), (19, 21), (19, 22), (20, 21), (20, 22), (21, 22) \}
$

\noindent
$
{\bf x} = \big[\begin{matrix} 3 & -1 & -1 & -1 & -1 & 3 & 12 & -7 & -2 & -12 & 6 & -5 & 6 & -5 & -1 & 17 & -18 & 4 & 1 & 3 \end{matrix} \\
\begin{matrix} -2 & -5 & 6 \end{matrix}\big]
$

\vspace{\baselineskip}
\noindent
$\boldsymbol{n = 24, \rho = 10, |\Sigma| = 2}$

\noindent
$
E(G) = \{ (0, 1), (0, 2), (0, 3), (0, 4), (0, 5), (0, 6), (0, 7), (0, 8), (0, 9), (0, 10), (1, 2), (1, 3), (1, 4), (1, 5),\\
(1, 6), (1, 7), (1, 8), (1, 9), (1, 10), (2, 3), (2, 4), (2, 5), (2, 6), (2, 7), (2, 8), (2, 9), (2, 10), (3, 4), (3, 5), (3, 6),\\
(3, 7), (3, 8), (3, 9), (3, 10), (4, 5), (4, 6), (4, 7), (4, 8), (4, 9), (4, 10), (5, 6), (5, 7), (5, 11), (5, 12), (5, 13),\\
(6, 8), (6, 11), (6, 12), (6, 14), (7, 11), (7, 15), (7, 16), (7, 17), (8, 12), (8, 18), (8, 19), (8, 20), (9, 10), (9, 11),\\
(9, 12), (9, 13), (9, 14), (10, 11), (10, 12), (10, 13), (10, 14), (11, 12), (11, 13), (11, 14), (11, 15), (11, 18),\\
(12, 15), (12, 16), (12, 19), (12, 21), (13, 17), (13, 18), (13, 20), (13, 21), (13, 22)^{\rC}, (13, 23), (14, 15),\\
(14, 17), (14, 19), (14, 20), (14, 22)^{\rC}, (14, 23), (15, 16), (15, 19), (15, 20), (15, 21), (15, 22), (15, 23),\\
(16, 17), (16, 18), (16, 19), (16, 20), (16, 21), (16, 22), (16, 23), (17, 18), (17, 19), (17, 20), (17, 21), (17, 22),\\
(17, 23), (18, 19), (18, 20), (18, 21), (18, 22), (18, 23), (19, 21), (19, 22), (19, 23), (20, 21), (20, 22), (20, 23),\\
(21, 22), (21, 23), (22, 23) \}
$

\noindent
$
{\bf x} = \big[\begin{matrix} 1 & 1 & 1 & 1 & 1 & 9 & 3 & -5 & -5 & -3 & -3 & -17 & 7 & 7 & 1 & 1 & 1 & 1 & -17 & 1 & 1 & 13 & -9 \end{matrix} \\
\begin{matrix} 7 \end{matrix}\big]
$

\vspace{\baselineskip}
\noindent
$\boldsymbol{n = 25, \rho = 10, |\Sigma| = 2}$

\noindent
$
E(G) = \{ (0, 1), (0, 2), (0, 3), (0, 4), (0, 5), (0, 6), (0, 7), (0, 8), (0, 9), (0, 10), (1, 2), (1, 3), (1, 4), (1, 5),\\
(1, 6), (1, 7), (1, 8), (1, 9), (1, 10), (2, 3), (2, 4), (2, 5), (2, 6), (2, 7), (2, 8), (2, 9), (2, 10), (3, 4), (3, 5), (3, 6),\\
(3, 7), (3, 8), (3, 9), (3, 10), (4, 5), (4, 6), (4, 7), (4, 8), (4, 9), (4, 10), (5, 6), (5, 7)^{\rC}, (5, 11), (5, 12), (5, 13),\\
(6, 8), (6, 11), (6, 12), (6, 14), (7, 11)^{\rC}, (7, 15), (7, 16), (7, 17), (8, 12), (8, 15), (8, 16), (8, 18), (9, 10),\\
(9, 11), (9, 12), (9, 13), (9, 14), (10, 11), (10, 12), (10, 13), (10, 14), (11, 12), (11, 13), (11, 14), (11, 15),\\
(11, 16), (12, 13), (12, 15), (12, 17), (12, 19), (13, 17), (13, 19), (13, 20), (13, 21), (13, 22), (14, 19), (14, 20),\\
(14, 21), (14, 22), (14, 23), (14, 24), (15, 17), (15, 18), (15, 20), (15, 21), (15, 23), (15, 24), (16, 17), (16, 18),\\
(16, 19), (16, 20), (16, 22), (16, 23), (16, 24), (17, 18), (17, 21), (17, 22), (17, 23), (17, 24), (18, 19), (18, 20),\\
(18, 21), (18, 22), (18, 23), (18, 24), (19, 20), (19, 21), (19, 22), (19, 23), (19, 24), (20, 21), (20, 22), (20, 23),\\
(20, 24), (21, 22), (21, 23), (21, 24), (22, 23), (22, 24), (23, 24) \}
$

\noindent
$
{\bf x} = \big[\begin{matrix} 5 & 5 & 5 & 5 & 5 & 23 & -35 & -23 & -9 & 12 & 12 & 23 & -38 & 2 & -24 & 13 & 12 & -4 & 23 & -37 \end{matrix} \\
\begin{matrix} 14 & -2 & -3 & 8 & 8 \end{matrix}\big]
$

\vspace{\baselineskip}
\noindent
$\boldsymbol{n = 26, \rho = 10, |\Sigma| = 3}$

\noindent
$
E(G) = \{ (0, 1)^{\rC}, (0, 2), (0, 3), (0, 4), (0, 5), (0, 6), (0, 7), (0, 8), (0, 9), (0, 10), (1, 2), (1, 3), (1, 4), (1, 5),\\
(1, 6), (1, 7), (1, 8), (1, 9), (1, 10), (2, 3)^{\rC}, (2, 4), (2, 5), (2, 6), (2, 7), (2, 8), (2, 9), (2, 10), (3, 4), (3, 5),\\
(3, 6), (3, 7), (3, 8), (3, 9), (3, 10), (4, 5), (4, 6), (4, 7), (4, 8), (4, 9), (4, 10), (5, 6), (5, 7), (5, 11), (5, 12),\\
(5, 13), (6, 8), (6, 11), (6, 12), (6, 14), (7, 11), (7, 15), (7, 16), (7, 17), (8, 12), (8, 15), (8, 16), (8, 18), (9, 10),\\
(9, 11), (9, 12), (9, 13), (9, 14), (10, 11), (10, 12), (10, 13), (10, 14), (11, 12), (11, 13), (11, 14), (11, 15),\\
(11, 16), (12, 13), (12, 14), (12, 15), (12, 18), (13, 15), (13, 16), (13, 19), (13, 20), (13, 21), (14, 17), (14, 19),\\
(14, 20), (14, 21), (14, 22), (15, 16), (15, 22), (15, 23), (15, 24), (15, 25), (16, 17), (16, 19), (16, 23), (16, 24),\\
(16, 25), (17, 18), (17, 20), (17, 21), (17, 22), (17, 23), (17, 24), (17, 25), (18, 19), (18, 20), (18, 21), (18, 22),\\
(18, 23), (18, 24), (18, 25), (19, 20)^{\rC}, (19, 21), (19, 22), (19, 23), (19, 24), (19, 25), (20, 21), (20, 22),\\
(20, 23), (20, 24), (20, 25), (21, 22), (21, 23), (21, 24), (21, 25), (22, 23), (22, 24), (22, 25), (23, 24), (23, 25),\\
(24, 25) \}
$

\noindent
$
{\bf x} = \big[\begin{matrix} 1 & 1 & 1 & 1 & 3 & 5 & 1 & -1 & -11 & 1 & 1 & 1 & -1 & -7 & -1 & 19 & -17 & -15 & -9 & -1 & -3 \end{matrix} \\
\begin{matrix} -5 & 21 & 5 & 5 & 5 \end{matrix}\big]
$

\pagebreak
\noindent
$\boldsymbol{n = 27, \rho = 10, |\Sigma| = 2}$

\noindent
$
E(G) = \{ (0, 1), (0, 2), (0, 3), (0, 4), (0, 5), (0, 6), (0, 7), (0, 8), (0, 9), (0, 10), (1, 2), (1, 3), (1, 4), (1, 5),\\
(1, 6), (1, 7), (1, 8), (1, 9), (1, 10), (2, 3), (2, 4), (2, 5), (2, 6), (2, 7), (2, 8), (2, 9), (2, 10), (3, 4), (3, 5), (3, 6),\\
(3, 7), (3, 8), (3, 9), (3, 10), (4, 5), (4, 6), (4, 7), (4, 8), (4, 9), (4, 10), (5, 6), (5, 7), (5, 11), (5, 12), (5, 13),\\
(6, 8), (6, 11), (6, 12)^{\rC}, (6, 14), (7, 11), (7, 15), (7, 16)^{\rC}, (7, 17), (8, 12), (8, 15), (8, 16), (8, 18), (9, 10),\\
(9, 11), (9, 12), (9, 13), (9, 14), (10, 11), (10, 12), (10, 13), (10, 14), (11, 12), (11, 13), (11, 14), (11, 15),\\
(11, 16), (12, 13), (12, 14), (12, 15), (12, 16), (13, 14), (13, 15), (13, 16), (13, 19), (13, 20), (14, 15), (14, 21),\\
(14, 22), (14, 23), (15, 19), (15, 20), (15, 21), (15, 24), (16, 17), (16, 18), (16, 24), (16, 25), (16, 26), (17, 18),\\
(17, 19), (17, 20), (17, 21), (17, 22), (17, 23), (17, 25), (17, 26), (18, 19), (18, 21), (18, 22), (18, 23), (18, 24),\\
(18, 25), (18, 26), (19, 20), (19, 22), (19, 23), (19, 24), (19, 25), (19, 26), (20, 21), (20, 22), (20, 23), (20, 24),\\
(20, 25), (20, 26), (21, 22), (21, 23), (21, 24), (21, 25), (21, 26), (22, 23), (22, 24), (22, 25), (22, 26), (23, 24),\\
(23, 25), (23, 26), (24, 25), (24, 26), (25, 26) \}
$

\noindent
$
{\bf x} = \big[\begin{matrix} 5 & 5 & 5 & 5 & 5 & 39 & -19 & -2 & -44 & 3 & 3 & 12 & 8 & -24 & -24 & -2 & 18 & -56 & -30 & -69 \end{matrix} \\
\begin{matrix} 12 & 51 & -16 & -16 & 80 & 26 & 26 \end{matrix}\big]
$

\vspace{\baselineskip}
\noindent
$\boldsymbol{n = 28, \rho = 10, |\Sigma| = 2}$

\noindent
$
E(G) = \{ (0, 1), (0, 2), (0, 3), (0, 4), (0, 5), (0, 6), (0, 7), (0, 8), (0, 9)^{\rC}, (0, 10), (1, 2), (1, 3), (1, 4), (1, 5),\\
(1, 6), (1, 7), (1, 8), (1, 9), (1, 10), (2, 3), (2, 4), (2, 5), (2, 6), (2, 7), (2, 8), (2, 9), (2, 10), (3, 4), (3, 5), (3, 6),\\
(3, 7), (3, 8), (3, 9), (3, 10), (4, 5), (4, 6), (4, 7), (4, 8), (4, 9), (4, 10), (5, 6), (5, 7), (5, 8), (5, 9), (5, 10),\\
(6, 7), (6, 8), (6, 11), (6, 12), (7, 11), (7, 12), (7, 13), (8, 11), (8, 13), (8, 14), (9, 11), (9, 14), (9, 15), (9, 16),\\
(10, 12), (10, 14), (10, 17), (10, 18), (11, 12), (11, 13), (11, 14), (11, 15), (11, 16), (11, 17), (12, 13), (12, 14),\\
(12, 15), (12, 16), (12, 17), (12, 18), (13, 14), (13, 15), (13, 16), (13, 17), (13, 18), (13, 19), (14, 15), (14, 16),\\
(14, 17), (14, 19), (15, 17), (15, 18), (15, 20), (15, 21), (15, 22), (16, 20), (16, 21), (16, 23), (16, 24), (16, 25),\\
(17, 22), (17, 23), (17, 26), (17, 27), (18, 20), (18, 22), (18, 24), (18, 25), (18, 26), (18, 27), (19, 20), (19, 21),\\
(19, 22), (19, 23), (19, 24), (19, 25), (19, 26), (19, 27), (20, 21), (20, 23)^{\rC}, (20, 24), (20, 25), (20, 26),\\
(20, 27), (21, 22), (21, 23), (21, 24), (21, 25), (21, 26), (21, 27), (22, 23), (22, 24), (22, 25), (22, 26), (22, 27),\\
(23, 24), (23, 25), (23, 26), (23, 27), (24, 25), (24, 26), (24, 27), (25, 26), (25, 27), (26, 27) \}
$

\noindent
$
{\bf x} = \big[\begin{matrix} 1 & -1 & -1 & -1 & -1 & -1 & 5 & -11 & 11 & -1 & -1 & 5 & -1 & -5 & -1 & -1 & 3 & 1 & 5 & -11  \end{matrix} \\
\begin{matrix} -1 & -3 & 1 & 1 & 3 & 3 & 1 & 1 \end{matrix}\big]
$

\vspace{\baselineskip}
\noindent
$\boldsymbol{n = 29, \rho = 10, |\Sigma| = 1}$

\noindent
$
E(G) = \{ (0, 1), (0, 2), (0, 3), (0, 4), (0, 5), (0, 6), (0, 7), (0, 8), (0, 9), (0, 10), (1, 2), (1, 3), (1, 4), (1, 5),\\
(1, 6), (1, 7), (1, 8), (1, 9), (1, 10), (2, 3), (2, 4), (2, 5), (2, 6), (2, 7), (2, 8), (2, 9), (2, 10), (3, 4), (3, 5), (3, 6),\\
(3, 7), (3, 8), (3, 9), (3, 10), (4, 5), (4, 6), (4, 7), (4, 8), (4, 9), (4, 10), (5, 6), (5, 7), (5, 8), (5, 9), (5, 10),\\
(6, 7), (6, 8), (6, 11), (6, 12), (7, 11), (7, 12), (7, 13), (8, 11), (8, 13), (8, 14), (9, 11), (9, 14), (9, 15), (9, 16),\\
(10, 12), (10, 14), (10, 17), (10, 18), (11, 12), (11, 13), (11, 14), (11, 15), (11, 16), (11, 17), (12, 13), (12, 14),\\
(12, 15), (12, 16), (12, 17), (12, 18), (13, 14), (13, 15), (13, 16), (13, 17), (13, 18), (13, 19), (14, 15), (14, 16),\\
(14, 17), (14, 18), (15, 16), (15, 17), (15, 19), (15, 20), (15, 21)^{\rC}, (16, 22), (16, 23), (16, 24), (16, 25),\\
(17, 20), (17, 22), (17, 26), (17, 27), (18, 20), (18, 23), (18, 24), (18, 26), (18, 27), (18, 28), (19, 20), (19, 21),\\
(19, 22), (19, 23), (19, 25), (19, 26), (19, 27), (19, 28), (20, 21), (20, 22), (20, 23), (20, 24), (20, 25), (20, 28),\\
(21, 22), (21, 23), (21, 24), (21, 25), (21, 26), (21, 27), (21, 28), (22, 24), (22, 25), (22, 26), (22, 27), (22, 28),\\
(23, 24), (23, 25), (23, 26), (23, 27), (23, 28), (24, 25), (24, 26), (24, 27), (24, 28), (25, 26), (25, 27), (25, 28),\\
(26, 27), (26, 28), (27, 28) \}
$

\noindent
$
{\bf x} = \big[\begin{matrix} 1 & 1 & 1 & 1 & 1 & 1 & -13 & 1 & 3 & -15 & 19 & -16 & 6 & 17 & 6 & -11 & 15 & -9 & -9 & 14 & -7  \end{matrix} \\
\begin{matrix} 11 & 8 & -2 & -8 & 15 & -11 & -11 & -9  \end{matrix}\big]
$

\vspace{\baselineskip}
\noindent
$\boldsymbol{n = 30, \rho = 10, |\Sigma| = 1}$

\noindent
$
E(G) = \{ (0, 1), (0, 2), (0, 3), (0, 4), (0, 5), (0, 6), (0, 7), (0, 8), (0, 9), (0, 10)^{\rC}, (1, 2), (1, 3), (1, 4), (1, 5),\\
(1, 6), (1, 7), (1, 8), (1, 9), (1, 10), (2, 3), (2, 4), (2, 5), (2, 6), (2, 7), (2, 8), (2, 9), (2, 10), (3, 4), (3, 5), (3, 6),\\
(3, 7), (3, 8), (3, 9), (3, 10), (4, 5), (4, 6), (4, 7), (4, 8), (4, 9), (4, 10), (5, 6), (5, 7), (5, 8), (5, 9), (5, 10),\\
(6, 7), (6, 8), (6, 11), (6, 12), (7, 11), (7, 12), (7, 13), (8, 11), (8, 13), (8, 14), (9, 11), (9, 14), (9, 15), (9, 16),\\
(10, 12), (10, 14), (10, 17), (10, 18), (11, 12), (11, 13), (11, 14), (11, 15), (11, 16), (11, 17), (12, 13), (12, 14),\\
(12, 15), (12, 16), (12, 17), (12, 18), (13, 14), (13, 15), (13, 16), (13, 17), (13, 18), (13, 19), (14, 15), (14, 16),\\
(14, 17), (14, 18), (15, 16), (15, 17), (15, 18), (15, 19), (15, 20), (16, 17), (16, 18), (16, 19), (16, 21), (17, 19),\\
(17, 20), (17, 22), (18, 21), (18, 22), (18, 23), (18, 24), (19, 23), (19, 25), (19, 26), (19, 27), (19, 28), (19, 29),\\
(20, 21), (20, 23), (20, 24), (20, 25), (20, 26), (20, 27), (20, 28), (20, 29), (21, 22), (21, 24), (21, 25), (21, 26),\\
(21, 27), (21, 28), (21, 29), (22, 23), (22, 24), (22, 25), (22, 26), (22, 27), (22, 28), (22, 29), (23, 24), (23, 25),\\
(23, 26), (23, 27), (23, 28), (23, 29), (24, 25), (24, 26), (24, 27), (24, 28), (24, 29), (25, 26), (25, 27), (25, 28),\\
(25, 29), (26, 27), (26, 28), (26, 29), (27, 28), (27, 29), (28, 29) \}
$

\noindent
$
{\bf x} = \big[\begin{matrix}  3 & -3 & -3 & -3 & -3 & -3 & 5 & 1 & -19 & 25 & -3 & 25 & 5 & -23 & 5 & -11 & -7 & 19 & -11  \end{matrix} \\
\begin{matrix}  -7 & -31 & -27 & 47 & 17 & -3 & 1 & 1 & 1 & 1 & 1  \end{matrix}\big]
$

\vspace{\baselineskip}
\noindent
$\boldsymbol{n = 31, \rho = 10, |\Sigma| = 1}$

\noindent
$
E(G) = \{ (0, 1), (0, 2), (0, 3), (0, 4), (0, 5), (0, 6), (0, 7)^{\rC}, (0, 8), (0, 9), (0, 10), (1, 2), (1, 3), (1, 4), (1, 5),\\
(1, 6), (1, 7), (1, 8), (1, 9), (1, 10), (2, 3), (2, 4), (2, 5), (2, 6), (2, 7), (2, 8), (2, 9), (2, 10), (3, 4), (3, 5), (3, 6),\\
(3, 7), (3, 8), (3, 9), (3, 10), (4, 5), (4, 6), (4, 7), (4, 8), (4, 9), (4, 10), (5, 6), (5, 7), (5, 8), (5, 9), (5, 10),\\
(6, 7), (6, 8), (6, 11), (6, 12), (7, 11), (7, 12), (7, 13), (8, 11), (8, 13), (8, 14), (9, 11), (9, 14), (9, 15), (9, 16),\\
(10, 12), (10, 14), (10, 17), (10, 18), (11, 12), (11, 13), (11, 14), (11, 15), (11, 16), (11, 17), (12, 13), (12, 14),\\
(12, 15), (12, 16), (12, 17), (12, 18), (13, 14), (13, 15), (13, 16), (13, 17), (13, 18), (13, 19), (14, 15), (14, 16),\\
(14, 17), (14, 18), (15, 16), (15, 17), (15, 18), (15, 19), (15, 20), (16, 17), (16, 18), (16, 19), (16, 20), (17, 19),\\
(17, 21), (17, 22), (18, 21), (18, 23), (18, 24), (18, 25), (19, 20), (19, 23), (19, 26), (19, 27), (19, 28), (19, 29),\\
(20, 22), (20, 23), (20, 24), (20, 25), (20, 26), (20, 27), (20, 30), (21, 22), (21, 23), (21, 24), (21, 25), (21, 26),\\
(21, 28), (21, 29), (21, 30), (22, 23), (22, 24), (22, 26), (22, 27), (22, 28), (22, 29), (22, 30), (23, 25), (23, 27),\\
(23, 28), (23, 29), (23, 30), (24, 25), (24, 26), (24, 27), (24, 28), (24, 29), (24, 30), (25, 26), (25, 27), (25, 28),\\
(25, 29), (25, 30), (26, 27), (26, 28), (26, 29), (26, 30), (27, 28), (27, 29), (27, 30), (28, 29), (28, 30), (29, 30) \}
$

\noindent
$
{\bf x} = \big[\begin{matrix}  1 & -1 & -1 & -1 & -1 & -1 & 2 & -1 & -3 & 2 & 3 & 3 & 5 & -4 & 3 & -1 & -1 & -2 & -2 & -1   \end{matrix} \\
\begin{matrix} -3 & -5 & -2 & 1 & -2 & 1 & -1 & 5 & 3 & 3 & 1  \end{matrix}\big]
$

\vspace{\baselineskip}
\noindent
$\boldsymbol{n = 16, \rho = 11, |\Sigma| = 2}$

\noindent
$
E(G) = \{ (0, 1), (0, 2), (0, 3), (0, 4), (0, 5), (0, 6), (0, 7), (0, 8), (0, 9), (0, 10), (0, 11), (1, 2), (1, 3), (1, 4),\\
(1, 5), (1, 6), (1, 7), (1, 8), (1, 9), (1, 10), (1, 11), (2, 3)^{\rC}, (2, 4), (2, 5), (2, 6), (2, 7), (2, 8), (2, 9), (2, 10),\\
(2, 12), (3, 4), (3, 5), (3, 6), (3, 7), (3, 11), (3, 12), (3, 13), (3, 14), (4, 5), (4, 6), (4, 8), (4, 11), (4, 13), (4, 14),\\
(4, 15), (5, 9), (5, 10), (5, 11), (5, 12), (5, 13), (5, 15), (6, 9), (6, 11), (6, 12), (6, 13), (6, 14), (6, 15)^{\rC}, (7, 8),\\
(7, 9), (7, 10), (7, 11), (7, 12), (7, 14), (7, 15), (8, 9), (8, 10), (8, 12), (8, 13), (8, 14), (8, 15), (9, 12), (9, 13),\\
(9, 14), (9, 15), (10, 11), (10, 12), (10, 13), (10, 14), (10, 15), (11, 13), (11, 14), (11, 15), (12, 13), (12, 14),\\
(12, 15), (13, 14), (13, 15), (14, 15) \}
$

\noindent
$
{\bf x} = \begin{bmatrix} 1 & 1 & -2 & 2 & 1 & 1 & -2 & 1 & 1 & 4 & -5 & -2 & -1 & -1 & -1 & 2 \end{bmatrix}
$

\vspace{\baselineskip}
\noindent
$\boldsymbol{n = 18, \rho = 11, |\Sigma| = 2}$

\noindent
$
E(G) = \{ (0, 1), (0, 2), (0, 3), (0, 4), (0, 5), (0, 6), (0, 7), (0, 8)^{\rC}, (0, 9), (0, 10), (0, 11), (1, 2), (1, 3), (1, 4),\\
(1, 5), (1, 6), (1, 7), (1, 8), (1, 9), (1, 10), (1, 11), (2, 3), (2, 4), (2, 5), (2, 6), (2, 7), (2, 8), (2, 9), (2, 10),\\
(2, 11), (3, 4), (3, 5), (3, 6), (3, 7), (3, 8), (3, 9), (3, 10), (3, 11), (4, 5), (4, 6), (4, 7), (4, 8), (4, 12), (4, 13),\\
(4, 14), (5, 6), (5, 7), (5, 9), (5, 12), (5, 13), (5, 15), (6, 10), (6, 12), (6, 14)^{\rC}, (6, 16), (6, 17), (7, 11), (7, 14),\\
(7, 15), (7, 16), (7, 17), (8, 12), (8, 13), (8, 14), (8, 15), (8, 16), (8, 17), (9, 12), (9, 13), (9, 14), (9, 15),\\
(9, 16), (9, 17), (10, 11), (10, 12), (10, 13), (10, 15), (10, 16), (10, 17), (11, 13), (11, 14), (11, 15), (11, 16),\\
(11, 17), (12, 13), (12, 14), (12, 15), (12, 16), (12, 17), (13, 14), (13, 15), (13, 16), (13, 17), (14, 15), (14, 16),\\
(14, 17), (15, 16), (15, 17), (16, 17) \}
$

\noindent
$
{\bf x} = \begin{bmatrix} 3 & -1 & -1 & -1 & 5 & -1 & -3 & -5 & -2 & -4 & 5 & 4 & 1 & 8 & 2 & -2 & -4 & -4 \end{bmatrix}
$

\pagebreak
\noindent
$\boldsymbol{n = 20, \rho = 11, |\Sigma| = 3}$

\noindent
$
E(G) = \{ (0, 1), (0, 2), (0, 3), (0, 4), (0, 5), (0, 6), (0, 7)^{\rC}, (0, 8), (0, 9), (0, 10), (0, 11), (1, 2), (1, 3), (1, 4),\\
(1, 5), (1, 6), (1, 7), (1, 8), (1, 9), (1, 10), (1, 11), (2, 3), (2, 4), (2, 5), (2, 6), (2, 7), (2, 8), (2, 9), (2, 10),\\
(2, 11), (3, 4), (3, 5), (3, 6), (3, 7), (3, 8), (3, 9), (3, 10), (3, 11), (4, 5), (4, 6), (4, 7), (4, 8), (4, 9), (4, 10),\\
(4, 11), (5, 6), (5, 7)^{\rC}, (5, 8), (5, 12), (5, 13), (5, 14), (6, 12), (6, 13), (6, 15), (6, 16), (6, 17), (7, 12), (7, 14),\\
(7, 15), (7, 16), (7, 18), (8, 12), (8, 15), (8, 17), (8, 18), (8, 19), (9, 12), (9, 13), (9, 14), (9, 17), (9, 18),\\
(9, 19), (10, 13), (10, 14), (10, 15), (10, 16), (10, 17), (10, 19), (11, 13), (11, 14), (11, 15), (11, 16), (11, 18),\\
(11, 19), (12, 13)^{\rC}, (12, 15), (12, 16), (12, 17), (12, 18), (12, 19), (13, 14), (13, 16), (13, 17), (13, 18),\\
(13, 19), (14, 15), (14, 16), (14, 17), (14, 18), (14, 19), (15, 16), (15, 17), (15, 18), (15, 19), (16, 17), (16, 18),\\
(16, 19), (17, 18), (17, 19), (18, 19) \}
$

\noindent
$
{\bf x} = \begin{bmatrix} 15 & -1 & -1 & -1 & -1 & -23 & 13 & -8 & -14 & 13 & -7 & 14 & 6 & -4 & -20 & -1 & 9 & 2 & 2 & 3 \end{bmatrix}
$

\vspace{\baselineskip}
\noindent
$\boldsymbol{n = 22, \rho = 11, |\Sigma| = 2}$

\noindent
$
E(G) = \{ (0, 1)^{\rC}, (0, 2), (0, 3), (0, 4), (0, 5), (0, 6), (0, 7), (0, 8), (0, 9), (0, 10), (0, 11), (1, 2), (1, 3), (1, 4),\\
(1, 5), (1, 6), (1, 7), (1, 8), (1, 9), (1, 10), (1, 11), (2, 3), (2, 4), (2, 5), (2, 6), (2, 7), (2, 8), (2, 9), (2, 10),\\
(2, 11), (3, 4), (3, 5), (3, 6), (3, 7), (3, 8), (3, 9), (3, 10), (3, 11), (4, 5), (4, 6), (4, 7), (4, 8), (4, 9), (4, 10),\\
(4, 12), (5, 6), (5, 7), (5, 8), (5, 9), (5, 12), (5, 13), (6, 11), (6, 12), (6, 14), (6, 15), (6, 16), (7, 11), (7, 13),\\
(7, 14), (7, 15), (7, 17), (8, 12), (8, 13), (8, 14), (8, 17), (8, 18), (9, 12), (9, 13), (9, 17), (9, 19), (9, 20),\\
(10, 11), (10, 14), (10, 17), (10, 18), (10, 19), (10, 21), (11, 15), (11, 16), (11, 20), (11, 21), (12, 15)^{\rC},\\
(12, 16), (12, 18), (12, 19), (12, 20), (12, 21), (13, 14), (13, 15), (13, 16), (13, 18), (13, 19), (13, 20), (13, 21),\\
(14, 16), (14, 17), (14, 18), (14, 19), (14, 20), (14, 21), (15, 16), (15, 17), (15, 18), (15, 19), (15, 20), (15, 21),\\
(16, 17), (16, 18), (16, 19), (16, 20), (16, 21), (17, 18), (17, 19), (17, 20), (17, 21), (18, 19), (18, 20), (18, 21),\\
(19, 20), (19, 21), (20, 21) \}
$

\noindent
$
{\bf x} = \big[\begin{matrix}  3 & 3 & 9 & 9 & -20 & 9 & -37 & -14 & 20 & -14 & 9 & 32 & 3 & 38 & -25 & -9 & -14 & -49 & 20   \end{matrix} \\
\begin{matrix}  -14 & 9 & 32  \end{matrix}\big]
$

\vspace{\baselineskip}
\noindent
$\boldsymbol{n = 24, \rho = 11, |\Sigma| = 2}$

\noindent
$
E(G) = \{ (0, 1), (0, 2), (0, 3), (0, 4), (0, 5), (0, 6), (0, 7), (0, 8), (0, 9), (0, 10), (0, 11)^{\rC}, (1, 2), (1, 3), (1, 4),\\
(1, 5), (1, 6), (1, 7), (1, 8), (1, 9), (1, 10), (1, 11), (2, 3), (2, 4), (2, 5), (2, 6), (2, 7), (2, 8), (2, 9), (2, 10),\\
(2, 11), (3, 4), (3, 5), (3, 6), (3, 7), (3, 8), (3, 9), (3, 10), (3, 11), (4, 5), (4, 6), (4, 7), (4, 8), (4, 9), (4, 10),\\
(4, 11), (5, 6), (5, 7), (5, 8), (5, 9), (5, 10), (5, 12), (6, 7), (6, 8), (6, 9), (6, 13), (6, 14), (7, 10), (7, 15), (7, 16),\\
(7, 17), (8, 10), (8, 15), (8, 16), (8, 18), (9, 10), (9, 15), (9, 19), (9, 20), (10, 13), (10, 21), (11, 13), (11, 14),\\
(11, 17), (11, 18), (11, 22)^{\rC}, (11, 23), (12, 13), (12, 14), (12, 15), (12, 16), (12, 17), (12, 19), (12, 20),\\
(12, 21), (12, 22), (12, 23), (13, 17), (13, 18), (13, 19), (13, 20), (13, 21), (13, 22), (13, 23), (14, 15), (14, 16),\\
(14, 17), (14, 18), (14, 19), (14, 20), (14, 21), (14, 22), (15, 16), (15, 17), (15, 19), (15, 20), (15, 22), (15, 23),\\
(16, 17), (16, 18), (16, 19), (16, 21), (16, 22), (16, 23), (17, 18), (17, 20), (17, 21), (17, 23), (18, 19), (18, 20),\\
(18, 21), (18, 22), (18, 23), (19, 20), (19, 21), (19, 22), (19, 23), (20, 21), (20, 22), (20, 23), (21, 22), (21, 23),\\
(22, 23) \}
$

\noindent
$
{\bf x} = \big[\begin{matrix}  3 & -1 & -1 & -1 & -1 & -2 & 5 & 2 & 2 & -1 & -4 & -2 & -3 & 5 & -5 & 2 & -2 & 2 & 2 & -2 & 2   \end{matrix} \\
\begin{matrix}  -5 & 1 & 4 \end{matrix}\big]
$

\vspace{\baselineskip}
\noindent
$\boldsymbol{n = 26, \rho = 11, |\Sigma| = 2}$

\noindent
$
E(G) = \{ (0, 1), (0, 2), (0, 3), (0, 4), (0, 5), (0, 6), (0, 7), (0, 8), (0, 9), (0, 10), (0, 11), (1, 2), (1, 3), (1, 4),\\
(1, 5), (1, 6), (1, 7), (1, 8), (1, 9), (1, 10), (1, 11), (2, 3), (2, 4), (2, 5), (2, 6), (2, 7), (2, 8), (2, 9), (2, 10),\\
(2, 11), (3, 4), (3, 5), (3, 6), (3, 7), (3, 8), (3, 9), (3, 10), (3, 11), (4, 5), (4, 6), (4, 7), (4, 8), (4, 9), (4, 10),\\
(4, 11), (5, 6), (5, 7), (5, 8), (5, 9)^{\rC}, (5, 12), (5, 13), (6, 10), (6, 11), (6, 14), (6, 15), (6, 16), (7, 10), (7, 14),\\
(7, 15), (7, 16), (7, 17), (8, 10), (8, 14), (8, 15), (8, 18), (8, 19), (9, 18), (9, 19), (9, 20), (9, 21)^{\rC}, (9, 22),\\
(10, 12), (10, 13), (10, 14), (11, 12), (11, 13), (11, 14), (11, 15), (11, 16), (12, 13), (12, 14), (12, 15), (12, 16),\\
(12, 17), (12, 18), (12, 19), (12, 20), (13, 14), (13, 15), (13, 18), (13, 21), (13, 22), (13, 23), (13, 24), (14, 19),\\
(14, 20), (14, 23), (14, 25), (15, 17), (15, 21), (15, 23), (15, 24), (15, 25), (16, 17), (16, 19), (16, 20), (16, 22),\\
(16, 23), (16, 24), (16, 25), (17, 18), (17, 20), (17, 21), (17, 22), (17, 23), (17, 24), (17, 25), (18, 19), (18, 20),\\
(18, 21), (18, 22), (18, 24), (18, 25), (19, 21), (19, 22), (19, 23), (19, 24), (19, 25), (20, 21), (20, 22), (20, 23),\\
(20, 24), (20, 25), (21, 22), (21, 23), (21, 24), (21, 25), (22, 23), (22, 24), (22, 25), (23, 24), (23, 25), (24, 25) \}
$

\noindent
$
{\bf x} = \big[\begin{matrix}  1 & 1 & 1 & 1 & 1 & 3 & 13 & -14 & 1 & -6 & -2 & 1 & -2 & -9 & 6 & -5 & -8 & 1 & 6 & -13 & 20 & 9   \end{matrix} \\
\begin{matrix}   -6 & -5 & -5 & 10 \end{matrix}\big]
$

\vspace{\baselineskip}
\noindent
$\boldsymbol{n = 28, \rho = 11, |\Sigma| = 1}$

\noindent
$
E(G) = \{ (0, 1), (0, 2), (0, 3), (0, 4), (0, 5), (0, 6), (0, 7), (0, 8), (0, 9)^{\rC}, (0, 10), (0, 11), (1, 2), (1, 3), (1, 4),\\
(1, 5), (1, 6), (1, 7), (1, 8), (1, 9), (1, 10), (1, 11), (2, 3), (2, 4), (2, 5), (2, 6), (2, 7), (2, 8), (2, 9), (2, 10),\\
(2, 11), (3, 4), (3, 5), (3, 6), (3, 7), (3, 8), (3, 9), (3, 10), (3, 11), (4, 5), (4, 6), (4, 7), (4, 8), (4, 9), (4, 10),\\
(4, 11), (5, 6), (5, 7), (5, 8), (5, 9), (5, 12), (5, 13), (6, 10), (6, 11), (6, 14), (6, 15), (6, 16), (7, 10), (7, 12),\\
(7, 17), (7, 18), (7, 19), (8, 10), (8, 14), (8, 15), (8, 20), (8, 21), (9, 11), (9, 14), (9, 15), (9, 16), (9, 20),\\
(10, 12), (10, 13), (10, 14), (11, 12), (11, 13), (11, 14), (11, 15), (12, 13), (12, 14), (12, 15), (12, 16), (12, 17),\\
(12, 18), (12, 19), (13, 14), (13, 15), (13, 16), (13, 17), (13, 18), (13, 19), (13, 20), (14, 16), (14, 17), (14, 22),\\
(14, 23), (15, 18), (15, 22), (15, 24), (15, 25), (15, 26), (16, 19), (16, 22), (16, 23), (16, 24), (16, 25), (16, 27),\\
(17, 21), (17, 22), (17, 23), (17, 24), (17, 25), (17, 26), (17, 27), (18, 20), (18, 21), (18, 22), (18, 24), (18, 25),\\
(18, 26), (18, 27), (19, 21), (19, 22), (19, 23), (19, 24), (19, 25), (19, 26), (19, 27), (20, 21), (20, 22), (20, 23),\\
(20, 24), (20, 25), (20, 26), (20, 27), (21, 22), (21, 23), (21, 24), (21, 25), (21, 26), (21, 27), (22, 23), (22, 26),\\
(22, 27), (23, 24), (23, 25), (23, 26), (23, 27), (24, 25), (24, 26), (24, 27), (25, 26), (25, 27), (26, 27) \}
$

\noindent
$
{\bf x} = \big[\begin{matrix}  1 & -1 & -1 & -1 & -1 & 8 & -1 & -2 & 6 & -1 & -7 & -1 & 2 & -1 & -1 & 5 & -1 & 1 & -6 & 5   \end{matrix} \\
\begin{matrix}  -5 & 3 & 2 & 1 & -1 & -1 & 2 & -4 \end{matrix}\big]
$

\vspace{\baselineskip}
\noindent
$\boldsymbol{n = 30, \rho = 11, |\Sigma| = 2}$

\noindent
$
E(G) = \{ (0, 1), (0, 2), (0, 3), (0, 4), (0, 5), (0, 6), (0, 7), (0, 8), (0, 9), (0, 10), (0, 11), (1, 2), (1, 3), (1, 4),\\
(1, 5), (1, 6), (1, 7), (1, 8), (1, 9), (1, 10), (1, 11), (2, 3), (2, 4), (2, 5), (2, 6), (2, 7), (2, 8), (2, 9), (2, 10),\\
(2, 11), (3, 4), (3, 5), (3, 6), (3, 7), (3, 8), (3, 9), (3, 10), (3, 11), (4, 5), (4, 6), (4, 7), (4, 8), (4, 9), (4, 10),\\
(4, 11), (5, 6), (5, 7), (5, 8), (5, 9), (5, 12), (5, 13), (6, 10), (6, 11), (6, 14), (6, 15), (6, 16), (7, 10), (7, 12),\\
(7, 17), (7, 18), (7, 19), (8, 10), (8, 14), (8, 15), (8, 16), (8, 20), (9, 13), (9, 21), (9, 22), (9, 23), (9, 24),\\
(10, 12), (10, 13), (10, 14)^{\rC}, (11, 12), (11, 13), (11, 14), (11, 15), (11, 16), (12, 13), (12, 14), (12, 15),\\
(12, 16), (12, 17), (12, 18), (12, 19), (13, 14), (13, 15), (13, 16), (13, 17), (13, 18), (13, 19), (14, 15), (14, 16),\\
(14, 17)^{\rC}, (14, 18), (14, 19), (15, 16), (15, 17), (15, 18), (15, 19), (15, 20), (16, 17), (16, 21), (16, 22),\\
(16, 25), (17, 18), (17, 20), (17, 23), (17, 25), (17, 26), (18, 21), (18, 25), (18, 27), (18, 28), (18, 29), (19, 23),\\
(19, 24), (19, 26), (19, 27), (19, 28), (19, 29), (20, 22), (20, 23), (20, 24), (20, 25), (20, 26), (20, 27), (20, 28),\\
(20, 29), (21, 22), (21, 23), (21, 24), (21, 25), (21, 26), (21, 27), (21, 28), (21, 29), (22, 23), (22, 24), (22, 25),\\
(22, 26), (22, 27), (22, 28), (22, 29), (23, 24), (23, 26), (23, 27), (23, 28), (23, 29), (24, 25), (24, 26), (24, 27),\\
(24, 28), (24, 29), (25, 26), (25, 27), (25, 28), (25, 29), (26, 27), (26, 28), (26, 29), (27, 28), (27, 29), (28, 29) \}
$

\noindent
$
{\bf x} = \big[\begin{matrix}  1 & 1 & 1 & 1 & 1 & -28 & 1 & -31 & -2 & 14 & 32 & 10 & -16 & 29 & -14 & 20 & -25 & -4 & 25 & -14   \end{matrix} \\
\begin{matrix}  10 & 2 & -13 & 7 & -2 & -13 & -20 & 9 & 9 & 9 \end{matrix}\big]
$

\vspace{\baselineskip}
\noindent
$\boldsymbol{n = 32, \rho = 11, |\Sigma| = 2}$

\noindent
$
E(G) = \{ (0, 1), (0, 2), (0, 3), (0, 4), (0, 5), (0, 6), (0, 7), (0, 8), (0, 9), (0, 10), (0, 11), (1, 2), (1, 3), (1, 4),\\
(1, 5), (1, 6), (1, 7), (1, 8), (1, 9), (1, 10), (1, 11), (2, 3), (2, 4), (2, 5), (2, 6), (2, 7), (2, 8), (2, 9), (2, 10),\\
(2, 11), (3, 4), (3, 5), (3, 6), (3, 7), (3, 8), (3, 9), (3, 10), (3, 11), (4, 5), (4, 6), (4, 7), (4, 8), (4, 9), (4, 10),\\
(4, 11), (5, 6), (5, 7), (5, 8), (5, 9), (5, 12), (5, 13), (6, 10), (6, 11), (6, 14), (6, 15), (6, 16), (7, 10), (7, 12),\\
(7, 17), (7, 18), (7, 19), (8, 10), (8, 14), (8, 15), (8, 16), (8, 20), (9, 12), (9, 17), (9, 18), (9, 21), (9, 22),\\
(10, 12), (10, 13), (10, 14), (11, 12), (11, 13), (11, 14), (11, 15), (11, 16), (12, 13), (12, 14), (12, 15), (12, 16),\\
(12, 17), (12, 18), (13, 14), (13, 15), (13, 16), (13, 17), (13, 18), (13, 19), (13, 20), (14, 15), (14, 16), (14, 17),\\
(14, 18), (14, 19), (15, 16), (15, 17), (15, 18), (15, 19)^{\rC}, (15, 20), (16, 17), (16, 18), (16, 19), (16, 20),\\
(17, 18), (17, 19), (17, 20), (17, 23), (18, 21), (18, 24), (18, 25), (19, 23), (19, 24), (19, 25), (19, 26)^{\rC},\\
(19, 27), (20, 22), (20, 23), (20, 28), (20, 29), (20, 30), (20, 31), (21, 23), (21, 24), (21, 25), (21, 26), (21, 27),\\
(21, 28), (21, 29), (21, 30), (21, 31), (22, 23), (22, 24), (22, 25), (22, 26), (22, 27), (22, 28), (22, 29), (22, 30),\\
(22, 31), (23, 26), (23, 27), (23, 28), (23, 29), (23, 30), (23, 31), (24, 25), (24, 26), (24, 27), (24, 28), (24, 29),\\
(24, 30), (24, 31), (25, 26), (25, 27), (25, 28), (25, 29), (25, 30), (25, 31), (26, 27), (26, 28), (26, 29), (26, 30),\\
(26, 31), (27, 28), (27, 29), (27, 30), (27, 31), (28, 29), (28, 30), (28, 31), (29, 30), (29, 31), (30, 31) \}
$

\noindent
$
{\bf x} = \big[\begin{matrix}  1 & 1 & 1 & 1 & 1 & -4 & 1 & 8 & -4 & -5 & -3 & 3 & -2 & -3 & -5 & 3 & 1 & 2 & 3 & -1 & 3 & -3    \end{matrix} \\
\begin{matrix}  -1 & -2 & 2 & 2 & -1 & -3 & 1 & 1 & 1 & 1 \end{matrix}\big]
$

\vspace{\baselineskip}
\noindent
$\boldsymbol{n = 34, \rho = 11, |\Sigma| = 2}$

\noindent
$
E(G) = \{ (0, 1), (0, 2), (0, 3), (0, 4), (0, 5), (0, 6), (0, 7), (0, 8), (0, 9), (0, 10), (0, 11)^{\rC}, (1, 2), (1, 3), (1, 4),\\
(1, 5), (1, 6), (1, 7), (1, 8), (1, 9), (1, 10), (1, 11), (2, 3), (2, 4), (2, 5), (2, 6), (2, 7), (2, 8), (2, 9), (2, 10),\\
(2, 11), (3, 4), (3, 5), (3, 6), (3, 7), (3, 8), (3, 9), (3, 10), (3, 11), (4, 5), (4, 6), (4, 7), (4, 8), (4, 9), (4, 10),\\
(4, 11), (5, 6), (5, 7), (5, 8), (5, 9), (5, 12), (5, 13), (6, 10), (6, 11), (6, 14), (6, 15), (6, 16), (7, 10), (7, 12),\\
(7, 17), (7, 18), (7, 19), (8, 10), (8, 14), (8, 15), (8, 16), (8, 20), (9, 12), (9, 17), (9, 18), (9, 20), (9, 21),\\
(10, 12), (10, 13), (10, 14), (11, 12), (11, 13), (11, 14), (11, 15), (11, 16), (12, 13), (12, 14), (12, 15), (12, 16),\\
(12, 17), (12, 18), (13, 14), (13, 15), (13, 16), (13, 17), (13, 18), (13, 19), (13, 20), (14, 15), (14, 16), (14, 17),\\
(14, 18), (14, 19), (15, 16), (15, 17), (15, 18), (15, 19), (15, 20), (16, 17), (16, 18), (16, 19), (16, 20), (17, 18),\\
(17, 19), (17, 20), (17, 21), (18, 19), (18, 20), (18, 22), (19, 20), (19, 21), (19, 23), (19, 24), (20, 25), (20, 26),\\
(20, 27)^{\rC}, (21, 22), (21, 23), (21, 24), (21, 25), (21, 28), (21, 29), (21, 30), (21, 31), (22, 23), (22, 24),\\
(22, 25), (22, 26), (22, 27), (22, 28), (22, 29), (22, 32), (22, 33), (23, 25), (23, 26), (23, 27), (23, 28), (23, 30),\\
(23, 31), (23, 32), (23, 33), (24, 25), (24, 26), (24, 28), (24, 29), (24, 30), (24, 31), (24, 32), (24, 33), (25, 27),\\
(25, 29), (25, 30), (25, 31), (25, 32), (25, 33), (26, 27), (26, 28), (26, 29), (26, 30), (26, 31), (26, 32), (26, 33),\\
(27, 28), (27, 29), (27, 30), (27, 31), (27, 32), (27, 33), (28, 29), (28, 30), (28, 31), (28, 32), (28, 33), (29, 30),\\
(29, 31), (29, 32), (29, 33), (30, 31), (30, 32), (30, 33), (31, 32), (31, 33), (32, 33) \}
$

\noindent
$
{\bf x} = \big[\begin{matrix}  12 & -4 & -4 & -4 & -4 & 23 & 20 & 4 & -10 & -6 & -23 & -8 & -3 & -1 & -6 & 9 & 9 & -2 & 1 & 8    \end{matrix} \\
\begin{matrix}  -8 & -7 & -4 & -7 & 8 & 4 & -8 & 4 & -7 & 4 & 1 & 1 & 4 & 4 \end{matrix}\big]
$

\end{appendices}

\end{document}